\documentclass[reqno]{amsart}
\usepackage[utf8]{inputenc}
\usepackage{amssymb}
\usepackage{mdframed}
\usepackage{bm,bbm}
\usepackage{xcolor}
\usepackage{stmaryrd}
\usepackage{tikz-cd}\usetikzlibrary{decorations.pathmorphing}
\usepackage{comment}
\usepackage[all,cmtip]{xy}
\usepackage{microtype}
\usepackage{graphicx}
\usepackage{amsthm,etoolbox}
\usepackage{eucal, mathbbol}

\usepackage{hyperref}
\hypersetup{hidelinks, colorlinks}
\hypersetup{
    citecolor = {teal},
    linkcolor = {black},
}
\usepackage[capitalise]{cleveref}

\numberwithin{equation}{section}

\calclayout

\usepackage[left= 3.4 cm, right= 3.4 cm, top= 2.8 cm, bottom=2.7 cm, foot= 1.2 cm, marginparwidth=2.7 cm, marginparsep=0.5 cm]{geometry}

\numberwithin{equation}{section}
\usepackage{soul}

\usepackage{hyperref}
\hypersetup{colorlinks}
\definecolor{darkred}{rgb}{0.5,0,0}
\definecolor{darkgreen}{rgb}{0,0.5,0}
\definecolor{darkblue}{rgb}{0,0,0.5}
\hypersetup{colorlinks, linkcolor=darkblue, filecolor=darkgreen, urlcolor=darkred, citecolor=darkblue}
\makeatletter 
\@addtoreset{equation}{section}
\makeatother  

\numberwithin{equation}{section}

\newtheorem{thm}{Theorem}[section]
\newtheorem{cor}[thm]{Corollary}

\newtheorem{conj}[thm]{Conjecture}
\newtheorem{assm}[thm]{Assumption}
\newtheorem{prop}[thm]{Proposition}
\newtheorem{ques}[thm]{Question}

\newtheorem{lemma}[thm]{Lemma}
\newtheorem{def-lemma}[thm]{Definition-Lemma}
\theoremstyle{definition}
\newtheorem{defn}[thm]{Definition}
\theoremstyle{remark}

\theoremstyle{remark}
\newtheorem{rem}[thm]{Remark}
\AtBeginEnvironment{rem}{%
    \small
  \pushQED{\qed}%
}
\AtEndEnvironment{rem}{\popQED}

\newtheorem{example}[thm]{Example}

\newcommand{\beq}{\begin{equation}}
\newcommand{\eeq}{\end{equation}}
\newcommand{\beqn}{\begin{equation*}}
\newcommand{\eeqn}{\end{equation*}}

\newcommand{\mb}{\mathbb}

\newcommand{\curly}{\mathrel{\leadsto}}

\DeclareMathOperator\rel{rel}

\makeatletter
\newcommand{\make@circled}[2]{%
  \ooalign{$\m@th#1\smallbigcirc{#1}$\cr\hidewidth$\m@th#1#2$\hidewidth\cr}%
}
\newcommand{\smallbigcirc}[1]{%
  \vcenter{\hbox{\scalebox{0.77778}{$\m@th#1\bigcirc$}}}%
}
\makeatother

\makeatletter
\newcommand{\colim@}[2]{%
  \vtop{\m@th\ialign{##\cr
    \hfil$#1\operator@font colim$\hfil\cr
    \noalign{\nointerlineskip\kern1.5\ex@}#2\cr
    \noalign{\nointerlineskip\kern-\ex@}\cr}}%
}
\newcommand{\colim}{%
  \mathop{\mathpalette\colim@{\rightarrowfill@\textstyle}}\nmlimits@
}
\makeatother

\title[Quantum Adams operations in quasimap $K$-theory]{Quantum Adams operations in quasimap $K$-theory \\ \small{Opérations d'Adams quantiques pour la K-théorie des quasi-applications}}

\author{Shaoyun Bai}
\address{Department of Mathematics, MIT, Boston, MA, 02139, USA}
\email{shaoyunb@mit.edu}

\author{Jae Hee Lee}
\address{Department of Mathematics, Stanford University, Stanford, CA, 94305, USA}
\email{jaeheelee@stanford.edu}

\thanks{The first-named author is supported by NSF DMS-2404843.}

\begin{document}

\maketitle

\begin{abstract}
\noindent(English)\par
We define quantum deformations of Adams operations in $K$-theory, in the framework of quasimap quantum $K$-theory. They provide $K$-theoretic analogs of the quantum Steenrod operations from equivariant symplectic Gromov--Witten theory. We verify the compatibility of these operations with the K\"ahler and equivariant $q$-difference module structures, provide sample computations via $\mathbb{Z}/k$-equivariant localization, and identify them with $p$-curvature operators of the K\"ahler $q$-difference connections as studied in Koroteev--Smirnov. We also formulate and verify a $K$-theoretic quantum Hikita conjecture at roots of unity, and propose an indirect algebro-geometric definition of quantum Steenrod operations. 

\vspace{1em}
\noindent Résumé. (Français)\par
Nous introduisons des déformations quantiques des opérations d'Adams en $K$-théorie, dans le cadre de la $K$-théorie quantique des quasi-applications. Celles-ci fournissent des analogues $K$-théoriques aux opérations de Steenrod quantiques issues de la théorie de Gromov–Witten symplectique équivariante. Nous établissons la compatibilité de ces opérations avec les structures de modules aux $q$-différences kählériens et équivariants, et présentons des exemples de calculs via la localisation $\mathbb{Z}/k$-équivariante. Nous les identifions ensuite aux opérateurs de $p$-courbure des connexions aux $q$-différences kählériennes étudiées par Koroteev--Smirnov. Enfin, nous formulons et vérifions une version $K$-théorique de la conjecture de Hikita quantique aux racines de l’unité, et proposons une définition algébro-géométrique indirecte des opérations de Steenrod quantiques.
\end{abstract}

\setcounter{tocdepth}{1}
\tableofcontents

\section{Introduction}\label{sec:intro}

The goal of this paper is to study quantum deformations of Adams operations in $K$-theory induced from the enumerative geometry of genus $0$ curves in the context of quasimap counts of holomorphic symplectic quotients. We equate such operations with the analog of $p$-curvature for $q$-difference modules in the K\"ahler variables and discuss their relation with Frobenius-constant quantizations in the $K$-theoretic setting under the $3D$ mirror symmetry correspondence.

\subsection{Background}
Before stating the main results, we would like to provide three different contexts which all lead to quantum Adams operations: power operations in topology, $p$-curvature and their $q$-analogs in ODE theory, and symplectic duality in representation theory.

\subsubsection{Quantum deformation of cohomological operations}
Gromov--Witten invariants, which originated from enumerating stable maps, give rise to operations on rational cohomology that can be organized in to a Cohomological Field Theory \cite{kontsevich-manin}. In particular, genus $0$ Gromov--Witten invariants with primary insertions define a quantum deformation of the ordinary cup product. Ordinary cohomology with positive characteristic coefficients carry additional interesting cohomological operations, the Steenrod operations. Parallel to the quantum product, as proposed in \cite{Fuk97} and pursued in \cite{Wil20, seidel-wilkins}, one can construct quantum Steenrod operations deforming the Steenrod operations using genus $0$ stable maps. These operations have led to unexpected advances in both enumerative aspects of geometric representation theory \cite{Lee23b,bai-lee} and arithmetic study of quantum $D$-modules \cite{chen2024exponential}.

Moduli spaces of stable maps can also be used to define $K$-theoretic counts which carry rich structures \cite{givental-lee,lee-K}. Steenrod operations, which are the power operations in ordinary cohomology, have their $K$-theoretic cousin given by Adams operations \cite{adams}. In fact, Steenrod operations and Adams operations are closely related under the tight relation between ordinary cohomology and $K$-theory \cite{atiyah-power}. The following question arises naturally from a theoretical perspective.

\begin{ques}
    How to construct quantum deformations of Adams operations? How to compare them with quantum Steenrod operations?
\end{ques}

\subsubsection{$p$-curvature in $q$-difference modules}
Let $p$ be a prime number. Given a ($t$-)connection over $\mathbb{F}_p(\!(z)\!)$ with regular singularity at $z=0$
\begin{equation}
    \nabla_{z \partial_z} = t z \frac{d}{dz} + A(z), \quad \quad \quad A \in \mathrm{Mat}_{n \times n}(\mathbb{F}_p[\![z]\!]),
\end{equation}
the (Grothendieck--Katz) $p$-curvature of $\nabla$ is the $\mathbb{F}_p(\!(z)\!)$-linear endomorphism
\begin{equation}\label{eqn:p-curvature}
    \nabla_{z \partial_z}^p - t^{p-1} \nabla_{z \partial_z}.
\end{equation}
The $p$-curvature encodes important information of the connection: it is the basic obstruction to formal solvability of the differential equation in characteristic $p$ (cf. Cartier's pioneering work \cite{Cartier}), and is the subject of the famous Grothendieck--Katz $p$-curvature conjecture. The conjecture posits that for a connection over a characteristic $0$ field with well-defined mod $p$ reduction for almost all primes $p$, it is equivalent that the $p$-curvature of the mod $p$ reduction vanishes for almost all primes and that a full set of \emph{algebraic} solutions exists: see \cite{katz-1982} for a systematic discussion. Moreover, nilpotency behavior of $p$-curvatures imposes strong constraints on connections, serving as important criteria for characterizing Picard--Fuchs' equations from all differential connections (cf. \cite{katz-nilpotent}).

There is a $q$-difference analog of $D$-modules, where the role of the derivation $\frac{d}{dz}$ is replaced by the $q$-shift $f(z) \mapsto f(qz)$, while the connection coefficient $A(z)$ is replaced by a multiplicative version $M(q,z)$, which is a invertible matrix-valued function. Here, flat solutions are characterized by $f(z) = f(q,z)$ satisfying
\begin{equation}
    f(qz) = M(q,z)f(z),
\end{equation}
and covariantly constant endomorphisms are those ones which preseve the flat solutions. In the multi-variable setting, i.e., where we can perform the $q$-shift along multiple variables $q_1, \dots, q_m$ with associated connection matrices $M_1(q_1, \dots, q_k,z), \dots, M_k(q_1, \dots, q_k,z)$, the $q$-connection is called \emph{flat} if for any $1 \leq i,j \leq k$, we have
\begin{equation}
    M_i(q_1, \dots, q_k, q^jz) \circ M_j(q_1, \dots, q_k, z) = M_j(q_1, \dots, q_k, q^i z) \circ M_i(q_1, \dots, q_k, z).
\end{equation}
Reducing mod $p$ for $D$-modules corresponds to specializing $q$ to be a root of unity for $q$-difference modules. In the single variable setting, the $p$-curvature is
\begin{equation}\label{eqn:q-p-curvature}
    M(q,q^{p-1}z) \cdot M(q, q^{p-2}z) \cdot \cdots \cdot M(q, qz) \cdot M(q,z) \ |_{q=\zeta_p},
\end{equation}
where $\zeta_p$ denotes a $p$-th root of unity. One can formulate and study a $q$-difference version of the Grothendieck--Katz conjecture, see \cite{q-grothendieck-katz}.

We are interested in $D$-modules and $q$-difference modules of enumerative-geometric origin, namely, quantum $D$-modules and $q$-difference modules arising from quantum cohomology and quantum $K$-theory \cite{MO19, Oko17}. For the former, the second-named author conjectured the equivalence between $p$-curvatures of quantum connections and quantum Steenrod operations, and established the conjecture for a large class of symplectic resolutions \cite{Lee23b}, thereby giving a geometric interpretation of the $p$-curvature in this context. For the latter, in parallel, we are led to the following question.
\begin{ques}
    How to compute $p$-curvatures of K\"ahler $q$-difference connections using simply describable moduli spaces?
\end{ques}

\subsubsection{Large center at roots of unity}
Continuing our discussion of ($q$-difference) connections, the $p$-curvatures \eqref{eqn:p-curvature} and \eqref{eqn:q-p-curvature} are both covariantly constant under the ($q$-difference) connection, i.e. they commute with the action of covariant derivatives/$q$-shifts. In other words, they act by \emph{$D$-module/$q$-difference endomorphisms} of the $D$-modules/$q$-difference modules, and give rise to a distinguished submodule preserved by the connection which arise from ``large centers'' in the ring of differential operators. In fact, there is a well-known analogy between reduction mod $p$ and specializing the variable $q$ to a root of unity in representation theory \cite{zbMATH00596170}, with one manifestation being that large centers naturally arise in both settings.

We are interested in one particular instance of the large center phenomena, known as the \emph{Frobenius-constant quantization} \cite{bezrukavnikov-kaledin-quantp}, which, for a Poisson variety $X$ over a field $k$ of characteristic $p > 0$, is a quantization $A$ of $\mathcal{O}(X)$ and an algebra map
    \begin{equation}
        \Lambda : \mathcal{O}(X)^{(1)} \to Z (A)
    \end{equation}
such that $\Lambda(f) \equiv f^p$ mod $\hbar^{p-1}$, where the superscript ${}^{(1)}$ denotes the Frobenius twist. Here $Z(A)$ is the center of $A$. In the pioneering work \cite[Section 4]{lonergan}, Lonergan proved that the quantization of Coulomb branches \cite{BFN2} induced from the loop-equivariant homology of the moduli space of triples actually gives rise to a Frobenius-constant quantization by discretizing the loop rotation. In the $K$-theoretic setting, the multiplicative version of Coulomb branches arise from $K$-homology of moduli of triples, and the loop-equivariant variable $q$ can be viewed as the generator of the representation ring of $\mathbb{C}^{\times}$. In \cite[Section 4]{lonergan}, instead of reducing mod $p$, it is shown that specializing $q$ to a root of unity leads to an analog of Frobenius-constant quantization for the multiplicative Coulomb branch.

Under the $3D$ mirror symmetry duality \cite{3d-mirror, kamnitzer-survey}, the representation-theoretic information of the Coulomb branch associated to a pair $(G, V)$, where $G$ is a complex reductive group and $V$ is a finite-dimensional $G$-representation, should be reflected by enumerative geometry of the Higgs branch variety $T^*V  /\!\!/\!\!/\!\!/ G$, the holomorphic symplectic reduction of the cotangent $T^*V$ under the natural $G$-action. Therefore, we can ask the following question.
\begin{ques}
    How does one interpret Lonergan's Frobenius-constant quantization of a multiplicative Coulomb branch in terms of the dual Higgs branch?
\end{ques}

\subsection{Statement of results}
In this paper, we provide uniform answers to Questions 1.1 - 1.3 by studying equivariant $K$-theoretic counts of quasimaps into Higgs branches. Suppose $X = T^*V  /\!\!/\!\!/\!\!/ G$ is a hyperK\"ahler reduction acted on by the group $\mathbf{T} = \mathbb{G}_{\mathrm{m}} \times T$, where $\mathbb{G}_{\mathrm{m}}$ comes from scaling the cotangent direction with weight $\hbar^{-1}$ and $T$ is a Hamiltonian torus. Let $k \geq 2$ be an integer, and denote by $\mu_k$ the cyclic group of order $k$. In Section \ref{sec:oper}, using variants of $\mu_k$-equivariant quasimaps from $\mathbb{P}^1$ to $X$ where $\mu_k$ acts by multiplication of $k$-th roots of unity, we define \emph{quantum Adams operators} $Q\psi_{\mathcal{F}}^k$ for any $K$-theory class $\mathcal{F} \in K_{\mathbf{T}}(X)$. We show that quantum Adams operators satisfy the following properties.

\begin{thm}[See Proposition \ref{prop:adams-property} and Lemma \ref{lem:qadams-additivity}]\label{thm:intro-1}
$Q\psi_{\mathcal{F}}^k$ is a $z$-linear endomorphism of $K_{\mathbf{T}}(X)[\![ z^{\mathrm{eff}} ]\!]$, where $z^{\mathrm{eff}}$ denotes the K\"ahler variable from the cone of effective curve classes, such that the following holds.
    \begin{enumerate}
        \item For any $\mathcal{F}, \mathcal{G} \in K_{\mathbf{T}}(X)$, we have
        \begin{equation}
            Q\psi_{\mathcal{F} + \mathcal{G}}^k = Q\psi_{\mathcal{F}}^k + Q\psi_{\mathcal{G}}^k.
        \end{equation}
        \item $Q\psi_{\mathcal{F}}^k |_{z=0} = \psi^k_{\mathcal{F}} \otimes (-)$, where $\psi^k_{\mathcal{F}}$ is the $k$-th Adams operation applied to $\mathcal{F}$.
        \item Let $\star$ be the PSZ quantum product (cf. \cite{PSZ-quantum}) on $K_{\mathbf{T}}(X)[\![ z^{\mathrm{eff}} ]\!]$. Then
        \begin{equation}
            Q\psi_{\mathcal{F} \star \mathcal{G}}^k = Q\psi_{\mathcal{F}}^k \circ Q\psi_{\mathcal{G}}^k.
        \end{equation}
        \item $Q\psi_{\mathcal{F}}^k$ is covariantly constant under both the K\"ahler and equivariant $q$-difference connections defined from quasimap $K$-theory \cite{Oko17}.
    \end{enumerate}
\end{thm}

Theorem \ref{thm:intro-1} provides the first systematic study of quantum Adams operators. The second item shows that our $Q\psi_{\mathcal{F}}$ is indeed a quantum deformation of the Adams operation. Item (4) should be viewed as the $K$-theory version of the results proved in \cite{seidel-wilkins,Lee23b}. In fact, before constructing the quantum Adams operators, we define \emph{quantum cyclic powers} (see Section \ref{ssec:oper-definition}), the quantum deformation of full power operations in $K$-theory, from which the quantum Adams operators are obtained by specializing the generator $q$ of $K_{\mu_k}(\mathrm{pt}) \cong \mathbb{Z}[q]/(q^k=1)$ to a $k$-th root of unity.

\begin{rem}
    In our definition of $Q\psi_{\mathcal{F}}$, the class $\mathcal{F}$ plays the role of insertions placed at $k$-th roots of unity on $\mathbb{P}^1$. Following \cite{Oko17}, we can either ask the quasimap to have either \emph{relative} or \emph{descendant} insertions at the marked points. So, we also define descendant versions of quantum Adams operators $Q\psi^{\mathrm{desc},k}_{\tau}$ where $\tau$ is a $K$-theory class of the ambient stack in which $X$ is an open substack. The two operations $Q\psi_{\mathcal{F}}$ and $Q\psi^{\mathrm{desc},k}_{\tau}$ are related (Proposition \ref{prop:qadams-equivalence}). In general, it is easier to use $Q\psi_{\mathcal{F}}$ for proving general statements and making contact with PSZ $K$-theory, but $Q\psi^{\mathrm{desc},k}_{\tau}$ are better for computational purposes due to the simplicity of moduli spaces: see Section \ref{ssec:oper-example} for a sample computation for $X = T^* \mathbb{P}^n$.
\end{rem}

\begin{rem}\label{rem:stable}
    Differently from the setting of Question 1.1, we construct the quantum Adams operators using quasimaps instead of stable maps. We view this as a feature instead of a loss of generality: one particular reason is that quasimap $K$-theory is much better understood following Okounkov's framework \cite{Oko17}. Nevertheless, we will investigate quantum Adams operators using stable maps in forthcoming work where the target space can be any compact \emph{symplectic manifold}.
\end{rem}

As mentioned above, \cite[Section 8]{Oko17} constructs $q$-difference module structures of K\"ahler variables from quasimap counts of Higgs branch. Our next main result answers Question 1.2 in full generality under the slogan ``$p$-curvature = quantum power operation."

\begin{thm}[= Theorem \ref{thm:qadams=pcurvature}]\label{thm:intro-2}
    Suppose $L$ is a line bundle of the ambient stack in which $X$ is an open substack. Then the $p$-curvature of the K\"ahler $q$-difference connection along $L$ is equal to the descendant quantum Adams operator $Q\psi^{\mathrm{desc},p}_L$.
\end{thm}

To clarify, we do not ask $p$ to be a prime number here: it works for any integer $p \geq 2$, though we would like to keep the notations more streamlined in the introduction. In contrast to the cohomological setting, where a similar result crucially relies on \emph{algebraic} arguments from shift operators and semi-simplicity of the quantum multiplication, the proof of Theorem \ref{thm:intro-2} is purely \emph{geometric}. In other words, we argue from the point view of moduli spaces of quasimaps that the $p$-curvature of the K\"ahler $q$-difference connection of $X$ allows a modular interpretation using quantum Adams operators.

Theorem \ref{thm:intro-2} also inspires our approach to the second part of Question 1.1. In \cite{koroteev-smirnov}, Koroteev--Smirnov argued that certain limit of the $p$-curvature of the K\"ahler $q$-difference connection recovers the $p$-curvature of the quantum connection. In view of the slogan ``$p$-curvature = quantum power operation," we have a schematic diagram
\begin{equation}
    \begin{tikzcd}
\text{Quantum Adams} \arrow[d, dotted] \arrow[rr] &  & q\text{-difference } p\text{-curvature} \arrow[d, dotted] \arrow[ll] \\
\text{Quantum Steenrod} \arrow[rr]                &  & \text{differential } p\text{-curvature} \arrow[ll]                      
\end{tikzcd}
\end{equation}
where the vertical arrows indicate taking the cohomological limit. We formulate precise conjectures (cf. Conjecture \ref{conj:qadams-degenerates-to-qst-divisors} and Conjecture \ref{conj:qadams-degenerates-to-qst}) to relate quantum Adams operators and quantum Steenrod operators, which provide a conjectural algebro-geometric approach to quantum Steenrod operators that was only available using symplectic enumerative geometry (see, e.g., \cite[Section 3.1]{bai-lee} for a more detailed discussion).

Finally, regarding Question 1.3, we appeal to the \emph{quantum Hikita conjecture} \cite{KMP21} and take inspriation from the mod $p$ arithmetic version \cite{bai-lee} to provide a conjectural understanding of Lonergan's Frobenius-constant quantization as distinguished central elements of $q$-difference modules.

\begin{conj}[= \cref{conj:arithmetic-k-hikita}]
    Given a complex reductive group $G$ and a finite-dimensional complex $G$-representation $V$, the $q = \zeta_k$ specialization of the Calabi--Yau specialization of the K\"ahler $q$-difference module of the Higgs branch $T^*V  /\!\!/\!\!/\!\!/ G$ is isomorphic to the $q = \zeta_k$ specialization of the $q$-difference module of twisted traces associated with the multiplicative Coulomb branch of $(G,V)$, and the quantum Adams operators correspond to the endomorphisms induced by Lonergan's Frobenius-constant quantization under this isomorphism.
\end{conj}

We refer the reader to Section \ref{sec:quantumHikita} for relevant definitions. As a proof of principle, we verify the conjecture for hypertoric varieties, i.e., $G$ being a complex torus, with an emphasis on the simplest case of $G = \mathbb{C}^{\times}$ which acts on $V = \mathbb{C}$ by the standard weight $1$ action.

\subsection{Further directions}
The major goal of this paper is to initiate the study of quantum Adams operators and reveal its distinguished position in the recent developments of enumerative geometry. Pursuing the conjectures posted in this paper for concrete examples should be interesting, and pinning down closed-form formulas of quantum Adams operators based on equivariant localization techniques as exhibited in Section \ref{ssec:oper-example} may also be meaningful. 

We would like to point out one potential line of investigation based on Theorem \ref{thm:intro-2}. Note that the formula for $p$-curvature in Definition \ref{defn:p-curvature} is essentially a quantum multiplication operator for $L^k$, whose description is obtained as an iterated product of wall-crossing operators from the associated quantum group action of $U_\hbar (\hat{\mathfrak{g}}_Q)$ by the main theorem of \cite{okounkov-smirnov} when $X$ is a Nakajima quiver variety. Hence, the main results in \emph{loc. cit.} provide a representation-theoretic interpretation of the quantum Adams operator in terms of quantum affine algebras at roots of unity. It is an interesting question to understand the role of quantum Adams operations in representation theory, especially $\mu_k$-localization method should make the computation quite approachable. We hope to pursue this relationship in more detail in forthcoming work.

Of course, continuing the discussion in Remark \ref{rem:stable}, using stable maps, one should be able to construct quantum Adams operations for more general targets building on either the algebro-geometric \cite{lee-K} or the symplectic \cite{AMS2} framework. More generally, power operations exist for larger classes of generalized cohomology theories (see e.g. \cite[Section 9]{devalapurkar-satake}), and it would be interesting to see if they admit quantum deformations.

\subsection{Organization}
In Section \ref{sec:background}, we recall background knowledge of quasimap counts following \cite{quasi-map} and \cite{okounkov15}. We define the two flavors (relative/descendant) of quantum cyclic power operators and quantum Adams operators, and study their properties in Section \ref{sec:oper}. Section \ref{ssec:oper-example} is devoted to the computation of quantum Adams operators of $T^* \mathbb{P}^n$, which relies on $\mu_k$-localization. In Section \ref{sec:pcurvature}, we prove that the $p$-curvature of the K\"ahler $q$-difference connection is equal to the quantum Adams operator, and provide conjectures on the relation between the latter and quantum Steenrod operators. Finally, in Section \ref{sec:quantumHikita}, we state the $K$-theoretic quantum Hikita conjecture at roots of unity and demonstrate its validity in the example of abelian gauge theories.

\subsection{Acknowledgements}
We thank Sanath Devalapurkar, Hunter Dinkins, Vasily Krylov, Andrei Okounkov, and Andrey Smirnov for very helpful conversations. In particular, we thank Vasily Krylov and Andrey Smirnov for explaining their works on the quantum Hikita conjecture and quantum $K$-theory of quiver varieties at roots of unity, respectively. We also thank the anonymous referee for a careful examination of our paper, whose suggestions improved the exposition of this article.

\section{Background on quasimaps}\label{sec:background}
\subsection{Stable quasimaps and Higgs branches}\label{ssec:oper-quasimaps}
In this subsection we briefly review the definition of stable quasimaps into GIT quotients of affine varieties and their relative compactifications following \cite{quasi-map} and \cite{Oko17}. We work over the field of complex numbers $\mathbb{C}$.

Let $W$ be an affine algebraic variety acted on by a reductive algebraic group $G$. Choosing a character $\theta \in \chi(G)$, denote by $W \! \sslash \! G = W \! \sslash_\theta \! G$ the GIT quotient under the stability condition specified by $\theta$, which is a quasi-projective variety with a proper morphism $W \! \sslash_\theta \! G \to W/G$ over the affine quotient. We assume that the semi-stable locus $W^{\mathrm{ss}} = W^{\mathrm{ss}}(\theta)$ coincides with the stable locus $W^s = W^s(\theta)$, and $W^s$ is nonsingular on which $G$ acts freely. Then $W \! \sslash \! G$ coincides with the quotient stack $[W^s/G]$, which is an open substack of $[W/G]$, and it comes with a polarization $L_{\theta}$.

Write $\mathrm{Pic}^G(W)$ the group of isomorphism classes of $G$-linearized line bundles on $W$. Given a prestable pointed curve $(C, z_1, \dots, z_k)$, a map from $(C, z_1, \dots, z_k)$ to the quotient stack $[W/G]$ encodes the information of a principal $G$-bundle $P \to C$ together with a section $u: C \to P \times_G W$. In particular, given $L \in \mathrm{Pic}^G(W)$, we obtain a line bundle $u^*(P \times_G L)$ on $C$. The degree of $(P, u) \in \mathrm{Map}(C, [W/G])$ is the homomorphism 
\begin{equation}
    d: \mathrm{Pic}^G(W) \to \mathbb{Z}, \quad \quad \quad d(L) = \mathrm{deg}(u^*(P \times_G L)).
\end{equation}

\begin{defn}
    A stable quasimap from $(C, z_1, \dots, z_k)$ to $W \! \sslash \! G$ of degree $d$ consists of $(P, u) \in \mathrm{Map}(C, [W/G])$ of class $s$ such that all but finitely points (called \emph{base points}) $p \in C$ satisfies $u(p) \in W^s$. An isomorphism between two quasimaps is given by an isomorphism of pointed curves covered by an isomorphism of principal bundles under which the sections are intertwined. 
\end{defn}

Any $d: \mathrm{Pic}^G(W) \to \mathbb{Z}$ that is realized as the degree of a quasimap is called an effective class. We will use the notation $z^d$ to record the degrees of quasimaps in generating series, and use the symbol $[\![ z^{\mathrm{eff}} ]\!]$ to denote the Novikov ring generated by the effective classes $\mathrm{Eff}(X)$.

In this paper, we shall focus on the case when the targets are Higgs branches of cotangent type which is equipped with a polarization. In other words, we consider a finite-dimensional vector space $V$ on which $G$ acts. Then, the $G$-action extends to a symplectic action on $T^*V = V \oplus V^{\vee}$ with moment map $\mu: T^*V \to \mathfrak{g}$. The variety $W$ will be $\mu^{-1}(0)$. Upon choosing $\theta \in \chi(G)$, we obtain the hyperK\"ahler reduction
\begin{equation}
X := T^*V  /\!\!/\!\!/\!\!/_{\theta} G = \mu^{-1}(0) \sslash_{\theta} G \subset [\mu^{-1}(0) /  G].
\end{equation}
In our setting, the $G$-action on $T^* V$ admits an extension to a $G \times \mathbb{G}_{\mathrm{m}} \times T$-action, where $T$ is a torus giving rise to a Hamiltonian action and $\mathbb{G}_{\mathrm{m}}$ scales the cotangent direction with weight $\hbar^{-1}$. Denote by $\mathbf{T} = \mathbb{G}_{\mathrm{m}} \times T$, which acts on $X$. 

\begin{defn}
    Let $X$ be the hyperK\"ahler reduction as above. A \emph{polarization} is a $K$-theoretic class $T^{1/2}X \in K_{\mathbf{T}}(X)$ such that 
    \begin{equation}
        TX = T^{1/2}X + \hbar^{-1}(T^{1/2}X)^{\vee} \in K_{\mathbf{T}}(X).
    \end{equation}
\end{defn}

\begin{example}[Nakajima quiver varieties \cite{Nak94}]
Consider a quiver with set of vertices $I$ and $m_{ij}$ arrows from $i \in I$ to $j \in I$. Let $n = |I|$ be the number of vertices. Then the quiver variety $\mathcal{M}(\mathsf{v}, \mathsf{w})$ with dimension vector $\mathsf{v}, \mathsf{w} \in \mathbb{Z}_{\geq 0}^n$ is the holomorphic symplectic reduction
\begin{equation}
    \mathcal{M}(\mathsf{v}, \mathsf{w}) = T^*M  /\!\!/\!\!/\!\!/_{\theta} G = \mu^{-1}(0) \sslash_{\theta} G,
\end{equation}
where $M$ is the quiver representation
\begin{equation}
    M = \bigoplus_{i,j \in I} \mathrm{Hom}(V_i, V_j) \otimes Q_{ij} \oplus \bigoplus_{i \in I}\mathrm{Hom}(W_i, V_i)
\end{equation}
with the dimensions of $V_i$ and $W_i$ specified by $\mathsf{v}$ and $\mathsf{w}$, and $Q_{ij}$ is the vector space of dimension $m_{ij}$. The group $G$ is $\prod_{i \in I}GL(V_i)$ and $\mu$ is the moment map $\mu: T^*M \to \mathfrak{g}^{\vee}$. The stability condition $\theta$ comes from $\theta_i \in \mathbb{Z}$ for $i \in I$, corresponding to the character
\begin{equation}
    G \ni (g_i)_{i \in I} \to \prod_{i \in I} \det(g_i)^{\theta_i}.
\end{equation}
The quiver variety $\mathcal{M}(\mathsf{v}, \mathsf{w})$ admits an action by $\mathbf{T} = \mathbb{G}_{\mathrm{m}} \times T$, where $\mathbb{G}_{\mathrm{m}}$ scales the cotangent direction by weight $\hbar^{-1}$ and $T$ is the maximal torus of the automorphism group $\prod_{i \in I} GL(W_i)$ from the framing spaces.

The space $M$ is a Lagrangian subspace of $T^* M$, and the $\mathbf{T}$-equivariant $K$-theory class
\begin{equation}
    M - \sum_{i} \mathrm{End}(V_i)
\end{equation}
descends to $\mathcal{M}(\mathsf{v}, \mathsf{w})$ and defines a polarization \cite[Section 6.1]{Oko17}.
\end{example}

\begin{example}[Hypertoric varieties \cite{hyper-toric}]\label{exmp:hypertoric}
    Consider $V = \mathbb{C}^n$ together with the coordinate scaling action of the torus $(\mathbb{G}_{\mathrm{m}})^n$. For a fixed map $(\mathbb{G}_{\mathrm{m}})^n \to T$ for $T$ some torus, there is an associated exact sequence of algebraic groups
\begin{equation}
    \begin{tikzcd}
        1 \rar & K \rar & (\mathbb{G}_{\mathrm{m}})^n \rar & T \rar & 1.
    \end{tikzcd}
\end{equation}
By restriction, $\mathbb{C}^n$ admits a $K$-action, which extends to a Hamiltonian $K$-action on $T^*\mathbb{C}^n$. Denote the corresponding moment map by $\mu : T^*\mathbb{A}^n \to \mathfrak{k}^*$, where $\mathfrak{k}^* := \mathrm{Hom}(K, \mathbb{G}_{\mathrm{m}})$ is the Lie coalgebra. By choosing a generic stability condition $\chi \in \mathfrak{k}^*$, the associated projective GIT quotient $X : = \mu^{-1}(0)/\!\!/_\chi K$ is a \emph{hypertoric variety}, acted on by $\mathbf{T} = \mathbb{G}_{\mathrm{m}} \times T$ where the $\mathbb{G}_{\mathrm{m}}$-action is induced from scaling the cotangent directions. Let $L_i \to X$ be the line bundle associated with the character $\theta_i \in \chi(K)$ given by the composition $K \to (\mathbb{G}_{\mathrm{m}})^n$ and the restriction to the action to the $i$-th coordinate. Then the $\mathbf{T}$-equivariant $K$-theory class
\begin{equation}
    \sum_i L_i - \mathcal{O}^{\oplus k}
\end{equation}
defines a polarization of $X$, where $k = \dim K$ \cite{smirnov-zhou}.
\end{example}

From now on, $X$ denotes a Higgs branch of cotangent type with polarization $T^{1/2}X$. We also assume that the $\mathbf{T}$-fixed point locus of $X$ is proper, which is relevant for Definition \ref{defn:equiv-shift}. Write $\mathfrak{X}$ the quotient stack $[\mu^{-1}(0) / G]$. 

We study quasimaps of $X$ following Okounkov \cite{Oko17}. The base situation is when $C$ is fixed to be $\mathbb{P}^1$. Denote by $\mathsf{QM}_d(X)$ the moduli space parametrizing stable quasimaps from $\mathbb{P}^1$ to $X$ of degree $d$. Then $\mathsf{QM}_d(X)$ admits a perfect obstruction theory \cite{behrend-fantechi} with virtual tangent complex given by the restriction of $R^{\bullet}\pi_*f^*T_{\mathfrak{X}}$, where $\pi$ is the projection and $f$ is the evaluation
\begin{equation}
    \begin{tikzcd}
{\mathbb{P}^1 \times \mathrm{Map}(\mathbb{P}^1, \mathfrak{X})} \arrow[d, "\pi"'] \arrow[r, "f"] & \mathfrak{X} \\
{\mathrm{Map}(\mathbb{P}^1, \mathfrak{X}).}                                                      &             
\end{tikzcd}
\end{equation}
Working $\mathbf{T}$-equivariantly, one can construct the virtual structure sheaf $\mathcal{O}_{\mathrm{vir}} \in K_{\mathbf{T}}(\mathsf{QM}_d(X))$ \cite{lee-K}. It can be twisted as follows. Denote by $p_1 = 0$ and $p_2 = \infty$ on $\mathbb{P}^1$. Then consider
\begin{equation}\label{eqn:o-twist}
    \hat{\mathcal{O}}_{\mathrm{vir}} := \mathcal{O}_{\mathrm{vir}} \otimes \big(\mathcal{K}_{\mathrm{vir}} \frac{\det f^* (T^{1/2}X)|_{p_2}}{\det f^* (T^{1/2}X)|_{p_1}}\big)^{\frac12},
\end{equation}
where $\mathcal{K}_{\mathrm{vir}}$ is the determinant of the virtual cotangent complex. Note that $\mathsf{QM}_d(X)$ admits another $\mathbb{C}^{\times}_q$-action induced from $\mathrm{Aut}(\mathbb{P}^1, p_1, p_2)$, and the twisted virtual structure sheaf $\hat{\mathcal{O}}_{\mathrm{vir}}$ lifts to a class in $K_{\mathbb{C}^{\times}_q \times \mathbf{T}}(\mathsf{QM}_d(X))$.

We denote by $\mu_k \cong \mu_k(\mathbb{C})$ the cyclic group of order $k$ with a preferred isomorphism to $\mathbb{Z}/k$. In this paper, we will look at cyclic subgroups $\mu_k \subset \mathbb{C}^{\times}_q$, which consist of $k$-th roots of unity, and study $\mu_k$-equivariant enumerative invariants. Unless otherwise stated, the twisted virtual structure sheaf is viewed as an element
\begin{equation}
    \hat{\mathcal{O}}_{\mathrm{vir}} \in K_{\mu_{k} \times \mathbf{T}}(\mathsf{QM}_d(X)).
\end{equation}
For our later discussions, we will use
\begin{equation}
    (\mathsf{QM}_d(X), \hat{\mathcal{O}}_{\mathrm{vir}}) := \hat{\mathcal{O}}_{\mathrm{vir}}
\end{equation}
to denote the twisted virtual structure sheaf on $\mathsf{QM}_d(X)$, similarly for other variants of quasimap moduli spaces discussed below. More generally, we will use $(\mathsf{QM}_d(X), -)$ to denote (equivariant) $K$-theory classes over $\mathsf{QM}_d(X)$.

$K$-theoretic quasimap invariants are defined by pairing $\hat{\mathcal{O}}_{\mathrm{vir}, \mathsf{QM}_d}$ with $K$-theory classes of either $\mathfrak{X}$ or $X$ using the evaluation maps. To have insertions from $X$, due to non-properness of $\mathsf{QM}_d(X)$, one can use \emph{relative} moduli spaces to carry out the construction. 

For example, the relative moduli space with degeneration at $p_2$, denoted by $\mathsf{QM}_d(X)_{\rel  p_2}$, is the following. Write 
\begin{equation}
    \mathbb{P}^1[l] := \mathbb{P}^1 \cup \mathbb{P}^1 \cup \cdots \cup \mathbb{P}^1
\end{equation}
the nodal curve obtained by attaching a chain of $l$ $\mathbb{P}^1$'s to the point $\infty$, and we mark the point $\infty$ on the last component as $p_2$. Then, $\mathbb{P}^1[l]$ is equipped with a projection map
\begin{equation}\label{eqn:chain-p-1}
    \pi: \mathbb{P}^1[l] \to \mathbb{P}^1
\end{equation}
collapsing the additional chain of $\mathbb{P}^1$'s. Moreover, the parametrized component of $\mathbb{P}^1[l]$ which is mapped isomorphically onto $\mathbb{P}^1$ via $\pi$ admits actions by cyclic groups which will be used later.

\begin{defn}
A \emph{relative quasimap} from $\mathbb{P}^1[l]$ to $X$ is a stable quasimap from $\mathbb{P}^1[l]$ to $X$ such that $p_2$ is not a base point. It is called stable if its automorphism group is finite and the base points avoid the nodal points.
\end{defn}

By the general theory in \cite{li-relative}, there exists a smooth Artin stack $\mathcal{B}$ with universal family $\mathcal{C} \to \mathcal{B}$ which parametrizes all possible extended pairs $(\mathbb{P}^1[l], p_2)$, in the sense that a geometric point of $\mathcal{B}$ is represented by the curve $(\mathbb{P}^1[l], p_2)$ for some $l \geq 0$, and the automorphism group at this point is $(\mathbb{C}^{\times})^{l}$, given by rotating the bubble components. Then a family of relative quasimaps with respect to $p_2 \in \mathbb{P}^1$ over $S$ consists of a Cartesian diagram
\begin{equation}
    \begin{tikzcd}
\mathcal{C}_S \arrow[r] \arrow[d] & \mathcal{C} \arrow[d] \\
S \arrow[r]                       & \mathcal{B}          
\end{tikzcd}
\end{equation}
together with a map $\mathcal{C}_S \to \mathfrak{X}$ such that each fiber is a relative quasimap from $\mathbb{P}^1[l]$ to $X$. 

\begin{defn}
Denote by $\mathsf{QM}_d(X)_{\rel  p_2}$ the moduli stack parametrizing all stable relative quasimaps with respect to $\infty \in \mathbb{P}^1$ of degree $d$. 
\end{defn}

$\mathsf{QM}_d(X)_{\rel  p_2}$ is a Deligne--Mumford stack with a perfect obstruction theory, and one can define its virtual structure sheaf and the twisted version, which will also be denoted by $\hat{\mathcal{O}}_{\mathrm{vir}}$ when the context is clear. The key feature is that the evaluation map at $p_2$
\begin{equation}
    \mathrm{ev}_{\infty}: \mathsf{QM}_d(X)_{\rel  p_2} \to X
\end{equation}
is well-defined and proper. Similarly, for any $p \in \mathbb{P}^1$, one can construct the moduli space of relative quasimaps with respect to $p$ of degree d, denoted by $\mathsf{QM}_d(X)_{\rel  p}$, and the twisted virtual structure sheaf thereon in the same fashion. More generally, given multiple marked points $p_1, \dots, p_m \in \mathbb{P}^1$, we can consider the relative degenerations at all these marked points, and we write the resulting moduli space as $\mathsf{QM}_d(X)_{\rel  p_1,\dots, p_m}$. It comes with the $\mathbf{T}$-equivariant proper map
\begin{equation}
    \mathrm{ev}_{p_1} \times \cdots \times \mathrm{ev}_{p_m}: \mathsf{QM}_d(X)_{\rel  p_1,\dots, p_m} \to X^{\times m}.
\end{equation}
When $m \leq 2$, the map is $\mu_k$-equivariant, where $\mu_k$ comes from discretized loop rotation,  if $X^m$ is endowed with the trivial $\mu_k$-action.

On the other hand, recall that the evaluation map $\mathrm{ev}: \mathbb{P}^1 \times \mathsf{QM}_d(X) \to \mathfrak{X}$ is proper \cite[Theorem 4.1.2]{quasi-map}, and the pullback in $K$-theory from the quotient stack gives rise to \emph{descendant insertions} (cf. \cite[Section 7]{Oko17}). Also, to construct important objects like the vertex functions, it is customary to consider quasimap moduli spaces with nonsingular conditions, i.e., we require that the evaluation of the quasimap at given marked point lies in the stable locus. We will use the subscript ${()}_{\mathrm{ns} \ p}$ to indicate the condition that the point $p$ is mapped into the stable locus.

There is a variant of quasimaps, known as \emph{twisted quasimaps}, defined using a cocharacter $\sigma: \mathbb{C}^{\times} \to T$. In our setting, the fact that $W$ admits a $G \times \mathbb{G}_{\mathrm{m}} \times T$-action means that $T$ acts on the stack $[W/G]$ as automorphisms. Viewing $\mathbb{P}^1$ as the quotient of $\mathbb{C}^2 \setminus \{0\}$ by the diagonal $\mathbb{C}^{\times}$-action of weight $1$, we consider 
\begin{equation}
    [W/G]^{\sim}_{\sigma} := [W/G] \times_{\mathbb{C}^{\times}} \big( \mathbb{C}^2 \setminus \{0\} \big) \to \mathbb{P}^1,
\end{equation}
which can be alternatively described as $[W^{\sim}_{\sigma} / G]$ where
\begin{equation}
    W^{\sim}_{\sigma} := W \times_{\mathbb{C}^{\times}} \big( \mathbb{C}^2 \setminus \{0\} \big) \to \mathbb{P}^1,
\end{equation}
in which $G$ acts on the $W$-factor. Denote by $\mathsf{QM}^{\sigma}_d(X)$ the moduli space of sections $u^{\sim}$ of $[W/G]^{\sim}_{\sigma} \to \mathbb{P}^1$ such that 
\begin{enumerate}
    \item all but finitely many points of $p \in \mathbb{P}^1$ lies in $W^s$ under the section;
    \item the degree of $u^{\sim}$ is $d$, which is the association 
    \begin{equation}
        \mathrm{Pic}^G(W_{\sigma}^{\sim}) \ni L^{\sim} \mapsto \mathrm{deg}((u^{\sim})^* L^{\sim}).
    \end{equation}
\end{enumerate}
Note that $d$ takes value in an $\mathrm{Eff}(X)$-torsor. We can similarly define moduli spaces $\mathsf{QM}^{\sigma}_d(X)_{\rel  p_1,\dots, p_m}$ by introducing relative degenerations, with the understanding that the evaluation map at $p_i$ takes value in the fiber over $p_i \in \mathbb{P}^1$. We will use $(\mathsf{QM}^{\sigma}_d(X)_{\rel  p_1,\dots, p_m}, -)$ to denote $K$-theory classes over $\mathsf{QM}^{\sigma}_d(X)_{\rel  p_1,\dots, p_m}$, including the twisted virtual structure sheaf $(\mathsf{QM}^{\sigma}_d(X)_{\rel  p_1,\dots, p_m}, \hat{\mathcal{O}}_{\mathrm{vir}})$.

Alternatively, given a cocharacter $\sigma: \mathbb{C}^{\times} \to T$, an element in $\mathsf{QM}^{\sigma}_d(X)$ encodes the information of a principal $G$-bundle $P \to \mathbb{P}^1$ together with a section of $(P \times_{\mathbb{P}^1} \mathbb{C}^{\times}(1)) \times_{G \times \mathbb{C}^{\times}} W \to \mathbb{P}^1$ which takes value in $W^s$ way from finitely many points on $\mathbb{P}^1$ such that the degree is $d$, where $\mathbb{C}^{\times}(1) \to \mathbb{P}^1$ is the frame bundle associated with $\mathcal{O}(1)$.

\subsection{Quantum $K$-theory}\label{ssec:oper-quantumK}
In this subsection we review quasimap quantum $K$-theory, including the PSZ quantum $K$-theory \cite{PSZ-quantum}, and the $q$-difference modules associated to K\"ahler and equivariant variables and their solutions, which are given by vertex functions \cite{Oko17}.

Recall that we have the action $T$ of a Hamiltonian torus on $X$, and $\mathbb{G}_{\mathrm{m}}$ by the conical scaling action on $X$. We denote the equivariant $K$-theory of $X$ by $K_{\mathbf{T}}(X)$. It is a module over the representation ring $K_\mathbf{T}(\mathrm{pt}) \cong \mathbb{Z}[\hbar^\pm , a_1^\pm, \dots, a_r^\pm]$ where $\hbar$ and $a_1, \dots a_r$ for $ r= \mathrm{rank}(T)$ are the equivariant parameters for the $\mathbb{G}_{\mathrm{m}}$ and $T$-actions respectively. When the rank of the Hamiltonian torus is understood, by a slight abuse of notation we will denote the equivariant parameters by $K_\mathbf{T}(\mathrm{pt}) \cong \mathbb{Z}[\hbar^\pm, a^\pm]$. Note that any class $\mathcal{F}' \in K_{\mathbf{T}}(X \times X)$ gives rise to a $K_{\mathbf{T}}$-linear operator
\begin{equation}\label{eqn:convolution}
    K_{\mathbf{T}}(X) \ni \mathcal{F} \mapsto (\pi_1)_{*}(\mathcal{F}' \otimes \pi_2^*\mathcal{F}\otimes K_X^{-1/2}) \in K_{\mathbf{T}}(X)
\end{equation}
where $\pi_i: X \times X \to X$ is the projection to the $i$-th factor.

To capture the $\mu_k$-action on quasimaps, we equip $X$ with the trivial $\mu_k$-action and consider the $\mu_k \times \mathbf{T}$-equivariant $K$-theory $K_{\mu_k \times \mathbf{T}}(X)$ of $X$, which is also a module over the representation ring $K_{\mu_k}(\mathrm{pt}) \cong \mathbb{Z}[q,q^{-1}]/(q^k - 1)$, and $q$ will be referred to as the \emph{loop parameter}.

\begin{defn}[Gluing operator]
We define the gluing operator to be the map
\begin{equation}\label{eqn:glue}
    \mathbf{G} \in \mathrm{End}(K_{\mu_k \times \mathbf{T}}(X))[\![ z^{\mathrm{eff}} ]\!]
\end{equation}
induced by the $K$-theory class
\begin{equation}
    \sum_d z^d (\mathrm{ev}_{p_1} \times \mathrm{ev}_{p_2})_* \big(\mathsf{QM}_d(X)_{\rel  p_1, p_2}, \hat{\mathcal{O}}_{\mathrm{vir}}\big) \in K_{\mu_k \times \mathbf{T}}(X^{\times 2})[\![ z^{\mathrm{eff}} ]\!].
\end{equation}
under the map \eqref{eqn:convolution}.
\end{defn}

By construction, $\mathbf{G} = \mathrm{Id} + O(z)$, which shows that it is invertible. 

\begin{rem}
    In \cite[Section 6.5]{Oko17}, an alternative construction of $\mathbf{G}$ is provided using the moduli spaces $\mathsf{QM}_d(X)^{\sim}_{\rel p_1, p_2}$, which consists of quasimaps with relative degenerations at $p_1$ and $p_2$ such that each irreducible component is unparametrized. The proof of \cite[Theorem 7.1.4]{Oko17}, which relies on the degeneration formula \cite[Proposition 6.5.27]{Oko17}, applies to the $\mu_k$-equivariant setup here. In particular, the gluing operator $\mathbf{G}$ does not depend on $q$ and defines an element 
    \begin{equation}
    \mathbf{G} \in \mathrm{End}(K_{\mathbf{T}}(X))[\![ z^{\mathrm{eff}} ]\!].
\end{equation}
\end{rem}

In the following definition, we work $\mathbf{T}$-equivariantly.
\begin{defn}[PSZ quantum $K$-theory ring]
    The Pushkar--Smirnov--Zeitlin quantum $K$-theory ring of $X$ is the unital commutative ring such that 
    \begin{itemize}
    \item the underlying space is $K_{\mathbf{T}}(X)[\![ z^{\mathrm{eff}} ]\!]$;
    \item the product, which is bilinear in the $z$-variable, is given by
    \begin{equation}
        (\mathcal{F}, \mathcal{G}) \mapsto \mathcal{F} \star \mathcal{G} := \sum_{d} z^d (\mathrm{ev}_{p_2})_* \big(\mathsf{QM}_d(X)_{\rel p_1, p_2, p_3 }, \hat{\mathcal{O}}_{\mathrm{vir}} \otimes \mathrm{ev}_{p_1}^* (\mathbf{G}^{-1} \mathcal{F}) \otimes \mathrm{ev}_{p_3}^* (\mathbf{G}^{-1} \mathcal{G}) \big);
    \end{equation}
    \item the identity class is 
    \begin{equation}
        \mathbf{1}(z) := \sum_d z^d (\mathrm{ev}_{p_1})_* \big( \mathsf{QM}_d(X)_{\rel p_1}, \hat{\mathcal{O}}_{\mathrm{vir}} \big).
    \end{equation}
    \end{itemize}
\end{defn}
For the proof of unitality and commutativity, see \cite[Section 3]{PSZ-quantum}. Moreover, $K_{\mathbf{T}}(X)[\![ z^{\mathrm{eff}} ]\!]$ can be equipped with a pairing
\begin{equation}
    \langle \mathcal{F}, \mathcal{G} \rangle := \sum_d z^d \chi\big(\mathsf{QM}_d(X)_{\rel p_1, p_2}, \hat{\mathcal{O}}_{\mathrm{vir}} \otimes \mathrm{ev}_{p_1}^* (\mathbf{G}^{-1} \mathcal{F}) \otimes \mathrm{ev}_{p_2}^* (\mathbf{G}^{-1} \mathcal{G}) \big),
\end{equation}
where $\chi$ is the $\mathbf{T}$-equivariant holomorphic Euler characteristic. It also follows from \cite[Section 3]{PSZ-quantum} that the PSZ quantum $K$-theory ring is a Frobenius algebra under this pairing. Later on, when discussing $\mu_k$-equivariant quasimaps, we extend the quantum product $\star$ $q$-linearly, which defines a product on $K_{\mu_k \times \mathbf{T}}(X)$.

Next, we recall the definitions of $q$-difference operators in both K\"ahler and equivariant variables in quasimap $K$-theory. We work $\mu_k \times \mathbf{T}$-equivariantly here, where $\mu_k$ acts on the quasimap domains by discretized loop rotation.

\begin{defn}[K\"ahler shifts]\label{defn:kahler-shift}
Let $L$ be a class in $\mathrm{Pic}^G(W)$. Then define the operator
\begin{equation}\label{eqn:kahler-shift}
\begin{aligned}
 M_{L}^{(k)}(z) &\in \mathrm{End}(K_{\mu_k \times \mathbf{T}}(X))[\![ z^{\mathrm{eff}} ]\!] \\
 M_{L}^{(k)}(z) &:= \big( \sum_d z^d (\mathrm{ev}_{p_1} \times \mathrm{ev}_{p_2})_* \big( \mathsf{QM}_{d}(X)_{\rel p_1, p_2}, \hat{\mathcal{O}}_{\mathrm{vir}} \otimes \mathrm{det} H^{\bullet}(L \otimes \pi^*(\mathcal{O}_{p_1})) \big) \big)\circ \mathbf{G}^{-1},
\end{aligned}
\end{equation}
where $\pi: \mathbb{P}^1[l] \to \mathbb{P}^1$ is the projection map from \eqref{eqn:chain-p-1}. The notation $L \otimes \pi^*(\mathcal{O}_{p_1})$ is the abbreviation for the tensor product between the pullback of $L \in \mathrm{Pic}^G(W)$ (via the universal quasimap $f: C \to \mathfrak{X}$) and the sheaf  $\pi^*(\mathcal{O}_{p_1})$. Also, $H^{\bullet}(L \otimes \pi^*(\mathcal{O}_{p_1}))$ stands for the equivariant $K$-theory class over $\mathsf{QM}_{d}(X)_{\rel p_1, p_2}$ by taking the cohomology of the sheaf $L \otimes \pi^*(\mathcal{O}_{p_1})$ along each fiber of the universal family. These notations follow the conventions in \cite[Section 4]{Oko17}. Moreover, we identify the $K$-theory class of $X \times X$ defined by the pushforward from $\mathsf{QM}_{d}(X)_{\rel p_1, p_2}$ under the evaluation map as an endomorphism on $K$-theory via \eqref{eqn:convolution}.
\end{defn}
We will write $q^L z^d := q^{\langle L, d \rangle}z^d$. The operator $M_{L}^{(k)}(z)$ is the connection coefficient of the $q$-difference operator along the direction specified by $L$. 

A similar construction works for shifts in the equivariant variables using the moduli spaces of $\sigma$-twisted quasimaps.

\begin{defn}[Equivariant shifts]\label{defn:equiv-shift}
    Let $\sigma: \mathbb{C}^{\times} \to \mathbf{T}$ be a cocharacter. Then define the operator
    \begin{equation}\label{eqn:shift}
        \begin{aligned}
            S_{\sigma}(z) &\in \mathrm{End}(K_{\mu_k \times \mathbf{T}} (X)_{\mathrm{loc}})[\![ z^{\mathrm{eff}} ]\!] \\
            S_{\sigma}(z) &:= \big( \sum_d z^d(\mathrm{ev}_{p_1} \times \mathrm{ev}_{p_2})_* \big( \mathsf{QM}^{\sigma}_{d}(X)_{\rel p_1, p_2}, \hat{\mathcal{O}}_{\mathrm{vir}} \big) \big)\circ \mathbf{G}^{-1},
        \end{aligned}
    \end{equation}
    where we identify the $K$-theory class of $X \times X$ defined by the pushforward from $\mathsf{QM}^{\sigma}_{d}(X)_{\rel p_1, p_2}$ under the evaluation map as an endomorphism on $K$-theory via \eqref{eqn:convolution}.
\end{defn}
It is worth noting that in Equation \eqref{eqn:shift}, the shift operator is an endomorphism of the \emph{localized} equivariant $K$-theory, where we localize with respect to the $\mathbf{T}$-action on the target. This is due to the noncompactness of moduli spaces $\sigma$-twisted quasimaps, which is remedied by imposing the properness assumption on the $\mathbf{T}$-fixed locus of $X$ and localization. Unlike \cite[Section 8]{Oko17}, which defines the shift operators using the $\mathbb{C}^{\times}_q$-action on the source curve, we need to localize over the target because the loop rotation is discretized in our setting.

Given an integral vector $\mathbf{n} = (n_1, \dots, n_r)$ writing $a^{\mathbf{n}} = a_1^{n_1} \cdots a_r^{n_r}$ as the $T$-equivariant parameter in $K$-theory. Define the equivariant shift to be $q^{\sigma}a^{\mathbf{n}} := q^{\langle \sigma, \mathbf{n} \rangle} a^{\mathbf{n}}$. Then we think $S_{\sigma}(q,z)$ as the connection coefficient of the $q$-difference connection in the equivariant variables.

\begin{rem}
    Strictly speaking, the degrees of $\sigma$-twisted quasimaps is an $\mathrm{Eff}(X)$-torsor. To make the identification, we can choose a stable fixed point of the induced $\mathbb{C}^{\times}$-action from $\sigma$, which defines a constant $\sigma$-twisted quasimap, to trivialize the torsor. Since the choice is auxiliary, we omit it from the formulas.
\end{rem}

Going back to $X$, it is equipped with the trivial $\mu_k \subset \mathbb{C}^{\times}_q$ action. Then the map
\begin{equation}\label{eqn:K-res}
    K_{\mathbb{C}^{\times}_q \times \mathbf{T}} (X) \to K_{\mu_k \times \mathbf{T}} (X)
\end{equation}
by the restriction homomorphism is induced by modulo the ideal generated by $q^k - 1$. For the moduli space $\mathsf{QM}_{d}(X)_{\rel p_1, p_2}$, the $\mu_k$-action thereon comes from restricting the $\mathbb{C}^{\times}_q \cong \mathrm{Aut}(\mathbb{P}^1, p_1, p_2)$ on the quasimap domains. Working $\mathbb{C}^{\times}_q \times \mathbf{T}$-equivariantly, we can similarly define an operator
\begin{equation}\label{eqn:kahler-connection}
    M_{L}(z) \in \mathrm{End}(K_{C^{\times}_{q} \times \mathbf{T}}(X))[\![ z^{\mathrm{eff}} ]\!]
\end{equation}
as in \eqref{eqn:kahler-shift}.
\begin{lemma}\label{lemma:trivial}
    The image of $M_{L}(z)$ under the restriction homomorphism in $K$-theory \eqref{eqn:K-res} coincides with $M_{L}^{(k)}(z)$.
\end{lemma}
\begin{proof}
    Working $\mathbb{C}^{\times}_q$-equivariantly, we represent the pushforward of $\big( \mathsf{QM}_{d}(X)_{\rel p_1, p_2}, \mathrm{det} H^{\bullet}(L \otimes \pi^*(\mathcal{O}_{p_1})) \big)$ to the coarse space of $\mathsf{QM}_{d}(X)_{\rel p_1, p_2}$ as a complex of $\mathbb{C}^{\times}_q \times \mathbf{T}$-equivariant coherent sheaves $\cdots \to \mathcal{F}_{i-1} \to \mathcal{F}_i \to \mathcal{F}_{i+1} \to \cdots$. Then its pushforward under the evaluation map is the alternating sum of the higher derived functor sheaves (see, e.g. \cite[Section 5.2.13]{chriss-ginzburg})
    \begin{equation}
        \sum_i (-1)^i\sum_j(-1)^j R^j (\mathrm{ev}_{p_1} \times \mathrm{ev}_{p_2})_*\mathcal{F}_i.
    \end{equation}
    Note that the same formula computes the pushforward of $\big( \mathsf{QM}_{d}(X)_{\rel p_1, p_2}, \mathrm{det} H^{\bullet}(L \otimes \pi^*(\mathcal{O}_{p_1})) \big)$ under the evaluation map as an $\mu_k \times \mathbf{T}$-equivariant $K$-theory class after composing with the group restriction homomorphism. Therefore, $M_{L}^{(k)}(z)$ is obtained from $M_{L}(z)$ by modulo the ideal generated by $(q^k - 1)$. 
\end{proof}
In the literature \cite{Oko17, okounkov-smirnov}, the K\"ahler difference operators are computed as $\mathbb{C}^{\times}_q$-equivariant objects. We include the above simple Lemma to justify that our $\mu_k$-equivariant setting can be viewed as specializing the loop parameter $q$ at $k$-th roots of unity.

\section{Quantum Adams operations}\label{sec:oper}
In this section, we introduce our main objects of study, quantum Adams operators and quantum cyclic powers, both of which come in two flavors, relative and descendant. We establish their properties which in particular prove Theorem \ref{thm:intro-1}.

\subsection{Adams operations}\label{sec:ordinary-adams}
We recall the definition of Adams operations on (equivariant) $K$-theory from the point of view of cyclic powers, which is closely tied to their quantum deformations that will be introduced below.

Consider $X$ with $\mathbf{T}$-action. Then taking the external power gives a map
\begin{equation}
    \begin{aligned}
        (-)^{\boxtimes k}_{\mathrm{eq}}: K_{\mathbf{T}}(X) &\to K_{\mu_k \ltimes \mathbf{T}^k}(X^k) \\
        \mathcal{F} &\mapsto \mathcal{F}^{\boxtimes k},
    \end{aligned}
\end{equation}
which agrees with the usual external tensor product after forgetting the $\mu_k$-action. We call $(-)^{\boxtimes k}_{\mathrm{eq}}$ the \emph{cyclic power operation}. Using the diagonal map $\Delta_k: X \to X^{k}$, we look at the composition
\begin{equation}\label{eqn:full-power}
    P^k: K_{\mathbf{T}}(X) \xrightarrow{(-)^{\boxtimes k}_{\mathrm{eq}}} K_{\mu_k \ltimes \mathbf{T}^k}(X^k) \xrightarrow{\Delta_k^*} K_{\mu_k \times \mathbf{T}}(X),
\end{equation}
where $X$ is equipped with the trivial $\mu_k$-action. Using the isomorphism
\begin{equation}
    K_{\mu_k \times \mathbf{T}}(X) \cong K_{\mu_k}(\mathrm{pt}) \otimes K_{\mathbf{T}}(X)
\end{equation}
and the computation
\begin{equation}\label{eqn:k-iso}
    K_{\mu_k}(\mathrm{pt}) \cong \mathbb{Z}[q, q^{-1}] / (q^k - 1),
\end{equation}
we can define the projection
\begin{equation}\label{eqn:projection}
    K_{\mu_k \times \mathbf{T}}(X) \to K_{\mathbf{T}}(X) \otimes \mathbb{Z}[q, q^{-1}] / \Phi_k(q),
\end{equation}
where $\Phi_m(q)$ is the $m$-th cyclotomic polynomial.
\begin{defn}
    The \emph{Adams operation} $\psi^k: K_{\mathbf{T}}(X) \to K_{\mathbf{T}}(X) \otimes \mathbb{Z}[q, q^{-1}] / \Phi_k(q)$ is the composition of \eqref{eqn:full-power} and \eqref{eqn:projection}.
\end{defn}

In fact, $\psi^k$ takes value in $K_{\mathbf{T}}(X)$, which can be seen from the splitting principle. The following is a list of standard facts about the Adams operations, see, e.g., \cite[Section 2]{atiyah-power}. Note that these properties also characterize the Adams operations uniquely.
\begin{itemize}
    \item For $\mathcal{F}, \mathcal{G} \in K_{\mathbf{T}}(X)$, we have $\psi^k(\mathcal{F}+ \mathcal{G}) = \psi^k(\mathcal{F}) + \psi^k(\mathcal{G})$.
    \item Denote by $\otimes$ the (tensor) product on $K_{\mathbf{T}}(X)$ Then $\psi^k(\mathcal{F} \otimes \mathcal{G}) = \psi^k(\mathcal{F}) \otimes \psi^k(\mathcal{G})$.
    \item Suppose $L \in K_{\mathbf{T}}(X)$ is a class represented by a $\mathbf{T}$-equivariant line bundle. Then we have $\psi^k(L) = L^{\otimes k}$.
    \item Suppose $f: X \to Y$ is a $\mathbf{T}$-equivariant map. Then $\psi^k$ is compatible with pullback, i.e., for $f^*: K_{\mathbf{T}}(Y) \to K_{\mathbf{T}}(X)$ we have $f^* \circ \psi^k = \psi^k \circ f^*$.
\end{itemize}

We can carry out the same construction verbatim for the stack $\mathfrak{X}$ using the cyclic power operation
\begin{equation}\label{eqn:power-stack}
\begin{aligned}
    (-)^{\boxtimes k}_{\mathrm{eq}} : K_\mathbf{T}(\mathfrak{X}) &\to K_{\mu_k \ltimes \mathbf{T}^k}(\mathfrak{X}^k) \\
     \tau & \mapsto \tau^{\boxtimes k}_{\mathrm{eq}} := \tau \boxtimes \cdots \boxtimes \tau.
\end{aligned}
\end{equation}
We will also denote the resulting Adams operation by $\psi^k$.

\subsection{Quantum cyclic powers}\label{ssec:oper-definition}
In this subsection, we introduce quantum deformations of the cyclic power operations using equivariant counts of quasimaps. They come in two flavors, depending on the type of insertions at the $k$-th roots of unity.

\subsubsection{Definitions}
Consider a parametrized domain $\mathbb{P}^1$ equipped with marked points
\begin{equation}
    p_1 = 0, \ p_2 =\infty, \ p_0' = 1, \  p_1' =  \zeta = e^{2\pi i /k}, \ \dots, \ p_{k-1}' =  \zeta^{k-1},
\end{equation}
by which we denote $\mathbb{P}^1_{\zeta_k}$. The $\mu_k$-action on $\mathbb{P}^1_{\zeta_k}$ whose generator acts by $\iota: z \mapsto e^{2\pi i /k} z$ cyclically permutes the marked points $\{ p_0', \dots, p_{k-1}'\}$. Then we have a $\mu_k$-equivariant diagram
\begin{equation}
    \begin{tikzcd}
{\mathbb{P}^1_{\zeta_k} \times \mathrm{Map}(\mathbb{P}^1_{\zeta_k}, \mathfrak{X})} \arrow[d, "\pi"'] \arrow[r, "f"] & \mathfrak{X} \\
{\mathrm{Map}(\mathbb{P}^1_{\zeta_k}, \mathfrak{X})}                                                      &             
\end{tikzcd}
\end{equation}
where $\mu_k$ acts on $\mathrm{Map}(\mathbb{P}^1_{\zeta_k}, \mathfrak{X})$ by precomposing with $\iota$ and acts trivially on $\mathfrak{X}$. Fixing the degree $d$ and imposing the quasimap condition, now the corresponding $\mu_k$-equivariant moduli space $\mathsf{QM}_d^{\mu_k}(X)$, which is isomorphic to $\mathsf{QM}_d(X)$, has its perfect obstruction theory upgraded to one that is $\mu_k$-equivariant.

Allowing relative degenerations at the marked points, there are two flavors of compactifications of $\mathsf{QM}_d^{\mu_k}(X)$ we can introduce. We can use them to define two versions of the quantum cyclic power operations.

The first one is
\begin{equation}\label{eqn:cyclic-moduli}
    \mathsf{QM}_d^{\mu_k}(X)_{\rel  p_1, p_2} \cong \mathsf{QM}_d(X)_{\rel  p_1, p_2},
\end{equation}
for which the $\mu_k$-action extends to one whose generator rotates the parametrized component and fixes the bubbles at $0$ and $\infty$. Introducing the same twisting as in \eqref{eqn:o-twist}, we can define its $\mu_k \times \mathbf{T}$-equivariant twisted virtual structure sheaf 
\begin{equation}
    (\mathsf{QM}_d^{\mu_k}(X)_{\rel  p_1, p_2}, \hat{\mathcal{O}}_{\mathrm{vir}}).
\end{equation}
The moduli space $\mathsf{QM}_d^{\mu_k}(X)_{\rel  p_1, p_2}$ comes with the evaluation map
\begin{equation}
    \mathrm{ev}_{p_1} \times (\mathrm{ev}_{p_0'} \times \cdots \times \mathrm{ev}_{p_{k-1}'}) \times \mathrm{ev}_{p_2}: \mathsf{QM}_d^{\mu_k}(X)_{\rel  p_1, p_2} \to X \times \mathfrak{X}^k \times X.
\end{equation}
For simplicity, we abbreviate
\begin{equation}
    \mathrm{ev}_k^{\mathrm{stack}} := \mathrm{ev}_{p_0'} \times \cdots \times \mathrm{ev}_{p_{k-1}'}.
\end{equation}
Note that $\mathfrak{X}^k$ is equipped with the natural $\mu_k \ltimes \mathbf{T}^k$-action. Under the group homomorphism
\begin{equation}
    \mu_k \times \mathbf{T} \hookrightarrow \mu_k \ltimes \mathbf{T}^k
\end{equation}
induced by the diagonal embedding $\mathbf{T} \hookrightarrow \mathbf{T}^{k}$, the evaluation map $\mathrm{ev}_k^{\mathrm{stack}}$ is equivariant. In particular, we have a map 
\begin{equation}\label{eqn:stack-ev}
    (\mathrm{ev}_k^{\mathrm{stack}})^*: K_{\mu_k \ltimes \mathbf{T}^k}(X^k) \to K_{\mu_k \times \mathbf{T}}(\mathsf{QM}_d^{\mu_k}(X)_{\rel  p_1, p_2}).
\end{equation}

\begin{defn}
    Fix $ \tau \in K_{\mathbf{T}}(\mathfrak{X})$. The \emph{descendant quantum cyclic $k$-th power operator} associated to $\tau$ is defined as
    \begin{equation}\label{eqn:descendant-power}
        \begin{aligned}
        Q\Psi_\tau^{\mathrm{desc}, k}: = \sum_{d} z^d (\mathrm{ev}_{p_1} \times \mathrm{ev}_{p_2})_* \left( \mathsf{QM}_d^{\mu_k}(X)_{\rel  p_1, p_2}, \hat{\mathcal{O}}_{\mathrm{vir}} \otimes (\mathrm{ev}^{\mathrm{stack}}_k)^* \tau^{\boxtimes k}_{\mathrm{eq}} \right) \circ \mathbf{G}^{-1} \\ \in \mathrm{End}(K_{\mu_k \times \mathbf{T}}(X))[\![ z^{\mathrm{eff}} ]\!],
        \end{aligned}
    \end{equation}
    where we view the $K$-theory classes of $X \times X$ as endomorphisms on $K_{\mu_k \times \mathbf{T}}(X)$ via \eqref{eqn:convolution}.
\end{defn}

We refer the reader to the leftmost picture of Figure \ref{figure:degeneration} for a pictorial description of the moduli space, in which the central dot stands for the parametrized $\mathbb{P}^1$-component.

On the other hand, we can consider the moduli space of quasimaps with relative degenerations at all marked points
\begin{equation}
    \mathsf{QM}_d^{\mu_k}(X)_{\rel  p_1,p_2, p'} := \mathsf{QM}_d(X)_{\rel  p_1,p_2, p_0', \dots, p_{k-1}'}.
\end{equation}
In this case, the $\mu_k$-action on the parametrized component extends to an action of $\mu_k$ on $\mathsf{QM}_d^{\mu_k}(X)_{\rel  p_1,p_2, p'}$ that also cyclically permutes the bubble components associated with $p_0', \dots, p_{k-1}'$. As the $\mu_k$ action covers the $\mu_k$ action on the moduli space of all possible extended pairs, we can define the $\mu_k \times \mathbf{T}$-equivariant twisted virtual structure sheaf 
\begin{equation}
    (\mathsf{QM}_d^{\mu_k}(X)_{\rel  p_1,p_2, p'}, \hat{\mathcal{O}}_{\mathrm{vir}})
\end{equation}
again using the twisting \eqref{eqn:o-twist}. This time, the evaluation map takes the form
\begin{equation}
    \mathrm{ev}_{p_1} \times (\mathrm{ev}_{p_0'} \times \cdots \times \mathrm{ev}_{p_{k-1}'}) \times \mathrm{ev}_{p_2}: \mathsf{QM}_d^{\mu_k}(X)_{\rel  p_1,p_2, p'} \to X \times X^k \times X,
\end{equation}
We abbreviate
\begin{equation}
    \mathrm{ev}_k := \mathrm{ev}_{p_0'} \times \cdots \times \mathrm{ev}_{p_{k-1}'}.
\end{equation}
Similar to \eqref{eqn:power-stack}, define the power map
\begin{equation}
    \begin{aligned}
    (-)^{\boxtimes k}_{\mathrm{eq}} : K_\mathbf{T}(X) &\to K_{\mu_k \ltimes \mathbf{T}^k}(X^k) \\
     \mathcal{F} & \mapsto \mathcal{F}^{\boxtimes k}_{\mathrm{eq}} := \mathcal{F} \boxtimes \cdots \boxtimes \mathcal{F},
    \end{aligned}
\end{equation}
which naturally extends to a map $K_\mathbf{T}(X)[\![ z^{\mathrm{eff}} ]\!] \to K_{\mu_k \ltimes \mathbf{T}^k}(X^k)[\![ z^{\mathrm{eff}} ]\!]$ Frobenius-linear in $z$, i.e., $(z^d \cdot\mathcal{F})^{\boxtimes k}_{\mathrm{eq}} = z^{kd} \cdot (\mathcal{F})^{\boxtimes k}_{\mathrm{eq}}$. Just as \eqref{eqn:stack-ev}, the evaluation map defines the pullback
\begin{equation}
    (\mathrm{ev}_k)^*: K_{\mu_k \ltimes \mathbf{T}^k}(X^k) \to K_{\mu_k \times \mathbf{T}} (\mathsf{QM}_d^{\mu_k}(X)_{\rel  p_1,p_2, p'}).
\end{equation}

\begin{defn}
    Take $\mathcal{F} \in K_{\mathbf{T}}(X)[\![ z^{\mathrm{eff}} ]\!]$. The \emph{quantum cyclic $k$th power operator} associated to $\mathcal{F}$ is defined as
    \begin{equation}\label{eqn:power}
    \begin{aligned}
        Q\Psi^k_{\mathcal{F}}: = \sum_{d} z^d (\mathrm{ev}_{p_1} \times \mathrm{ev}_{p_2})_* \left( \mathsf{QM}_d^{\mu_k}(X)_{\rel  p_1,p_2, p'}, \hat{\mathcal{O}}_{\mathrm{vir}} \otimes \mathrm{ev}_k^* ((\mathbf{G}^{-1}\mathcal{F})^{\boxtimes k}_{\mathrm{eq}}) \right) \circ \mathbf{G}^{-1} \\ \in \mathrm{End}(K_{\mu_k \times \mathbf{T}}(X))[\![ z^{\mathrm{eff}} ]\!],
    \end{aligned}
    \end{equation}
    where we view the $K$-theory classes of $X \times X$ as endomorphisms on $K_{\mu_k \times \mathbf{T}}(X)$ via \eqref{eqn:convolution}.
\end{defn}

For our purpose, we will use different versions of the quantum cyclic power in different contexts. As one can see from the definition, the descendant version comes from simpler moduli spaces so that they are better for computations. On the other hand, one may obtain $Q\Psi_\tau^{\mathrm{desc}, k}$ from $Q\Psi^k_{\mathcal{F}}$ for certain $\mathcal{F}$ in the following way.

We recall the following definition from \cite[Section 7.4.1]{Oko17}. Let $\mathsf{QM}_d(X)_{\rel  p_2}$ be the moduli space of degree $d$ quasimaps with relative degeneration at $p_2$ and recall that $p_1$ is the other fixed point of under the action of the automorphism group $\mathbb{C}^{\times}_q$. Write
\begin{equation}
    (\mathsf{QM}_d(X)_{\rel p_2}, \hat{\mathcal{O}}_{\mathrm{vir}})
\end{equation}
the $\mathbb{C}^{\times}_q \times \mathbf{T}$-equivariant virtual structure sheaf of $\mathsf{QM}_d(X)_{\rel p_2}$ with twisting as in \eqref{eqn:o-twist}. Fix $\tau \in K_{\mathbf{T}}(\mathfrak{X})$. Then the \emph{capped vertex with descendant insertion $\tau$} is
    \begin{equation}
        \widehat{V}^{(\tau)} := \sum_d z^d (\mathrm{ev}_{p_2})_* \left( \mathsf{QM}_d(X)_{\rel p_2}, \hat{\mathcal{O}}_{\mathrm{vir}} \otimes (\mathrm{ev}_{p_1}^* \tau) \right) \in K_{\mathbb{C}^{\times}_q \times \mathbf{T}}(X) [\![z^{\mathrm{eff}}]\!].
    \end{equation}

Because the map $\mathrm{ev}_{p_2}$ is proper, following \cite[Definition 10]{PSZ-quantum}, we can take the $q \to 1$ limit and make the following definition.
\begin{defn}
    The \emph{quantum tautological class} $\hat{\tau}(z)$ is defined to be the non-loop-equivariant limit $\widehat{V}^{(\tau)} |_{q=1} \in K_{ \mathbf{T}}(X) [\![z^{\mathrm{eff}}]\!]$.
\end{defn}

\begin{prop}\label{prop:qadams-equivalence}
    Fix $\tau \in K_{\mathbf{T}}(\mathfrak{X})$. Then
    \begin{equation}
        Q\Psi^{\mathrm{desc}, k}_\tau = Q\Psi^k_{\hat{\tau}}.
    \end{equation} 
\end{prop}
\begin{proof}
    This follows from a standard degeneration argument as in \cite[Proposition 6.5.27]{Oko17}. Namely, starting from the moduli space $\mathsf{QM}_d^{\mu_k}(X)_{\rel  p_1, p_2}$ with the $\mu_k$-action permuting $p_0', \dots, p_{k-1}'$, we introduce accordions (cf. \cite[Section 6.5.7]{Oko17}) at $p_0', \dots, p_{k-1}'$ while maintaining the $\mu_k$-equivariance so that the gluing points come equipped with relative insertions, with inverse of the gluing matrix inserted in between. Note that although our gluing matrix \ref{eqn:glue} has a parametrized component in its definition, it is proved in \cite[Theorem 7.1.4]{Oko17} that one can alternatively compute the gluing operator using accordions, a chain of unparametrized rational curves. This is exactly the formula \eqref{eqn:power} with $\hat{\tau}$ inserted.
\end{proof}

In particular, the computation of $Q\Psi_{\mathcal{F}}^k$ can be simplified to the computation of $Q\Psi^{k, \mathrm{desc}}_\tau$ for $\tau$ such that $\hat{\tau} = \mathcal{F}$. On the other hand, we will prove general properties of $Q\Psi_{\mathcal{F}}^k$, and they will imply the corresponding statement for $Q\Psi^{k, \mathrm{desc}}_\tau$ in view of Proposition \ref{prop:qadams-equivalence} by setting $\mathcal{F} = \hat{\tau}$.

\subsubsection{Properties}
Now we establish some properties of the quantum cyclic powers.

First, we show that the classical $z \to 0$ limit of the quantum cyclic power operator recovers the operation \eqref{eqn:full-power}.

\begin{prop}\label{prop:cyclic-classical-limit}
    Fix $\mathcal{F} \in K_{\mathbf{T}}(X)$. Then the operator
    \begin{equation}
        Q\Psi_{\mathcal{F}}^k |_{z=0}: K_{\mu_k \times \mathbf{T}}(X)[\![ z^{\mathrm{eff}} ]\!] \to K_{\mu_k \times \mathbf{T}}(X)[\![ z^{\mathrm{eff}} ]\!]
    \end{equation}
    is the $z$-linear extension of the endomorphism
    \begin{equation}
        \otimes P^k(\mathcal{F}): K_{\mu_k \times \mathbf{T}}(X) \to K_{\mu_k \times \mathbf{T}}(X),
    \end{equation}
    where $P^k$ denotes the operation defined in \eqref{eqn:full-power}.
\end{prop}
\begin{proof}
    Note that for $d=0$, the moduli space $\mathsf{QM}_0^{\mu_k}(X)_{\rel  p_1,p_2, p'}$, which consists of constant maps, is isomorphic to $X$. The assertion then follows from checking the formula \eqref{eqn:power}. Indeed, firstly, we know that $\mathbf{G}|_{z=0} = \mathrm{Id}$. On the other hand, the twisted virtual structure sheaf \eqref{eqn:o-twist} for $\mathsf{QM}_0^{\mu_k}(X)_{\rel  p_1,p_2, p'} \cong X$ is given by $K_X^{1/2}$, which means that the pushforward of $\hat{\mathcal{O}}_{\mathrm{vir}} \otimes \mathrm{ev}_k^*(\mathcal{F}^{\boxtimes k}_{\mathrm{eq}})$ as an operator using the twisted convolution operation \eqref{eqn:convolution} is indeed $\otimes P^k(\mathcal{F})$ because $\mathrm{ev}_k^*(\mathcal{F}^{\boxtimes k}_{\mathrm{eq}})$ agrees with the pullback of $\mathcal{F}^{\boxtimes k}_{\mathrm{eq}}$ under the diagonal map $\Delta_k$.
\end{proof}

Note that $X \subset \mathfrak{X}$ as an open substack. The same argument shows that for any $\tau \in K_{\mathbf{T}}(\mathfrak{X})$, the classical limit $Q\Psi^{\mathrm{desc}, k}_\tau |_{z=0}$ is the $z$-linear extension of the operator
\begin{equation}
    \otimes P^k(\tau|_{X}): K_{\mu_k \times \mathbf{T}}(X) \to K_{\mu_k \times \mathbf{T}}(X).
\end{equation}

Next, we consider the nonequivariant limit of the quantum cyclic power in the sense of setting the loop parameter $q=1$.

\begin{prop}\label{prop:nonequivariant}
    Fix $\mathcal{F} \in K_{\mathbf{T}}(X)[\![ z^{\mathrm{eff}} ]\!]$. The operator
    \begin{equation}
        Q\Psi_{\mathcal{F}}^k |_{q=1}: K_{ \mathbf{T}}(X)[\![ z^{\mathrm{eff}} ]\!] \to K_{\mathbf{T}}(X)[\![ z^{\mathrm{eff}} ]\!]
    \end{equation}
    is equal to the $k$-fold PSZ quantum product with $\mathcal{F}$
    \begin{equation}
        (\mathcal{F} \star -)^k: K_{ \mathbf{T}}(X)[\![ z^{\mathrm{eff}} ]\!] \to K_{\mathbf{T}}(X)[\![ z^{\mathrm{eff}} ]\!].
    \end{equation}
\end{prop}
\begin{proof}
    Setting $q=1$ is equivalent to forgetting the $\mu_k$-actions. In that setting, due to the absence of the $\mu_k$-equivariance requirement, we can freely carry out a step-by-step degeneration argument to separate the marked points $p_0', \dots, p_{k-1}'$ into different irreducible components, with incidence insertion given by $\mathbf{G}^{-1}\mathcal{F}$. For each step, we need to insert the gluing matrix to take account of the accordion insertions, see Figure \ref{fig:hand-wave} for an illustration. The final formula follows from the same reasoning as in \cite[Section 6.5]{Oko17}.
    \begin{figure}\label{figure:degeneration}
    \begin{centering}
    \includegraphics[width=130mm]{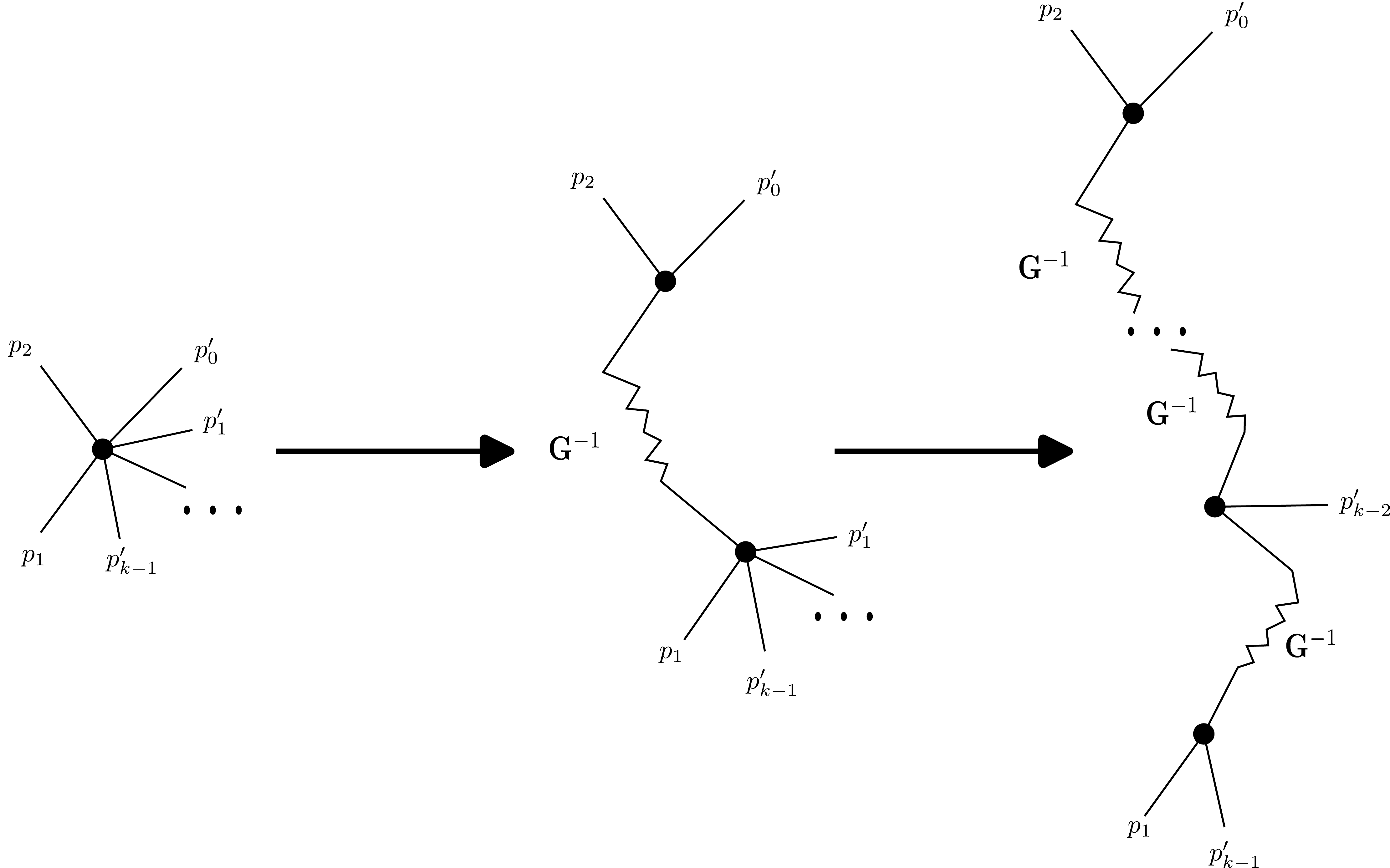}
     \caption{\label{fig:hand-wave}A degeneration argument}
     \end{centering}
     \end{figure}
\end{proof}

\begin{cor}
    The nonequivariant limit of the descendant quantum cyclic $k$-th power $Q\Psi^{\mathrm{desc}, k}_\tau|_{q=1}$ coincides with the $k$-fold iteration of the quantum product with the quantum tautological class $\hat{\tau}(z)$.
\end{cor}
\begin{proof}
    This follows from combining Proposition \ref{prop:qadams-equivalence} and Proposition \ref{prop:nonequivariant}.
\end{proof}

We then investigate the compatibility between quantum cyclic powers and the PSZ quantum product, which can be viewed as a generalization of the quantum Cartan relation for the quantum Steenrod operation as proven in \cite[Proposition 4.8]{seidel-wilkins}.

\begin{prop}
    Let $\mathcal{F}, \mathcal{G} \in K_{\mathbf{T}}(X)[\![ z^{\mathrm{eff}} ]\!]$, then we have
    \begin{equation}\label{eqn:cartan}
        Q\Psi_{\mathcal{F} \star \mathcal{G}}^k = Q\Psi_{\mathcal{F}}^k \circ Q\Psi_{\mathcal{G}}^k.
    \end{equation}
\end{prop}
\begin{proof}
    The proof can be viewed as a $\mu_k$-equivariant analog of the proof of associativity of the quantum product, aka the WDVV relation. 
    
    Recall that the Deligne--Mumford space $\overline{\mathcal{M}}_{0,4} \cong \mathbb{P}^1$, where three of the marked points can be normalized to $\{0,1, \infty\}$ while the fourth marked point $p$ parametrizes $\overline{\mathcal{M}}_{0,4}$, which contains the three distinguished divisors $D_0, D_1, D_{\infty}$ where they represent an reducible genus $0$ curve such that $p$ lies in the irreducible component containing $0, 1, \infty$ respectively. Over the locus $\overline{\mathcal{M}}_{0,4} \setminus (D_0 \cup D_1 \cup D_{\infty})$, we can take the fiberwise genus $0$ branced cover of the universal curve $\overline{\mathcal{C}}_{0,4}$ of degree $k$ ramified at $0$ and $\infty$ of order $k$ so that the preimages of $p$ and $1$ both consist of $k$ points, denoted by $\{p_0, \dots, p_{k-1}\}$ and $\{p_0', \dots, p_{k-1}'\}$ respectively. In fact, such a construction extends to the whole family $\overline{\mathcal{C}}_{0,4} \to \overline{\mathcal{M}}_{0,4}$. The reducible marked curves $\overline{\mathcal{C}}_{0,4}|_{D_1}$ and $\overline{\mathcal{C}}_{0,4}|_{D_0}$ are represented by curves depicted in Figure \ref{fig:hand-wave-2}.
        \begin{figure}
    \begin{centering}
    \includegraphics[width=130mm]{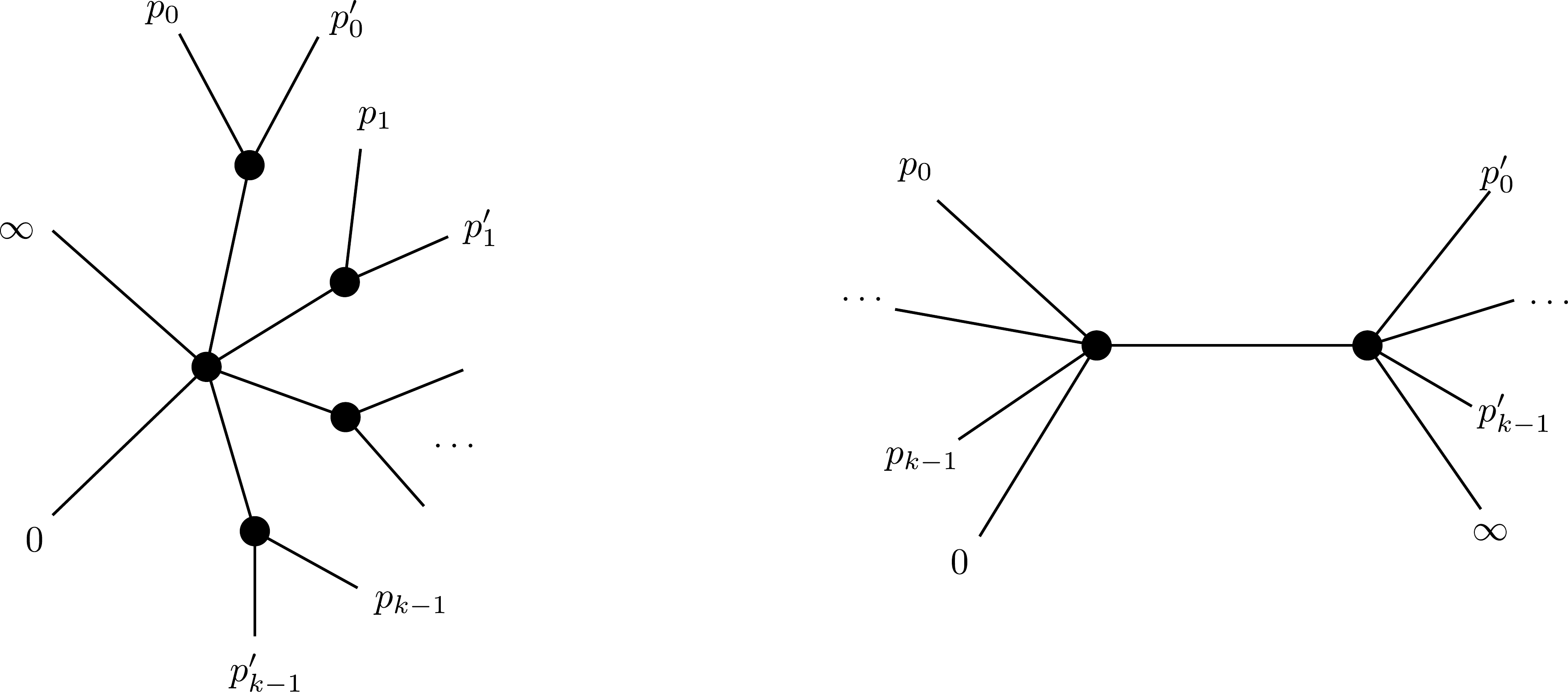}
     \caption{\label{fig:hand-wave-2}Branched covers of the reducible curves $\overline{\mathcal{C}}_{0,4}|_{D_1}$ and $\overline{\mathcal{C}}_{0,4}|_{D_0}$}
     \end{centering}
     \end{figure}
     Here, the Galois group of the covering permutes the irreducible components containing $\{ p_i, p_i' \}_{0 \leq i \leq k-1}$ cyclically for the first curve, while permutes $\{p_0, \dots, p_{k-1}\}$ and $\{p_0', \dots, p_{k-1}'\}$ cyclically for the second curve. We denote the family by $\overline{\mathcal{M}}_{0}^{\mathrm{cyc}, k} \cong \mathbb{P}^1$ with universal curve $\overline{\mathcal{C}}_{0}^{\mathrm{cyc}, k} \to \overline{\mathcal{M}}_{0}^{\mathrm{cyc}, k}$, which has two distinguished divisors $D_0^{\mathrm{cyc}}$ and $D_1^{\mathrm{cyc}}$.

     Taking the fibers of $\overline{\mathcal{C}}_{0}^{\mathrm{cyc}, k} \to \overline{\mathcal{M}}_{0}^{\mathrm{cyc}, k}$ to be the rigid component of quasimaps, we further impose relative degenerations at the marked points and the nodes. By imposing insertions $\mathcal{G}$ at $\{p_0, \dots, p_{k-1}\}$ and $\mathcal{F}$ at $\{p_0', \dots, p_{k-1}'\}$, we see the linear equivalence between $D_0^{\mathrm{cyc}}$ and $D_1^{\mathrm{cyc}}$ implies the Equation \eqref{eqn:cartan}.
\end{proof}

\begin{rem}
    The Koszul sign in the cohomological setting \cite[Proposition 4.8]{seidel-wilkins} does not show up here, as one can see from the fact that the classical limit $z=0$ and the $q = \zeta_k$ specialization recovers the relation $\psi^k(\mathcal{F} \otimes \mathcal{G}) = \psi^k(\mathcal{F}) \otimes \psi^k(\mathcal{G})$.
\end{rem}

\subsubsection{Compatibility with $q$-difference modules}\label{sec:qdiff}
As a final topic of this subsection, we show that the quantum cyclic powers are compatible with the $q$-difference equations in quantum $K$-theory with respect to the K\"ahler (cf. Definition \ref{defn:equiv-shift}) and equivariant variables (cf. Definition \ref{defn:equiv-shift}). These should be considered as $K$-theoretic analogues of the main results of \cite{seidel-wilkins} and \cite{Lee23b}.

We first deal with the covariant constancy of the quantum cyclic powers under the K\"ahler shifts.
\begin{thm}\label{thm:kahler-covariant-constancy}
    The quantum cyclic powers are covariantly constant with respect to the K\"ahler $q$-difference equation, that is, for any $\mathcal{F} \in K_{\mathbf{T}}(X)[\![ z^{\mathrm{eff}} ]\!]$ and $L \in \mathrm{Pic}^G(W)$,
    \begin{equation}
        Q\Psi^k_{\mathcal{F}}(q^L z) \circ M_L^k(z) = M^k_L(z) \circ Q\Psi^k_{\mathcal{F}}(z).
    \end{equation}
\end{thm}
\begin{proof}
    It is equivalent to proving
    \begin{equation}
        Q\Psi^k_{\mathcal{F}}(q^L z) = M^k_L(z) \circ Q\Psi^k_{\mathcal{F}}(z) \circ M_L^k(z)^{-1}.
    \end{equation}
    We first establish this when $\mathcal{F}$ is $z$-independent. Note that for a destabilization $\pi: C \to \mathbb{P}^1$ which is the domain of a quasimap $f: C \to \mathfrak{X}$ in $\mathsf{QM}_d^{\mu_k}(X)_{\rel  p_1,p_2, p'}$, we have
    \begin{equation}
        q^{\langle L,d \rangle} = \det H^{\bullet}(f^*L \otimes \pi^*(\mathcal{O}_{p_1} - \mathcal{O}_{p_2})).
    \end{equation}
    Plug this into \eqref{eqn:power}, we see that
    \begin{equation}
    \begin{aligned}
        &Q\Psi^k_{\mathcal{F}}(q^L z) = \sum_{d} z^d (\mathrm{ev}_{p_1} \times \mathrm{ev}_{p_2})_* \\
        &\left( \mathsf{QM}_d^{\mu_k}(X)_{\rel  p_1,p_2, p'}, \hat{\mathcal{O}}_{\mathrm{vir}} \otimes \mathrm{ev}_k^* ((\mathbf{G}^{-1}\mathcal{F})^{\boxtimes k}_{\mathrm{eq}}) \otimes \det H^{\bullet}(L \otimes \pi^*(\mathcal{O}_{p_1} - \mathcal{O}_{p_2})) \right) \circ \mathbf{G}^{-1},
    \end{aligned}
    \end{equation}
    where we abuse the notation to write $L$ as the pullback of $L$ under the evaluation map over the universal family. Then we simply ($\mu_k$-)equivariantly degenerate the curve $C$ at both $p_1$ and $p_2$
    \begin{equation}
        C \curly C_1 \cup C' \cup C_2,
    \end{equation}
    where $C_1$ and $C_2$ contains $p_1$ and $p_2$ respectively, while the component $C'$ contains the marked points $p_0', \dots, p_{k-1}'$. Then the degeneration formula reads
    \begin{equation}
        (\mathrm{ev}_{p_1} \times \mathrm{ev}_{p_2})_* \big( \mathsf{QM}_{d_1}(X)_{\rel p_1, p_2}, \mathrm{det} H^{\bullet}(L \otimes \pi^*(\mathcal{O}_{p_1})) \big) \big)\circ \mathbf{G}^{-1}
    \end{equation}
    for the $C_1$-component for some $d_1$, producing the connection coefficient for K\"ahler shift, while the minus sign in front of $\mathcal{O}_{p_2}$ means that assembling the contributions from $C_2$-component produces $M_L^k(z)^{-1}$. As for the $C'$-component, that is exactly the configuration which defines the quantum cyclic power operator. 
    
    Note that in general, we allow the class $\mathcal{F}$ to depend on $z$, which means that the shift $z \to q^L z$ also applies to that component when computing $Q\Psi^k_{\mathcal{F}}(q^L z)$. However, due to the $k$-fold insertion and the fact that we have $q^k=1$, we can directly deduce the general formula from the case when $\mathcal{F}$ is $z$-independent. This finishes the proof.
\end{proof}

\begin{cor}
    The descendant quantum cyclic powers are covariantly constant with respect to the K\"ahler $q$-difference connection, that is, for any $\tau \in K_{\mathbf{T}}(\mathfrak{X})$ and $L \in \mathrm{Pic}^G(W)$,
    \begin{equation}
        Q\Psi_\tau^{\mathrm{desc}, k}(q^L z) \circ M_L^k(z) = M^k_L(z) \circ Q\Psi_\tau^{\mathrm{desc}, k}(z).
    \end{equation}
\end{cor}
\begin{proof}
    This follows from Theorem \ref{thm:kahler-covariant-constancy} and Proposition \ref{prop:qadams-equivalence} by setting $\mathcal{F} = \hat{\tau}(z)$.
\end{proof}

We then show that the quantum cyclic power is covariantly constant under the $q$-difference connection defined by the equivariant shifts.

\begin{thm}
    For any $\mathcal{F} \in K_{\mathbf{T}}(X)[\![ z^{\mathrm{eff}} ]\!]$ and cocharacter $\sigma: \mathbb{C}^{\times} \to \mathbf{T}$, we have
    \begin{equation}
        Q\Psi^k_{\mathcal{F}}(q^{\sigma} a) \circ S_{\sigma} = S_{\sigma} \circ Q\Psi^k_{\mathcal{F}}(a),
    \end{equation}
    where $a$ denotes the equivariant variables.
\end{thm}
\begin{proof}
    The proof, as above, uses degeneration for an operator defined through counts of ($\sigma$-)twisted quasimaps. Consider the moduli space $\mathsf{QM}_d^{\sigma}(X)_{\rel p_1, p_2, p'}$, which carries relative evaluation maps
    \begin{equation}
        \mathrm{ev}_{p_1} \times \mathrm{ev}_k \times \mathrm{ev}_{p_2} : \mathsf{QM}_d^{\sigma}(X)_{\rel p_1, p_2, p'} \to E_\sigma \times E_\sigma^{\times k} \times E_\sigma
    \end{equation}
    where $E_\sigma \to \mathbb{P}^1$ is the fibration over $\mathbb{P}^1$ with fibers isomorphic to the stack $\mathfrak{X}$. We fix the equivariant structure so that $\mathbb{C}^\times_q$ acts trivially on the fiber over $p_2$, and acts by $\sigma$ on the fiber over $p_1$. The evaluation map at $p_i$ takes value at the fiber of $E_\sigma \to \mathbb{P}^1$ over $p_i \in \mathbb{P}^1$. We consider the relative compactification, hence the evaluation maps in fact land in the stable locus $X$.  Indeed, $\mathrm{ev}_k$ (considered as a map to $E_\sigma^k$) is $\mu_k$-equivariant with respect to the loop rotation of quasimaps and the cyclic permutation on the targets.
    
    Then one can consider the degeneration of the operator
    \begin{equation}
        \sum_ d z^d (\mathrm{ev}_{p_1} \times \mathrm{ev}_{p_2} )_*  \left(  \hat{\mathcal{O}}_{\mathrm{vir} ,\mathsf{QM}_d^{\sigma} (X)} \cdot   \mathrm{ev}_k^* (\mathbf{G}^{-1}\mathcal{F})_{\mathrm{eq}}^{\boxtimes k} \right) \mathbf{G}^{-1}
    \end{equation}
    in two different ways. First we may degenerate the source curve $C$ at $p_1$ to  \begin{equation}
        C \curly C_1 \cup C'
    \end{equation}
    so that $C_1$ contains $p_1$ and the marked points $p_0', \dots, p'_{k-1}$ and maps to the fiber of $E_\sigma \to \mathbb{P}^1$ over $p_1 \in \mathbb{P}^1$ denoted by $X_1$, while $C'$ carries two marked points at the poles, one identified with $p_2$. The $\mu_k$-action rotates the two components simultaneously.

    The degeneration formula yields the operator
    \begin{equation}
       \sum_{d_1} z^{d_1} (\mathrm{ev}_{p_1} \times \mathrm{ev}_{p'} )_*  \left(  \hat{\mathcal{O}}_{\mathrm{vir} ,\mathsf{QM}_d (X_1)} \cdot   \mathrm{ev}_k^* (\mathbf{G}^{-1}\mathcal{F})_{\mathrm{eq}}^{\boxtimes k} \right) \mathbf{G}^{-1}
    \end{equation}
    for the operator in $C_1$, which is exactly the configuration defining the quantum cyclic $k$th power operation, with the only difference being that $\mu_k$ does not act trivially on $X_1$. Since the identity map $X_1 \cong X$ with $X$ equipped with the trivial $\mu_k$-action is $\mu_k \times \mathbf{T}$-equivariant with respect to the group homomorphism $\mu_k \times \mathbf{T} \to \mu_k \times \mathbf{T}$ given by $(\eta, \theta) \mapsto (\eta, \theta + \sigma(\eta))$ for the cocharacter $\sigma$, this operator can be identified with $Q\Psi^k_\mathcal{F} (q^\sigma a)$.  The operator corresponding to $C'$-component is $S_\sigma$.

    Alternatively we may degenerate the source curve $C$ at $p_2$ to
    \begin{equation}
        C \curly C' \cup C_2
    \end{equation}
    where now $C_2$ contains $p_2$ and the marked point $p_0' , \dots, p'_{k-1}$ and maps to the fiber $X_2$ over $p_2 \in \mathbb{P}^1$. In this case, the operator corresponding to $C_2$ is exactly $Q\Psi^k_\mathcal{F}$, since the $\mu_k$-action on $X_2$ is already trivial. The operator corresponding to $C'$-component is $S_\sigma$.

    By comparing the two degenerations of the same operator, we arrive at the desired result.

\end{proof}
The following corollary is evident.

\begin{cor}
    The descendant quantum cyclic powers are covariantly constant with respect to the K\"ahler $q$-difference connection, i.e., for any $\tau \in K_{\mathbf{T}}(\mathfrak{X})$ and cocharacter $\sigma: \mathbb{C}^{\times} \to \mathbf{T}$,
    \begin{equation}
        Q\Psi_\tau^{\mathrm{desc}, k}(q^{\sigma} a) \circ S_{\sigma} = S_{\sigma} \circ Q\Psi_\tau^{\mathrm{desc}, k}(a).
    \end{equation}
\end{cor}

\subsection{Quantum Adams operations}\label{ssec:qAdams}
Equations \eqref{eqn:descendant-power} and \eqref{eqn:power} should be understood as the ``total power operation."  To obtain the corresponding Adams operation, we must pass to a further quotient by the ideal generated by $\Phi_k(q) \in K_{\mu_k}(\mathrm{pt})$, the $k$th cyclotomic polynomial, which is parallel to the discussion in Section \ref{sec:ordinary-adams}. This is the point of view taken in, e.g., \cite{atiyah-power} and, in particular for inspiring our research, \cite[Section 4]{lonergan}. Introduce the notation
\begin{equation}
    K_{\mu_k \times \mathbf{T}}(X) |_{q = \zeta} = K_{\mu_k \times \mathbf{T}}(X) / (\Phi_k(q)).
\end{equation}

\begin{defn}
    The \emph{quantum Adams operator} is defined as the induced operator in the quotient
    \begin{equation}
        Q\psi_{\mathcal{F}}^k := Q\Psi_{\mathcal{F}}^k |_{q = \zeta} \in \left( K_{\mu_k \times \mathbf{T}}(X)|_{q = \zeta} \right)^{\otimes 2} [\![z ^{\mathrm{eff}} ]\!].
    \end{equation}
    Similarly, the \emph{descendant quantum Adams operator} is defined as the induced operator in the quotient 
    \begin{equation}
        Q\psi_{\tau}^{\mathrm{desc},k} := Q\Psi_{\tau}^{\mathrm{desc},k} |_{q = \zeta} \in \left( K_{\mu_k \times \mathbf{T}}(X)|_{q = \zeta} \right)^{\otimes 2} [\![z ^{\mathrm{eff}} ]\!].
    \end{equation}
\end{defn}

\begin{rem}
    Passing to the quotient by the ideal $\Phi_k(q)$ is often described in more concrete terms (as in \cite{koroteev-smirnov}) as the $q = \zeta$ specialization, where $\zeta = e^{2\pi i/k}$ is a (preferred choice of) a $k$th roots of unity.
\end{rem}

Specializing $q=\zeta$ for all the results in Section \ref{ssec:oper-definition}, we obtain the following statements.

\begin{prop}\label{prop:adams-property}
     The quantum Adams operator satisfy the following properties.
    \begin{enumerate}
        \item For any $\tau \in K_{\mathbf{T}}(\mathfrak{X})$, we have $Q\psi_{\tau}^{\mathrm{desc},k} = Q\psi_{\hat{\tau}}^k$.
        \item Given $\mathcal{F} \in K_{\mathbf{T}}(X)$, the operator $Q\psi_{\mathcal{F}}^k|_{z=0}$ agrees with the $z$-linear extension of the endomorphism $\otimes \psi^k(\mathcal{F})$. Similarly, for any $\tau \in K_{\mathbf{T}}(\mathfrak{X})$, the operator $Q\psi_{\tau}^{\mathrm{desc},k}|_{z=0}$ coincides with $z$-linear extension of the endomorphism $\otimes \psi^k(\tau |_{X})$.
        \item Let $\mathcal{F}, \mathcal{G} \in K_{\mathbf{T}}(X)[\![ z^{\mathrm{eff}} ]\!]$, then we have
        \begin{equation}
        Q\psi^k_{\mathcal{F} \star \mathcal{G}} = Q\psi^k_{\mathcal{F}} \circ Q\psi^k_{\mathcal{G}}.
        \end{equation}
        \item Given $\mathcal{F} \in K_{\mathbf{T}}(X)[\![ z^{\mathrm{eff}} ]\!]$ and $L \in \mathrm{Pic}^G(W)$, we have
        \begin{equation}
            Q\psi^k_{\mathcal{F}}(\zeta^L z) \circ M_L^k(z)|_{q=\zeta} = M^k_L(z) |_{q=\zeta} \circ Q\psi^k_{\mathcal{F}}(z).
        \end{equation}
        Similarly, for $\tau \in K_{\mathbf{T}}(\mathfrak{X})$, we have
        \begin{equation}
            Q\psi_{\tau}^{\mathrm{desc},k}(\zeta^L z) \circ M_L^k(z)|_{q=\zeta} = M_L^k(z)|_{q=\zeta} \circ Q\psi_{\tau}^{\mathrm{desc},k}(z).
        \end{equation}
        \item Given $\mathcal{F} \in K_{\mathbf{T}}(X)[\![ z^{\mathrm{eff}} ]\!]$ and a cocharacter $\sigma: \mathbb{C}^{\times} \to \mathbf{T}$, we have
        \begin{equation}
            Q\psi^k_{\mathcal{F}}(\zeta^{\sigma} a) \circ S_{\sigma}|_{q=\zeta} = S_{\sigma}|_{q=\zeta} \circ Q\psi^k_{\mathcal{F}}(a).
        \end{equation}
        Similarly, for $\tau \in K_{\mathbf{T}}(\mathfrak{X})$, we have
        \begin{equation}
            Q\psi^{\mathrm{desc},k}_{\tau}(\zeta^{\sigma} a) \circ S_{\sigma}|_{q=\zeta} = S_{\sigma}|_{q=\zeta} \circ Q\psi^{\mathrm{desc},k}_{\tau}(a).
        \end{equation}
        \qed
    \end{enumerate}
\end{prop}

We would like to point out that there is no analog of Proposition \eqref{prop:nonequivariant} for the quantum Adams operators because $q$ has been specialized to a primitive root of unity. However, the quantum Adams operators satisfy the additivity relation not satisfied by the cyclic powers.

\begin{lemma}\label{lem:qadams-additivity}
    Given $\mathcal{F}, \mathcal{G} \in K_{\mathbf{T}}(X)[\![ z^{\mathrm{eff}} ]\!]$, we have
    \begin{equation}
        Q\psi_{\mathcal{F}+\mathcal{G}}^k = Q\psi_{\mathcal{F}}^k + Q\psi_{\mathcal{G}}^k.
    \end{equation}
\end{lemma}
\begin{proof}
    The key observation is
    \begin{equation}\label{eqn:additivity}
        (\mathbf{G}^{-1}(\mathcal{F} + \mathcal{G}))^{\boxtimes k}_{\mathrm{eq}} = (\mathbf{G}^{-1}\mathcal{F})^{\boxtimes k}_{\mathrm{eq}} + (\mathbf{G}^{-1}\mathcal{G})^{\boxtimes k}_{\mathrm{eq}} + \text{ induction of classes in } K_{\mu_k' \ltimes \mathbf{T}^k}(X^k)
    \end{equation}
    where $\mu_{k'} < \mu_k$ denotes a proper subgroup. Once we specialize $q = \zeta$, a primitive root of unity, the classes from induction on the RHS of \eqref{eqn:additivity} vanishes. To be more precise, the projection 
    \begin{equation}
        \mathbb{Z}[q, q^{-1}] / (q^k - 1) \to \mathbb{Z}[q, q^{-1}] / \Phi_k(q)
    \end{equation}
    has kernel $J$ generated by $\sum_{j=1}^d q^{jk/d}$ for all $d$ dividing $k$. By grouping $(\mathbf{G}^{-1}(\mathcal{F} + \mathcal{G}))^{\boxtimes k}_{\mathrm{eq}}$ into orbits under the ${\mb Z}/k$-action, except for $(\mathbf{G}^{-1}\mathcal{F})^{\boxtimes k}_{\mathrm{eq}}$ and $(\mathbf{G}^{-1}\mathcal{G})^{\boxtimes k}_{\mathrm{eq}}$, the rest lies in $J$. We then plug the computation into the formula \eqref{eqn:power} to conclude. 
\end{proof}

\begin{rem}
    We do not know if there is a quantum counterpart of the relation $\psi^k \circ \psi^l = \psi^{kl} = \psi^{l} \circ \psi^{k}$.
\end{rem}

\section{Example: $T^*\mathbb{P}^{n}$}\label{ssec:oper-example}
In this section we exhibit the nontriviality of the theory by working out the computation for the target $X = T^*\mathbb{P}^{n-1}$ using $\mathbb{Z}/k$-equivariant localization in $K$-theory. The key input is the stable envelope construction of Aganagic--Okounkov \cite{aganagic-okounkov-bethe}, which allows one to replace relative insertions with descendant insertions.

For computation we will use \emph{equivariant localization} with respect to $\mu_k$. The usual $\mathbb{G}_{\mathrm{m}}$-localization involves inverting all forms of the form $q^\mu - 1$ for $\mu \in \mathbb{Z} - \{0\}$; more generally, for a torus $T$, localization involves inverting $q^\mu =1$ for all characters $0 \neq \mu \in X^\bullet(T)= \mathrm{Hom}(T, \mathbb{G}_{\mathrm{m}})$. Here, for $\mu_k$-localization, we only invert the forms $q^\mu - 1$ where $1 \le \mu \le k-1$. We denote the localized ring with the subscript $loc$, i.e.
\begin{equation}
    K_{\mu_k}(\mathrm{pt})_{\mathrm{loc}} := K_{\mu_k}(\mathrm{pt}) [ (1-q^\mu)^{-1} : 1 \le \mu \le k-1].
\end{equation}

\subsection{Quasimap moduli for $X = T^*{\mathbb{P}^{n-1}}$}

By the main theorem of \cite{aganagic-okounkov-bethe}, operations involving relative insertions such as quantum Adams operators can be replaced by operations only involving descendant insertions, by considering the structure constants in the \emph{stable basis} when the target $X$ is a quiver variety. We first describe the descendant quasimap moduli space (with no insertions) for $X = T^*\mathbb{P}^{n-1}$ using its presentation as a Higgs branch (GIT quotient of a cotangent bundle of the quiver representation).

\begin{rem}
    In general, the moduli space of quasimaps to $X$ depends on the presentation of the target $X$ as a GIT quotient $X = W/\!\!/ G$ of an affine variety $W$ by a reductive group $G$ (more precisely, it depends on the pair $[W/G] \supseteq [W^{\mathrm{ss}}/G]$ where $W^{\mathrm{ss}} \subseteq W$ denotes the semistable locus, \cite[Proposition 4.6.1]{quasi-map}). For the example of $X = T^*{\mathbb{P}^{n-1}}$, its presentation as a hypertoric variety and as a Higgs branch of a quiver gauge theory in fact agree.
\end{rem}

Consider the quiver with one vertex with one framing satisfying $\mathrm{rank}(V) = k$ and $\mathrm{rank}(W) = n$. On $T^* \mathrm{Hom}(V, W) = \{(I, J) : I \in \mathrm{Hom}(V, W), J \in \mathrm{Hom}(W, V)\}$ one has a natural moment map 
\begin{equation}
\mu = J \circ I : T^*{\mathrm{Hom}(V,W)} \to  \mathfrak{gl}(k)^*. 
\end{equation}
For the determinant character $\theta : g \mapsto \det(g)$, $\theta$-semistability is equivalent to the condition that $J$ is surjective. By taking the GIT quotient of $\mu^{-1}(0) \subseteq T^*\mathrm{Hom}(V,W)$ for $\theta$ one obtains the Higgs branch $X = T^*\mathrm{Gr}(k,n)$.

We will be focused on the case of $k=1, n \ge2$. Consider quasimaps from $C = \mathbb{P}^1$ to $X = T^*\mathrm{Gr}(1,n) = T^*\mathbb{P}^{n-1}$. Degrees of maps are indexed by integers $\mathbb{Z} \cong H_2(X;\mathbb{Z})$. The framing $W$ determines a tautological bundle $\mathcal{O}^{\oplus n}$ over $C$. The vertex $V$ is of weight $1$ with respect to the $\mathrm{GL}(1)$-action, and defines a tautological bundle $\mathcal{O}(d)$ over $C$ for a degree $d$ map. Quasimap from the rigid domain $C = \mathbb{P}^1$ into $ X \subseteq [ \mu^{-1}(0) / \mathrm{GL}(1) ]$ is the data of a choice of $\mathcal{V}$ and a section
\begin{equation}
    f \in H^0 \left( C, \mathcal{H}om(\mathcal{W}, \mathcal{V}) \oplus \mathcal{H}om(\mathcal{V}, \mathcal{W} \otimes \hbar^{-1})  \right) = H^0\left(C, \mathcal{O}(d)^{\oplus n} \oplus \mathcal{O}(-d)^{\oplus n} \otimes \hbar^{-1} \right)
\end{equation}
satisfying the moment map condition. For $d \neq 0$, the moment map condition is automatically satisfied, and by the stability condition, the quasimap moduli space is empty for $d < 0$. The stability condition amounts to the generic non-vanishing of the nonzero section, hence the quasimap moduli space is identified with
\begin{equation}
    \mathsf{QM}_d(X) := \mathbb{P} H^0 \left( \mathbb{P}^1; \mathcal{O}(d)^{\oplus n} \right).
\end{equation}
We omit the superscript $\mu_k$ for the quasimap moduli spaces even if we work equivariantly under $\mu_k$. If one considers $W$ equivariantly with respect to the action of $\mathbb{G}_{\mathrm{m}}^{\operatorname{rank} W} = \mathbb{G}_{\mathrm{m}}^n$ with equivariant parameters $a_1, \dots, a_n$, then the quasimap moduli space is given as
\begin{equation}
    \mathsf{QM}_d(X) := \mathbb{P} H^0 \left( \mathbb{P}^1; \bigoplus_{i=1}^n a_i^{-1} \mathcal{O}(d) \right);
\end{equation}
indeed one may let $a_i =1$ to recover the non-equivariant computation.

\subsection{Ingredients of quantum Adams operations}
For $X = T^*{\mathbb{P}^{n-1}}$, take the moduli space of stable quasimaps of degree $d$, denoted $\mathsf{QM}_d(X)$. We would like to apply localization for the $\mu_k \times \mathbf{T}$-action on $\mathsf{QM}_d(X)$ to compute the quantum Adams operators. 

In the fixed-point basis of localized $K$-theory, the structure constants of the operator with relative insertions are given by operator with descendant insertions by stable envelope classes \cite{aganagic-okounkov-bethe}. To compute the structure constants, one must therefore compute the contribution from (i) the virtual structure sheaf, (ii) the equivariant insertions at equators, and (iii) the stable class descendant insertions at the input and output marked points.

\subsubsection{Virtual structure sheaf}
First, we compute the virtual structure sheaf.
To compute
\begin{equation}
    \hat{\mathcal{O}}_{\mathrm{vir}} := \mathcal{O}_{\mathrm{vir}} \otimes \left(\mathcal{K}_{\mathrm{vir}} \frac{\det f^* (T^{1/2}X)|_{p_2}}{\det f^* (T^{1/2}X)|_{p_1}}\right)^{\frac12},
\end{equation}
where we refer the reader to \eqref{eqn:o-twist} for the explanation of notations, we first compute $\mathcal{O}_{\mathrm{vir}} = \Lambda^\bullet \mathrm{Obs}^\vee$; here $\mathrm{Obs}$ is the obstruction part to the deformation theory given by
\begin{equation}
    T_{\mathrm{vir}} = H^\bullet(\mathcal{M} \oplus \hbar^{-1}\mathcal{M}^\vee) - \left(1+\hbar^{-1}\right) \sum \mathrm{Ext}^\bullet (\mathcal{V}_i, \mathcal{V}_i) =  \mathrm{Def} - \mathrm{Obs},
\end{equation}
following the notations in \cite[Section 4.3.16]{Oko17}, where $\mathcal{M}$ denotes the pullback of the (undoubled) quiver representation datum under the evaluation map. The obstruction part arises as
\begin{equation}
  \mathrm{Obs} =  H^1(\mathcal{M} \oplus \hbar^{-1}\mathcal{M}^\vee) + \hbar^{-1} \sum \mathrm{Hom}(\mathcal{V}_i, \mathcal{V}_i);
\end{equation}
the first term is the obstruction to deforming the section and the second term imposes the moment map constraint. The other terms in $T_\mathrm{vir}$ correspond to deformations of the quasimap section and deformations of $\mathcal{V}_i$'s respectively.

Specializing to degree $d$ maps to $T^*\mathbb{P}^{n-1}$, the obstruction part is written equivariantly as
\begin{equation}
    \mathrm{Obs}_d := \hbar^{-1}H^1\left( \bigoplus_{i=1}^n a_i\mathcal{O}(-d) \right) + \hbar^{-1} H^0(\mathcal{O}).
\end{equation}
Our linearization is so that $H^0(\mathcal{O}(d)) \cong \sum_{j=0}^d q^d$ as weights. By Serre duality and the isomorphism $K_C \cong \mathcal{O}(-2) \otimes \mathbb{C}_q$, we see that  $H^1(a_i \mathcal{O}(-d))^\vee \cong H^0(a_i^{-1} \mathcal{O}(d) \otimes K_C) = a_i^{-1} \sum_{j=1}^{d-1} q^j$, from which one can compute that
\begin{equation}\label{eqn:ovir-computation}
    \mathcal{O}_{\mathrm{vir}} = \Lambda^\bullet \mathrm{Obs}^\vee = \prod_{w \in \mathrm{wts}(\mathrm{Obs}^\vee)} (1-w)= (1-\hbar) \prod_{1 \le i \le n} \prod_{ 1 \le j \le d-1} (1-\hbar a_i^{-1}q^j).
\end{equation}

Similar computation for the whole virtual tangent bundle

\begin{align}
    T_{\mathrm{vir}} &= \mathrm{Def} - \mathrm{Obs} \\
    &= \bigoplus_{i=1}^n H^0 \left( a_i^{-1} \mathcal{O}(d) \right)  - H^0(\mathcal{O}) - \hbar^{-1} \bigoplus_{i=1}^n H^1(a_i \mathcal{O}(-d)) - \hbar^{-1} H^0(\mathcal{O}) \\
    &= \sum_{i=1}^n a_i^{-1} \sum_{j=0}^d q^j -1 - \hbar^{-1} \sum_{i=1}^n a_i \sum_{j=1}^{d-1} q^{-j} - \hbar^{-1}
\end{align}
yields

\begin{align}
    \mathcal{K}_{\mathrm{vir}} = \Lambda^{top} T^\vee_{\mathrm{vir}} &= \prod_{w \in \mathrm{wts}(T^\vee_{\mathrm{vir}})} w \\ &= \prod_{i=1}^n \prod_{j=0}^d a_i q^{-j}  \cdot \left(  \prod_{i=1}^n \prod_{j=1}^{d-1} (\hbar a_i^{-1} q^j) \cdot 1 \cdot \hbar \right)^{-1} 
    \\
    &= \hbar^{-1-n(d-1)} \cdot \prod_{i=1}^n a_i^{2d} \cdot q^{-nd^2}.
\end{align}

If we linearize the action of $(\mathbb{G}_{\mathrm{m}})_q$ on $\mathbb{P}^1$ such that $\mathcal{O}(d)|_{p_2} = q^d \cdot \mathcal{O}(d)|_{p_1}$, so that
\begin{equation}
    f^*(T^{1/2}X)|_{p_1} = \sum_{i=1}^n a_i^{-1} + 1 , \quad f^*(T^{1/2}X)|_{p_2} = \sum_{i=1}^n a_i^{-1}q^d  + 1,
\end{equation}
we have
\begin{equation}\label{eqn:ovir-twist-computation}
    \widehat{\mathcal{O}}_{\mathrm{vir}} =  \mathcal{O}_{\mathrm{vir}} \left( \mathcal{K}_{\mathrm{vir}} \frac{\det f^*(T^{1/2}X )|_{p_2}}{ \det f^*(T^{1/2} X)|_{p_1}} \right)^{1/2} = \mathcal{O}_{\mathrm{vir}} \cdot \prod_{i=1}^n a_i^{d} \cdot \frac{\hbar^{(n(d-1)+1))/2}}{q^{nd(d-1)/2}}.
\end{equation}

\subsubsection{Equator insertions}

As explained in the previous section, one may consider different versions of the quantum Adams operator depending on the type of the equator insertions (relative/descendant). The relationship between the two operators are given by \cref{prop:qadams-equivalence} and \cite{aganagic-okounkov-bethe}. Namely, up to computing the gluing matrix $\mathbf{G}$ and the abelianized stable envelope classes of \cite{aganagic-okounkov-bethe}, it is equivalent to consider either relative or descendant insertions along the equator.

Hence, to simplify the example, we consider the case of descendant quantum Adams operator $Q\psi^{\mathrm{desc}, k}_V$ for
\begin{equation}
    [V] \in K_{\mathbf{T}} \left( \mathfrak{X} \right) \cong K_{\mathbf{T} \times \mathbb{G}_{\mathrm{m}}} (\mathrm{pt}) \cong \mathbb{Z} [a_{i}^\pm , \hbar^{\pm}, s^{\pm}].  
\end{equation}
The insertions are given by pullback along evaluation maps. A choice of a point $p \in \mathbb{P}^1$ defines an evaluation map
\begin{equation}
    \mathrm{ev}_p : \mathsf{QM}_d(X) \to \mathfrak{X} = \left[\mu^{-1}(0) / \mathbb{G}_{\mathrm{m}} \right]
\end{equation}
inducing the pullback
\begin{equation}
    \mathrm{ev}^*_p : K_{\mathbf{T} \times \mathbb{G}_{\mathrm{m}}} (\mathrm{pt}) \to K_{\mathbf{T}}(\mathsf{QM}_d(X)).
\end{equation}
The map $\mathrm{ev}_p^*$ is linear over $K_{\mathbf{T}}(\mathrm{pt})$ and satisfies (cf. \cite[Lemma 3.21]{Lee23a})
\begin{equation}
    \mathrm{ev}^*_p (s) = [\mathcal{O}_{\mathsf{QM}_d(X)}(1)] \in K_{\mathbf{T}}(\mathsf{QM}_d(X))
\end{equation}
for the generator $s \in K_{\mathbb{G}_{\mathrm{m}}}(\mathrm{pt})$. 

For the distinguished collection of marked points $p_0' = 1, p_1' = \zeta, \dots, p_{k-1} = \zeta^{k-1}$, we have correspondingly the $\mu_k$-equivariant evaluation map
\begin{equation}
    \mathrm{ev}_k: \mathsf{QM}_d(X) \to \mathfrak{X}^k,
\end{equation}
inducing
\begin{equation}
    (\mathrm{ev}_k)^*: K_{\mu_k \ltimes \mathbf{T}^k} (\mathfrak{X}^k)  \to K_{\mu_k \times \mathbf{T}} \left(\mathsf{QM}_d(X)\right).
\end{equation}
By the additivity property of $Q\psi_V^{\mathrm{desc},k}$ (\cref{lem:qadams-additivity}) and linearity over $K_\mathbf{T}(\mathrm{pt})$, it suffices to consider the image of
\begin{equation}
    s \in K_\mathbf{T \times \mathbb{G}_{\mathrm{m}}} (\mathrm{pt}) \to s_{\mathrm{eq}}^{\boxtimes k} = s_1 \cdots s_k\in K_{\mu_k \ltimes (\mathbf{T} \times \mathbb{G}_{\mathrm{m}})^k} (\mathrm{pt})
\end{equation}
under this evaluation map, corresponding to the equivariant descendant insertion of the tautological bundle. The image under the equivariant pullback is given by
\begin{equation}
    \mathrm{ev}_k^* (s_1 \cdots s_k) = \bigotimes_{j=1}^k [\mathcal{O}_{\mathsf{QM}_d(X)}(1)] = [\mathcal{O}_{\mathsf{QM}_d(X)}(k)]  \in K_{\mu_k \times \mathbf{T}}(\mathsf{QM}_d(X))
\end{equation}
since the $\mu_k$-action on the rank 1 vector bundle $\bigotimes_{j=1}^k [\mathcal{O}_{\mathsf{QM}_d(X)}(1)] $ given by cyclic permutation is trivial.

\begin{rem}
    As the map $s \mapsto s_{\mathrm{eq}}^{\boxtimes k}$ fails to be linear, it will not be the case that the image of linear combinations of the generators under the pullback map $\mathrm{ev}_k^*$ will generally be $q$-independent. Consequently, the (quantum) cyclic power operation $Q\Psi^{\mathrm{desc}, k}_{V}$ is not determined by $Q\Psi^{\mathrm{desc},k}_L$ for rank $1$ bundles $L$. Hence it is the additivity property (\cref{lem:qadams-additivity}) which greatly simplifies the computation of $Q\psi^{\mathrm{desc},k}_V$.
\end{rem}


\subsubsection{Input and output insertions}

The input and output insertions use the relative compactification of the quasimap moduli space. One can bypass the usage of relative quasimaps by the following celebrated result of Aganagic--Okounkov:

\begin{thm}[{\cite[Theorem 1]{aganagic-okounkov-bethe}}]
    The operator $K_{\mathbf{T}}(X) \to K_{\mathbf{T} \times \mathbb{G}_{\mathrm{m}}}(X)[\![z]\!]$ given by
    \begin{equation}
        \alpha \mapsto (\mathrm{ev}_{p_2}^{\mathrm{rel}})_* \left( (\mathrm{ev}_{p_1}^{\mathrm{desc}})^* \mathbf{f}_\alpha \otimes \hat{\mathcal{O}}_{\mathrm{vir}} z^{\deg} \right) 
    \end{equation}
    for the \emph{stable envelope classes} $\mathbf{f}_\alpha$ is equal to the gluing matrix $\mathbf{G}$. Here, $\mathrm{ev}^{\mathrm{rel}}$ and $\mathrm{ev}^{\mathrm{desc}}$ denote the relative and descendant evaluation maps, respectively.
\end{thm}

The proof requires $\mathbb{C}^{\times}_q$-localization, but the resulting operator $\mathbf{G}$ is $q$-independent, see \cite[Section 7.1]{Oko17}. The result therefore remains valid upon passing to $K_{\mathbf{T} \times \mu_k}(X) [\![z]\!]$, i.e. by setting $q^k=1$, and we may replace relative insertions with the corresponding descendant insertions by the stable envelope classes $\mathbf{f}_\alpha$.

For the computation of the stable envelope classes $\mathbf{f}_\alpha$ in the case of affine type $A$ quiver varieties, we refer the reader to \cite[Section 3.2]{aganagic-okounkov-bethe} for details. Here we record the results of the computation following \cite{aganagic-okounkov-bethe}. 

The stable envelopes depend on the choice of a polarization $T^{1/2}X \in K_{\mathbf{T}}(\mathfrak{X})$ and a $1$-parameter subgroup $\sigma: \mathbb{G}_{\mathrm{m}} \to T$.  We fix the polarization to be
\begin{equation}
    T^{1/2} X = \mathrm{Hom}(W, V) - \mathrm{Hom}(V,V) =  \sum_{i=1}^n a_i^{-1}x - 1
\end{equation}
where $x$ denotes the Chern root of $V$. The choice of a (generic) $1$-parameter subgroup in $T = (\mathbb{C}^\times)^n$ gives a decomposition of weights into attracting and repelling directions; let us fix the cocharacter $\sigma$ such that $a_i/a_j$ is attracting if and only if $i < j$.

In our case the gauge group $G = \mathbb{C}^\times$ is abelian, and the relevant formulas are \cite[Equation (71, 79)]{aganagic-okounkov-bethe}. There are $n$ $\mathbf{T}$-fixed points of $X$, denoted $F_1, \dots, F_n$, such that $x \mapsto a_i$ under the restriction map $K_\mathbf{T}(\mathfrak{X}) \to K_\mathbf{T}(F_i)$. Our choice of the convention implies
\begin{equation}
    \mathbf{f}_i = \Lambda^{\diamond} \left(T^{1/2}_{<0} + \hbar T^{1/2}_{>0} \right) = \Lambda^\diamond
 \left( \sum_{j<i} \frac{a_i}{a_j} +  \hbar \sum_{j >i} \frac{a_i}{a_j} \right) = \prod_{j<i} (a_j - a_i) \prod_{j>i} (a_j - \hbar a_i) \end{equation}
as the stable envelope classes associated to the fixed point $F_i$, generating $K_\mathbf{T}(\mathfrak{X})_{\mathrm{loc}}$.

\subsection{Localization on quasimap moduli space}
As discussed in the previous subsection, the quasimap moduli space of degree $d$ is a projective space on
\begin{equation}
    H^0\left( \bigoplus_{i=1}^n a_i^{-1}\mathcal{O}(d)\right) = \sum_{j=0}^d \sum_{i=1}^n a_i^{-1} q^j = \sum_{\ell=0}^{k-1} \left( \left\lfloor \frac{d-\ell}{k} \right\rfloor +1 \right) \sum_{i=1}^n a_i^{-1} q^\ell,
\end{equation}
where the latter is the isotypic decomposition into weights in $K_{\mu_k \times \mathbf{T}}(\mathrm{pt})$ (in particular, $q^k = 1$). By the isotypic decomposition, one sees that the $\mu_k$-fixed strata $\mathsf{F}_\ell$ in $\mathsf{QM}_d(X) = \mathbb{P}H^0( \bigoplus a_i^{-1} \mathcal{O}(d))$ are exactly given by the projective spaces on the each of the $k$ summands in the last expression, for $\ell = 0, \dots, k-1$:
\begin{equation}
    \mathsf{F}_{\ell} := \mathbb{P} \left( \left( \left\lfloor \frac{d-\ell}{k} \right\rfloor +1 \right)   \sum_{i=1}^n a_i^{-1} q^\ell \right).
\end{equation}
The $\ell$-th fixed strata have normal bundle (moving part of the deformation space) given by the complement of the vector
\begin{equation}
    N_{\mathsf{F}_{\ell}} = \sum_{m \neq \ell} \left( \left\lfloor \frac{d-m}{k} \right\rfloor +1 \right)  \sum_{i=1}^n a_i^{-1} q^{m-\ell}.
\end{equation}
These computations suffice to apply (virtual) localization. Our goal is to apply localization to compute the pushforward of the class
\begin{equation}
    \hat{\mathcal{O}}_{\mathrm{vir}} \cdot \mathrm{ev}_{p_1}^* \mathbf{f}_{i} \cdot \mathrm{ev}_k^* \tau^{\boxtimes k}_{\mathrm{eq}}  \cdot \mathrm{ev}^*_{p_2} \mathbf{f}_j  \in K_{\mu_k \times \mathbf{T}} (\mathsf{QM}_d(X))
\end{equation}
to $K_{\mu_k \times \mathbf{T}}(\mathrm{pt})$, in order to determine the $(i,j)$th matrix coefficient of $z^d$ in $Q\Psi_{\tau}^{\mathrm{desc}, k} (z)$. The localization formula allows us to reduce the desired computation of the right vertical arrow (pushforward) in the diagram
\begin{equation}
    \begin{tikzcd}
        \bigoplus_{\ell=0}^{k-1} K_{\mu_k \times \mathbf{T}} (\mathsf{F}_\ell) \arrow[rr, yshift=-2] \arrow[rd] & & K_{\mu_k \times \mathbf{T}} (\mathsf{QM}_d(X)) \arrow[ll,yshift=2, " \sum_{\ell=0}^{k-1} (-)|_{\mathsf{F}_\ell}"'] \arrow[ld] \\ & K_{\mu_k \times \mathbf{T}} (\mathrm{pt})&
    \end{tikzcd}
\end{equation}
to the simpler left vertical arrows (pushforward from the fixed loci) via the identity
\begin{equation}
    \chi (\phi) = \sum_{\ell=0}^{k-1} \frac{ \chi (\phi|_{\mathsf{F}_\ell})}{\Lambda^\bullet (N^\vee_{\mathsf{F}_\ell})} \in K_{\mu_k \times \mathbf{T}}(\mathrm{pt})_{\mathrm{loc}}.
\end{equation}
In our case, the relevant computation for the correction (Euler class) $\Lambda^\bullet(N^\vee_{\mathsf{F}_\ell})$ is given by
\begin{equation}
     \frac{1}{\Lambda^\bullet (N_{\mathsf{F}_\ell}^\vee) } = \prod_{m \neq \ell} \prod_{i=1}^n \frac{1}{( 1- a_i q^{\ell-m}) ^{\left\lfloor \frac{d-m}{k } \right\rfloor +1}}.
\end{equation}
Here the condition $m \neq \ell$ ensures that the denominators are well-defined in $K_{\mu_k \times \mathbf{T}}(\mathrm{pt})_{\mathrm{loc}}$. This reduces any computation of the structure constants of the quantum Adams operators to an algorithmically solvable problem whose complexity is bounded in terms of $k$ and $d$.

\begin{example}[{cf. \cite[Section 6.3]{Lee23a}}]
    To illustrate our strategy for computation, let us consider the case of $X = T^*\mathbb{P}^1$ with the $T$-equivariant parameters specialized to  $a_1, a_2 = 1$. Then the degrees are indexed by $d \in \mathbb{Z}_{\ge 0}$; for convenience write $d = \alpha k + \beta$ for $\alpha \ge 0$ and $0 \le \beta \le k-1$. Then the quasimap space is given by
    \begin{equation}
        \mathsf{QM}_d(X) = \mathbb{P} H^0 (\mathcal{O}(d)^{\oplus 2}) \cong \mathbb{P}^{2d+1},
    \end{equation}
    with fixed strata given by
    \begin{equation}
        \mathsf{F}_{d, \ell} := \mathbb{P}^{2 \left\lfloor (d-\ell)/k \right\rfloor +1} = \begin{cases} \mathbb{P}^{2\alpha + 1}  & \mbox{ if } \ell \le \beta \\ \mathbb{P}^{2\alpha - 1} & \mbox { if } \ell > \beta \end{cases}.
    \end{equation}
    The obstruction bundle correction from the normal bundle
    \begin{equation}
        N_{\mathsf{F}_\ell} = \bigoplus_{0 \le m \le k-1, m \neq \ell}  \left( 2 \left\lfloor \frac{d-m}{k} \right\rfloor +2 \right) q^{m-\ell  }
    \end{equation}
    is given by
    \begin{equation}
        \frac{1}{\Lambda^\bullet ( N_{\mathsf{F}_\ell}^\vee)} = \prod_{m \neq \ell} \frac{1}{(1- q^{\ell-m})^{2 \left\lfloor \frac{d-m}{k} \right\rfloor + 2}} \in \mathbb{Z}[q, (1-q^\ell)^{-1} : \ell \neq k ]/(q^k=1).
    \end{equation}
    The insertions at the input and output marked points can be either $\mathbf{f}_1 = 1-x$ or $\mathbf{f}_2  = 1-\hbar x$ provided as the stable envelope classes.

    For simplicity, let us consider the equator insertion to be induced from $\mathbf{f} = 1-x$, corresponding to relative insertion of $\mathcal{O}(1)$. For the $\mu_k$-equivariant equator insertion, recall that the map $K_{\mu_k \ltimes \mathbb{G}_{\mathrm{m}}^k} (\mathrm{pt}) \to K_{\mu_k \times \mathbb{G}_{\mathrm{m}}}(\mathrm{pt})$ sends elementary symmetric polynomials in the $\mathbb{G}_{\mathrm{m}}$-equivariant parameters $s_i$ to the corresponding polynomial where $s_i$ is replaced with $q^{i-1}s$, hence $e_i(s_1, \dots, s_k) \mapsto e_i (s, qs, \dots, q^{k-1}s)$. It follows that
    \begin{equation}
        \mathbf{f}^{\mathrm{eq}} = \prod_{j=0}^{k-1} (1-q^j x) \equiv 1 - x^k \in K_{\mu_k \times \mathbf{T}}(\mathfrak{X})|_{q=\zeta}
    \end{equation}
    where we quotient out by $\Phi_k(q)$ is the equator insertion for the computation of the quantum Adams operation. Here, if one wishes to compute the full $k$th cyclic power operation, then the insertion should be $\prod_{j=0}^{k-1} (1-q^j x) \in K_{\mu_k \times \mathbf{T}}(\mathfrak{X})$.

    To compute the $z^d$-coefficient of the $(1,1)$-matrix coefficient of $Q\psi^k_{\mathcal{O}(1)}$, the relevant insertion is $(\mathbf{f}_1, \mathbf{f}^{\mathrm{eq}}, \mathbf{f}_1)$-insertion along $p_1, p'_0, \dots, p'_{k-1}, p_2$. Therefore, we are reduced to computing the pushforward of
    \begin{equation}
        \mathsf{C} := \overbrace{(1-\hbar) \prod_{1 \le j \le d-1}(1-\hbar q^j)^2 \cdot \frac{\hbar^{2d-1/2}}{q^{d(d-1)}}}^{\widehat{\mathcal{O}}_{\mathrm{vir}}(T^*\mathbb{P}^1), \eqref{eqn:ovir-computation}, \eqref{eqn:ovir-twist-computation}} \cdot \overbrace{(1 - L)(1-L^k)(1-L)}^{\mathrm{ev}^*(\mathbf{f}_1 \cdot \mathbf{f}^{eq} \cdot \mathbf{f}_1)}.
    \end{equation}
    within
    \begin{equation}
        K_{\mu_k}(\mathsf{QM}_d(X)) \cong \mathbb{Z}[q^\pm , L]/ \prod_{j=0}^d (1-q^j L)^2.
    \end{equation}
    The purpose of $\mu_k$-localization is to simplify the pushforward from this ring. The restriction map
    \begin{equation}
        K_{\mu_k}(\mathsf{QM}_d(X)) \to K_{\mu_k}(\mathsf{F}_\ell) \cong \mathbb{Z}[q^\pm, L]/(1-L)^{\dim \mathsf{F}_\ell+1}
    \end{equation}
    sends $L \mapsto q^{-\ell} L$, and the pushforward is given by the sum
    \begin{equation}
        \sum_{\ell=0}^{k-1} \frac{\chi( \mathsf{C}|_{\mathsf{F}_\ell})}{\Lambda^\bullet (N_{\mathsf{F}_\ell}^\vee)} = \sum_{\ell=0}^{k-1} \frac{ \chi\left((1-\hbar) \prod_{1 \le j \le d-1}(1-\hbar q^j)^2 \cdot \frac{\hbar^{2d-1/2}}{q^{d(d-1)}} \cdot (1 - q^{-\ell} L)^2(1-q^{-\ell k}L^k)\right)  }{ \prod_{m \neq \ell}(1- q^{\ell-m})^{2 \left\lfloor \frac{d-m}{k} \right\rfloor + 1}}
    \end{equation}
    where $L |_{\mathsf{F}_\ell} = q^{-\ell} L$ and $\chi : K_{\mu_k} (\mathsf{F}_\ell) \to K_{\mu_k}(\mathrm{pt})$ sends $L^k \mapsto \binom{\dim \mathsf{F}_\ell + k}{k}$. For the other matrix coefficients, one replaces the $\mathbf{f}_1$ insertions at $p_1$ or $p_2$ with $\mathbf{f}_2$ accordingly.

    The above outline shows that the computation now clearly shows:

    \begin{thm}
        The structure constants of $Q\psi_{\mathcal{F}}$ for $X = T^*\mathbb{P}^{n}$ and $V \in K_\mathbf{T}(X)$ can be fully computed in all degrees algorithmically, in complexity uniformly bounded by $k$ for each degree $d$.
    \end{thm}
    We emphasize that this is a new computation algorithm for the $p$-curvature of the associated $q$-difference equation (discussed in the next section) that does not depend on the presentation of the $p$-curvature as an iterated product.

    \begin{rem}[{cf. \cite[Remark 7.5]{Lee23a}}]
        Under $\hbar=q^m$ specialization, one can check that the structure constants of $Q\psi^k_{\mathcal{O}(1)}$ in the stable basis are given by certain (generalized) $q$-hypergeometric series; but we do not have an \emph{a priori} explanation for this result.
    \end{rem}    
\end{example}

\section{$p$-curvature of quantum connections and quantum power operations}\label{sec:pcurvature}

In this section, we prove that for any Higgs branch target $X$, the $p$-curvature of the K\"ahler $q$-difference connection is equal to the quantum Adams operator. We would like to point out that the proof we are going to present is geometric in nature, in contrast with the algebraic arguments in \cite{Lee23b, chen2024exponential}. Then, after recalling the cohomological analog of the quantum Adams operator, the quantum Steenrod operator, and its conjectural relation with the $p$-curvature of the quantum connection, we propose some conjectures regarding how to relate the quantum power operations from $K$-theory to the ones from cohomology by degeneration.

\subsection{$p$-curvature of quantum connections}
Quantum connections and the equivariant generalizations thereof provide rich examples of connections coming from enumerative geometry. Under favorable assumptions, the quantum connection can be lifted over $\mathbb{Z}$, therefore admits well-defined mod $p$ reduction. It is natural to ask how to understand the $p$-curvature of quantum connections in terms of enumerative geometry. In \cite{Lee23b}, the second-named author conjectured that the $p$-curvatures of the equivariant quantum connections of equivariant symplectic resolutions are equal to the quantum Steenrod operators, the cohomological counterpart of the $K$-theoretic quantum Adams operators discussed in this paper. The conjecture has been established for a large class of examples, see \cite[Theorem 1.2, Remark 1.4]{Lee23b}. The relationship between $p$-curvatures and quantum Steenrod operators has also been leveraged successfully in the study of the structure of the quantum connection of monotone symplectic manifolds (the symplectic counterpart of smooth Fano varieties), see \cite{chen2024exponential}.

Analogs of $p$-curvature for $q$-difference connections at roots of unity exist, which we will recall below. The main result is that for the K\"ahler $q$-difference equation, the $p$-curvature can be identified with the quantum Adams operators.

\subsection{Equivalence of $p$-curvature and quantum Adams operators}
Following \cite{koroteev-smirnov}, we define the $p$-curvature\footnote{We term it ``$k$-curvature" to emphasize that we do not have to restrict to primes for this discussion.} of the K\"ahler $q$-difference connection in the following way. We consider the quasimap counts of $X$. Let $M_L(z)$ be the connection coefficient associated with $L \in \mathrm{Pic}^G(W)$ as recalled in \eqref{eqn:kahler-shift}. The flatness of the $q$-difference connection means that for $L_1, L_2 \in \mathrm{Pic}^G(W)$, we have
\begin{equation}
    M_{L_1 L_2}(z) = M_{L_2}(zq^{L_1}) \circ M_{L_1}(z) = M_{L_1}(zq^{L_2}) \circ M_{L_2}(z).
\end{equation}
In view of Lemma \ref{lemma:trivial}, by passing to $\mu_k$-equivariant $K$-theory, we have
\begin{equation}\label{eqn:flat}
    M_{L_1 L_2}^{(k)}(z) = M_{L_2}^{(k)}(zq^{L_1}) \circ M_{L_1}^{(k)}(z) = M_{L_1}^{(k)}(zq^{L_2}) \circ M_{L_2}^{(k)}(z),
\end{equation}
where now $q^k = 1$.

For $k \in \mathbb{Z}_{\geq 0}$, we look at
\begin{equation}
\begin{aligned}
    M_{L^k}^{(k)}(z) = M_L^{(k)}(z q^{(k-1)L}) \circ M_L^{(k)}(zq^{(k-2)L}) \circ \cdots \circ M_L^{(k)}(zq^L) \circ M_L^{(k)}(z) \\ \in \mathrm{End}(K_{\mu_k \times \mathbf{T}}(X))[\![z^{\mathrm{eff}}]\!].
\end{aligned}
\end{equation}

\begin{defn}\label{defn:p-curvature}
    The $k$-curvature of the K\"ahler $q$-difference connection of $X$ along $L$ is defined as the projection of $M_{L^k}^{(k)}(z)$ onto the summand associated to the ideal $(\Phi_k(q))$,
    \begin{equation}
        M_{L, \zeta_k}(z) := M_{L^k}^{(k)}(z)|_{q = \zeta}.
    \end{equation}
\end{defn}

The next two statements study some basic properties of the $k$-curvature.

\begin{lemma}\label{lem:p-curvature-constancy}
    For any $L \in \mathrm{Pic}^G(W)$, the operator $M_{L, \zeta_k}(z)$ is covariantly constant under the K\"ahler $q$-difference connection specialized at $q = \zeta_k$.
\end{lemma}
\begin{proof}
    Let $L' \in \mathrm{Pic}(X)$, then
    \begin{equation}
    \begin{aligned}
        &M_{L, \zeta_k}(z q^{L'})|_{q=\zeta_k} = \\
        & \big( M_L^{(k)}(z q^{(k-1)L}q^{L'}) \circ M_L^{(k)}(zq^{(k-2)L}q^{L'}) \circ \cdots \circ M_L^{(k)}(zq^Lq^{L'}) \circ M_L^{(k)}(zq^{L'}) \big)|_{q=\zeta_k} \\
        = & \big((M_{L L'}^{(k)}(zq^{(k-1)L}) \circ M_{L'}^{(k)}(zq^{(k-1)L})^{-1}) \circ \cdots \circ (M_{L L'}^{(k)}(zq^{L}) \circ M_{L'}^{(k)}(zq^{L})^{-1}) \circ (M_{L L'}^{(k)}(z) \circ M_{L'}^{(k)}(z)^{-1}) \big)|_{q=\zeta_k} \\
        = &\big(M_{L L'}^{(k)}(zq^{(k-1)L}) \circ (M_{L'}^{(k)}(zq^{(k-1)L})^{-1} \circ M_{L L'}^{(k)}(zq^{(k-2)L})) \circ \cdots \circ (M_{L'}^{(k)}(zq^{L})^{-1} \circ M_{L L'}^{(k)}(z)) \circ M_{L'}^{(k)}(z)^{-1} \big)|_{q=\zeta_k} \\
        = &\big(M_{L'}^{(k)}(zq^{kL}) \circ M_{L}^{k}(zq^{(k-1)L}) \circ M_{L}^{(k)}(zq^{(k-2)L}) \circ \cdots \circ M_{L}^{(k)}(z) \circ M_{L'}^{(k)}(z)^{-1} \big)|_{q=\zeta_k} \\
        = &M_{L'}^{(k)}(z)|_{q=\zeta_k} \circ M_{L, \zeta_k}(z) \circ M_{L'}^{(k)}(z)^{-1} |_{q=\zeta_k},
    \end{aligned}
    \end{equation}
    where we use the flatness relation \eqref{eqn:flat} and the fact $q^k=1$ in the last equation. This is exactly the condition for $M_{L, \zeta_k}(z)$ being covariantly constant.
\end{proof}

\begin{lemma}\label{lem:p-curvarture-shift}
    The operator $M_{L, \zeta_k}$ is covariantly constant under the equivariant $q$-difference connection specialized at $q = \zeta_k$.
\end{lemma}
\begin{proof}
    We need to show that for any cocharacter $\sigma: \mathbb{C}^{\times}_q \to \mathbf{T}$, we have
    \begin{equation}\label{eqn:equivairant-constancy}
        M_{L, \zeta_k}(a \zeta_k^{\sigma}) \circ S_{\sigma}|_{q=\zeta_k} = S_{\sigma}|_{q=\zeta_k} \circ M_{L, \zeta_k}(a),
    \end{equation}
    where $a$ is the equivariant variable. This follows from the compatibility between the K\"ahler and equivariant $q$-difference connections. Namely, working $\mathbb{C}_{q}^{\times}$-equivariantly, we know that
    \begin{equation}\label{eqn:k-e-compatible}
        M_L(aq^{\sigma}) \circ S_{\sigma} = S_{\sigma} \circ M_L(a).
    \end{equation}
    This is true by \cite[Section 8]{Oko17}, because if we write $J(a)$ to be the fundamental solution, then $M_L(a) \circ J(a)$ is again a fundamental solution to the $q$-difference equations, which in turn implies that 
    \begin{equation}
        M_L(aq^{\sigma}) \circ J(aq^{\sigma}) \circ E(\sigma) = S_{\sigma} \circ M_L(a) \circ J(a),
    \end{equation}
    where $E(\sigma)$ is the operator from \cite[Equation (8.2.13)]{Oko17}. Then using the relation $J(aq^{\sigma}) \circ E(\sigma) = S_{\sigma} \circ J(a)$ (cf. \cite[Theorem 8.2.20]{Oko17}), we can multiply the above equation by $J(a)^{-1}$ to obtain \eqref{eqn:k-e-compatible}. Finally, to prove \eqref{eqn:equivairant-constancy}, we simply put $S_{\sigma}|_{q=\zeta_k}$ to the right of $M_{L, \zeta_k}$ and apply the commuting relation \eqref{eqn:k-e-compatible} $k$-times and use again $q^k = \zeta_k^k = 1$.
\end{proof}

Now we investigate the relationship between the $k$-curvature $M_{L^k}^{(k)}$ and quantum Adams operator $Q\psi^{\mathrm{desc},k}_{L}$, where we view $L \in \mathrm{Pic}^G(W)$ as a descendant insertion. The following lemma is the key observation for showing the equivalence of quantum Adams operators with the $k$-curvature operator. We work with the moduli space $\mathsf{QM}_d^{\mu_k}(X)_{\rel  p_1, p_2}$ from \eqref{eqn:cyclic-moduli}. Let $\pi: C \to \mathbb{P}^1$ be a destabilization which is the domain of a quasimap $f: C \to \mathfrak{X}$ in $\mathsf{QM}_d^{\mu_k}(X)_{\rel  p_1, p_2}$. Recall that we have the marked points $p_0', \dots, p_{k-1}'$ on $\mathbb{P}^1$ at the $k$-th roots of unity.

\begin{lemma}\label{lem:deformation-to-equator}
     There is an isomorphism of $K$-theory classes in $K_{\mu_k \times \mathbf{T}}(\mathsf{QM}_d^{\mu_k}(X)_{\rel  p_1, p_2})$,
    \begin{equation}
        H^\bullet \left(C, \left(\pi^* \mathcal{O}_{p_1} \right)^{\oplus k} \right) \cong H^\bullet \left( C, \pi^* \mathcal{O}_{p_0'} \oplus \cdots \oplus \pi^* \mathcal{O}_{p_{k-1}'} \right).
    \end{equation}
\end{lemma}
\begin{proof}
    We refer the reader to Definition \ref{defn:kahler-shift} for a discussion on the notations. Both classes are obtained as pushforwards from the universal curve $\mathcal{C} \to \mathsf{QM}_d^{\mu_k}(X)_{\rel  p_1, p_2}$. It suffices to explain the fiberwise computation, i.e. with a fixed relative quasimap with domain $C$ with stabilization map $\pi: C \to \mathbb{P}^1$.

    The projection $\tilde{p}: C \to \mathrm{pt}$ factors through $ p \circ \pi : C \to \mathbb{P}^1 \to \mathrm{pt}$. Hence by the projection formula
    \begin{equation}
        R\widetilde{p}_* (\pi^*\mathcal{O}_{p_1}^{\oplus k} ) \cong Rp_* R\pi_* (\pi^* \mathcal{O}_{p_1}^{\oplus k}) \cong Rp_* \left( \mathcal{O}_{p_1}^{\oplus k } \otimes_{\mathcal{O}_{\mathbb{P}^1}} R\pi_* \mathcal{O}_{C} \right) \cong Rp_* (\mathcal{O}_{p_1}^{\oplus k} )
    \end{equation}
    where the last isomorphism follows from the fact that the fibers of $\pi: C \to \mathbb{P}^1$ are connected nodal genus 0 curves (or points) so that $R\pi_* \mathcal{O}_{C} \cong \mathcal{O}_{\mathbb{P}^1}$.
    Applying the projection formula to $H^\bullet \left( C, \pi^* \mathcal{O}_{p_0'} \oplus \cdots \oplus \pi^* \mathcal{O}_{p_{k-1}'} \right)$ the same way gives
    \begin{equation}
        R\widetilde{p}_* (\pi^* \mathcal{O}_{p_0'} \oplus \cdots \oplus \pi^* \mathcal{O}_{p_{k-1}'} ) \cong Rp_* \left( \mathcal{O}_{p_0'} \oplus \cdots \oplus \mathcal{O}_{p_{k-1}'}\right), 
    \end{equation}
    so the claim reduces to showing that $Rp_* \mathcal{O}_{p_1}^{\oplus k} \cong Rp_* \left( \mathcal{O}_{p_0'} \oplus \cdots \oplus \mathcal{O}_{p_{k-1}'} \right)$. This follows from the $\mu_k$-equivariant deformation $(t p'_0, \dots, t p'_{k-1})$ for $t \in \mathbb{A}^1$ interpolating $(p_1, \dots, p_1)$ and $(p_0', \dots, p_{k-1}')$ which identifies the corresponding $K$-theory classes.
\end{proof}

Now we are ready to prove the equivalence between the quantum Adams operator and the $k$-curvature operator.

\begin{thm}\label{thm:qadams=pcurvature}
    Let $Q\psi_L^{\mathrm{desc},k}(z) \in \mathrm{End}(K_{\mu_k \times \mathbf{T}}(X)|_{q = \zeta} ) [\![z^{\mathrm{eff}}]\!]$ be the descendant quantum Adams operator for $L \in \mathrm{Pic}^G(W)$. Then
    \begin{equation}
        Q\psi_L^{\mathrm{desc},k}(z) = M_{L, \zeta} (z).
    \end{equation}
\end{thm}

\begin{proof}
    This is an analog of \cite[Theorem 16]{PSZ-quantum}. Recall that we have
    \begin{equation}
        M_{L^k}^{(k)}(z) = \big( \sum_d z^d (\mathrm{ev}_{p_1} \times \mathrm{ev}_{p_2})_* \big( \mathsf{QM}_{d}(X)_{\rel p_1, p_2}, \hat{\mathcal{O}}_{\mathrm{vir}} \otimes\mathrm{det} H^{\bullet}(L^k \otimes \pi^*(\mathcal{O}_{p_1})) \big) \big)\circ \mathbf{G}^{-1}.
    \end{equation}
    We first claim that the following relation holds:
    \begin{equation}
        \mathrm{det} H^{\bullet}(L^k \otimes \pi^*(\mathcal{O}_{p_1})) \cong \mathrm{det} H^{\bullet}(L^{\oplus k} \otimes \pi^*(\mathcal{O}_{p_1})).
    \end{equation}
    To see this, continuing the notations in introduced in Lemma \ref{lem:deformation-to-equator}, observe that
    \begin{align}
        \mathrm{det} H^{\bullet}(L^{\oplus k} \otimes \pi^*(\mathcal{O}_{p_1})) &\cong \bigotimes_k \det H^\bullet (L \otimes \pi^* (\mathcal{O}_{p_1})) \\ & = \bigotimes_k \det Rp_* R\pi_{*}(f^*L \otimes \pi^*(\mathcal{O}_{p_1})) \\
        &\cong \bigotimes_k \det Rp_* (R\pi_* L \otimes \mathcal{O}_{p_1} ) \\
        &\cong \bigotimes_k \det (R\pi_*L)|_{p_1}
    \end{align}
    To compute this, consider the fiber $E:= \pi^{-1}(p_1)$ of $p_1$ which is a connected genus $0$ curve such that $C = \mathbb{P}^1\cup_q E$ where $q$ is the nodal point and $\pi:C \to \mathbb{P}^1$ collapses $E$. The perfect complex $R\pi_*L$ on $\mathbb{P}^1$ can be decomposed via the triangle
    \begin{equation}
        R \pi_*L \to L|_{\mathbb{P}^1} \oplus i_* R\pi_* (L|_E) \to i_* (L|_{p_1})
    \end{equation}
    for the inclusion $i:p_1 \to \mathbb{P}^1$. Taking determinants, it follows that
    \begin{equation}
       \det R\pi_* L \simeq L|_{\mathbb{P}^1} \otimes \mathcal{O}(p_1)^{\otimes (\chi(E, L|_E) -1)} = L|_{\mathbb{P}^1} \otimes \mathcal{O}(p_1)^{\otimes \deg_E L},
    \end{equation}
    i.e. we twist by $\mathcal{O}(p_1)$ exactly as many times as the divisor of $L$ meets $E$; the final equality uses that $\chi(E, \mathcal{O}_E) =1$ since $E$ is a connected nodal genus $0$ curve. Hence $L \to \det R\pi_*L$ is clearly multiplicative in $L$ with tensor products in line bundles $L$, and this implies that $\bigotimes_k \det(R\pi_*L)|_{p_1}$ maps isomorphically to 
    \begin{equation}
        \det (R\pi_* L^k)|_{p_1} \cong \det Rp_* \left( R\pi_* f^*L^k \otimes \mathcal{O}_{p_1} \right) \cong \det H^\bullet (L^k \otimes \pi^* \mathcal{O}_{p_1}).
    \end{equation}

    Applying Lemma \ref{lem:deformation-to-equator} and the above computation, we conclude that in the $\mu_k \times \mathbf{T}$-equivariant $K$-theory, we have
    \begin{equation}
    \begin{aligned}
        &M_{L^k}^{(k)}(z) = \\
        & \big( \sum_d z^d (\mathrm{ev}_{p_1} \times \mathrm{ev}_{p_2})_* \big( \mathsf{QM}_{d}(X)_{\rel p_1, p_2}, \hat{\mathcal{O}}_{\mathrm{vir}} \otimes\mathrm{det} H^{\bullet}(L \otimes \pi^*(\mathcal{O}_{p_1})^{\oplus k}) \big) \big)\circ \mathbf{G}^{-1} \\
        &= \big( \sum_d z^d (\mathrm{ev}_{p_1} \times \mathrm{ev}_{p_2})_* \big( \mathsf{QM}_{d}(X)_{\rel p_1, p_2}, \hat{\mathcal{O}}_{\mathrm{vir}} \otimes\mathrm{det} H^{\bullet}(L \otimes (\pi^* \mathcal{O}_{p_0'} \oplus \cdots \oplus \pi^* \mathcal{O}_{p_{k-1}'})) \big) \big)\circ \mathbf{G}^{-1} \\
        &= \big( \sum_d z^d (\mathrm{ev}_{p_1} \times \mathrm{ev}_{p_2})_* \big( \mathsf{QM}_{d}(X)_{\rel p_1, p_2}, \hat{\mathcal{O}}_{\mathrm{vir}} \otimes (\mathrm{ev}^{\mathrm{stack}}_k)^* L^{\boxtimes k}_{\mathrm{eq}} \big) \big)\circ \mathbf{G}^{-1} = Q\Psi^{\mathrm{desc},k}_{L}(z).
    \end{aligned}
    \end{equation}
    Therefore, by specializing $q = \zeta_k$, we see that $Q\psi_L^{\mathrm{desc},k}(z) = M_{L, \zeta} (z)$.
\end{proof}

In view of \cref{thm:qadams=pcurvature}, Lemma \ref{lem:p-curvature-constancy} and Lemma \ref{lem:p-curvarture-shift} provide alternative proofs of the compatibility between the quantum Adams operator and the $q$-difference module structure stated in Proposition \ref{prop:adams-property}, and vice versa. We keep both proofs as they offer different insights on covariant constancy.

\subsection{Cohomological analog: Quantum Steenrod operators}\label{ssec:quantumSteenrod-defn}

Quantum Steenrod operations were introduced by \cite{Fuk97} and developed by \cite{Wil-sur, seidel-wilkins}. In \cite{Lee23a, Lee23b}, the second-named author introduced torus-equivariant generalizations of these operations. These are quantum deformations of the total Steenrod power operator for a fixed prime $k=p>2$ such that $H^*_{\mu_p}(\mathrm{pt};\mathbb{F}_p) \cong \mathbb{F}_p [\![ t, \theta]\!]$ generated by commuting variables of degree $|t|=2$, $|\theta|=1$:
\begin{align}
H^*_{\mathbf{T}}(X;\mathbb{F}_p) &\to H^*_{\mu_p \ltimes \mathbf{T}^p}(X^k;\mathbb{F}_p) \overset{\Delta^*}{\to} H^*_{\mu_p \times \mathbf{T}} (X;\mathbb{F}_p) \cong H^*_{\mathbf{T}}(X;\mathbb{F}_p) [\![ t, \theta]\!] \\
b & \mapsto  (b^{\otimes p})_{\mathrm{eq}} \mapsto \mathrm{St}(b).
\end{align}
Here, following our convention in $K$-theory, $(b^{\otimes p})_{\mathrm{eq}}$ is the external $p$-fold cup product of $b$ viewed as a $\mu_p$-equivariant class. The coefficients of $\mathrm{St}(b)$ in the expansion as a polynomial over $t, \theta$ (under the last K\"unneth isomorphism) are the classical Steenrod power operations applied to $b$, up to signs. 

Consider the moduli space of stable maps $\mathsf{M}_d(X)$ of degree $d$ into $X$ considered as a quasiprojective variety with a rigid component $\mathbb{P}^1$. The moduli space carries evaluation maps
\begin{equation}
    \mathrm{ev}_{p_1} \times \mathrm{ev}_p \times \mathrm{ev}_{p_2} : \mathsf{M}_d(X) \to X \times X^{\times p} \times X.
\end{equation}

\begin{defn}
    Fix $p$ a prime. Take $b \in H^*_{\mathbf{T}}(X)[\![ z^{\mathrm{eff}} ]\!]$. The \emph{quantum Steenrod power operator} associated to $\mathcal{F}$ is defined as
    \begin{equation}\label{eqn:power-H}
    \begin{aligned}
        Q\Sigma_b : = \sum_{d} z^d (\mathrm{ev}_{p_1} \times \mathrm{ev}_{p_2})_* \left( [\mathsf{M}_d(X)]^{\mathrm{vir}}_{\mu_p} \cap \mathrm{ev}_p^* ((b^{\otimes p})_{\mathrm{eq}}) \right) \\ \in \mathrm{End}(H^*_{\mu_k \times \mathbf{T}}(X))[\![ z^{\mathrm{eff}} ]\!],
    \end{aligned}
    \end{equation}
    where we view the cohomology classes of $X \times X$ as endomorphisms on $H^*_{\mu_k \times \mathbf{T}}(X)$ via convolution.
\end{defn}

For the explanation of the counts $(\mathrm{ev}_{p_1} \times \mathrm{ev}_{p_2})_* \left( [\mathsf{M}_d(X)]^{\mathrm{vir}}_{\mu_p} \cap \mathrm{ev}_p^* ((b^{\otimes p})_{\mathrm{eq}}) \right)$ from the (hypothetical) $\mu_p$-equivariant virtual fundamental class, which we do not explain in full detail here, we refer the reader to \cite[Section 3.1.1]{bai-lee} and references therein. For the discussion in this section, it suffices to consider the quantum Steenrod operators simply as the cohomological analogs of quantum Adams operators, defined in $\mathbb{F}_p$-coefficient quantum cohomology, using the exact same configuration of $\mu_p$-symmetric parametrized copy of $\mathbb{P}^1$ with marked points at the roots of unity as the source curve. Due to technical difficulties in defining the mod $p$ counts, the construction uses stable map compactification in symplectic enumerative geometry as opposed to an algebro-geometric framework, see \cite[Section 4]{seidel-wilkins} or \cite[Section 2.4]{Lee23a}.

In \cite{Lee23b}, the second-named author verified that for conical symplectic resolutions with isolated torus fixed points and semisimple quantum cohomology, the quantum Steenrod operators agree with the $p$-curvature of the quantum  \emph{differential} connection.  Recall that the quantum differential connection is an $\mathbb{F}_p[\![t, \theta]\!] \otimes H^*_{\mu_p \times \mathbf{T}}(\mathrm{pt})$-linear operator indexed by $x \in H^2_\mathbf{T}(X;\mathbb{Z})$,
\begin{equation}\label{eqn:qconn}
    \nabla_x := t \partial_{\bar{x}} + x \ \star_\mathbf{T} : H^*_{\mu_p \times \mathbf{T}} (X) [\![z^{\mathrm{eff}}]\!] \to H^*_{\mu_p \times \mathbf{T}} (X) [\![z^{\mathrm{eff}}]\!]
\end{equation}
where $\partial_{\bar{x}}$ acts on $z^\alpha \in \mathbb{F}_p[\![z^{\mathrm{eff}}]\!]$ by $\partial_{\bar{x}} z^\alpha = \langle \alpha, \bar{x} \rangle z^\alpha$. Here, $t$ is the equivariant parameter for the $\mu_p(\mathbb{C})$-action on the \emph{source} curve for holomorphic maps $u : \mathbb{P}^1 \to X$ counted for defining the quantum connection, i.e. the cohomological analogue of the loop parameter $q \in K_{\mu_p}(\mathrm{pt}) \cong \mathbb{Z}[q]/(q^p=1)$. A key property of $\nabla_x$ is that it is a flat connection, i.e. the operators $\nabla_x, \nabla_{x'}$ commute for two choices of $x, x' \in H^2_{\mathbf{T}}(X)$.

\begin{thm}[{\cite[Theorem 1.2]{Lee23b}}]\label{thm:qst-is-pcurv}
    Fix a conical Hamiltonian symplectic resolution $X$ with isolated $T$-fixed points. Then for almost all $p$ and $x \in H^2_{\mathbf{T}}(X;\mathbb{F}_p)$, we have
    \begin{equation}
        Q\Sigma_x = \nabla_x^p - t^{p-1} \nabla_x + N
    \end{equation}
    where $N$ is a nilpotent operator.
\end{thm}
The expression $\nabla_x^p - t^{p-1}\nabla_x$ on the right hand side is the $p$-curvature of the quantum differential connection. Although the connection $\nabla_x$ \emph{is not} $z$-linear as it satisfies the Leibniz rule, the $p$-curvature $\nabla_x^p - t^{p-1} \nabla_x$ \emph{is} $z$-linear.

Conjecturally the nilpotent error always vanishes so that the quantum Steenrod operators on degree $2$ classes are expected to always be the $p$-curvature of the quantum connection. The nilpotent error certainly vanishes under assumptions of generic semisimplicity. Recall that a collection of pairwise commuting linear operators $\{F_i\}$ acting on a vector space $V$ over a field $K$ has \emph{jointly simple spectrum} if the simultaneous eigendecomposition for $\{F_i\}$ acting on $V \otimes_K \overline{K}$ is into rank $1$ subspaces.

\begin{cor}[{\cite[Corollary 5.2]{Lee23b}}]\label{thm:qst-is-pcurv-onthenose}
    Suppose the quantum multiplication operators $y  \ \star_\mathbf{T}$ associated to degree $2$ classes $y \in H^2_{\mathbf{T}}(X;\mathbb{F}_p)$ have jointly simple spectrum. Then for $x \in H^2_{\mathbf{T}}(X;\mathbb{F}_p)$,
        \begin{equation}
        Q\Sigma_x = \nabla_x^p - t^{p-1} \nabla_x.
    \end{equation}
\end{cor}
For example, such semisimplicity assumption is satisfied for Springer resolutions \cite{Lee23a} and hypertoric varieties \cite{bai-lee}.

\subsection{Degenerating $K$-theory to cohomology}\label{ssec:quantumSteenrod-limit}

We explain the relationship between the statement about the equivalence of $k$-curvature with quantum Adams operators and the corresponding statement in cohomology, by describing the cohomological limit $q \to \zeta$ for the quantum Adams operations. As one application, this yields a proposal for the algebro-geometric definition of the quantum Steenrod operations.

Here we follow the computational approach as explained in \cite[Section 5.1, 5.2]{koroteev-smirnov}, which should be considered as a formal and numerical recipe for taking the cohomological limit. The conceptual justification of this recipe is later explained in \cref{rem:degeneration-k-to-h}.


Fix the integer $k$ to be a prime $k=p  >2$, so that $K_{\mu_p}(\mathrm{pt} ) \cong \mathbb{Z}[q]/(q^p=1)$, and $K_{\mu_p}(\mathrm{pt})|_{q=\zeta} \cong \mathbb{Z}[q]/(\Phi_p(q)) \cong \mathbb{Z}[\zeta]$ where $\zeta = \zeta_p$ is a $p$th root of unity. The projection map down to the quotient by the ideal $(\Phi_p(q))$ sends $q \mapsto \zeta$. The quantum Adams operators and the $p$-curvature described in the earlier subsections
\begin{equation}
    Q\psi^p_{\hat{L}} (z) = M_{L, \zeta}(z) \in \mathrm{End}(K_{\mu_k \times \mathbf{T}} (X)|_{q=\zeta} ) [\![z^{\mathrm{eff}}]\!]
\end{equation}
are indeed linear over the ring $K_{\mu_p}(\mathrm{pt})|_{q=\zeta} \cong \mathbb{Z}[\zeta]$.

Now we formally introduce two invertible variables $\beta, t$ of (cohomological) degrees $|\beta|=-2$ and $|t|=2$ such that 
\begin{equation}
    q = \zeta = 1 + \beta t 
\end{equation}
and consider the coefficients of the operators $Q\psi^p_{\hat{L}}(z)$ and $M_{L, \zeta}(z)$ in the $p$-completion $\mathbb{Z}_p[\zeta]$. 

\begin{rem}
    The passage to $p$-completion is  an artifact of the  ``2-periodization'' that we implicitly introduced by introducing the variables $\beta$ and $t$; see \cref{rem:degeneration-k-to-h} for a further discussion.
\end{rem}

The following is elementary:
\begin{lemma}
    
    There is a unique homomorphism $\mathbb{Z}_p[\zeta] \to \mathbb{F}_p$ which fits into the diagram

    \begin{center}
    \begin{tikzcd}
        \mathbb{Z}_p[q]/(q^p=1) \rar["q \mapsto 1"] \dar["q \mapsto \zeta"'] & \mathbb{Z}_p \dar["/(p)"] \\ \mathbb{Z}_p[\zeta] \rar & \mathbb{F}_p
    \end{tikzcd}
    \end{center}
    and sends $\zeta \mapsto 1$.
\end{lemma}
In view of the expansion $q = 1 + \beta t$, the homomorphism $\mathbb{Z}_p[\zeta] \to \mathbb{F}_p$ can be understood as ``$ \beta \mapsto 0$''.

In \cite{koroteev-smirnov}, one uses the coefficients $\mathbb{Q}_p(\pi)$ for ``Dwork's element'' $\pi$, that is a solution to $\pi^{p-1} = -p$, instead of $\mathbb{Q}_p(\zeta) \cong \mathrm{Frac} \ \mathbb{Z}_p[\zeta]$. The relationship between these coefficients and our choice is given by another elementary lemma:

\begin{lemma}
    There is an isomorphism $\mathbb{Q}_p(\zeta) \cong \mathbb{Q}_p(\pi)$ of field extensions of $\mathbb{Q}_p$.
\end{lemma}
\begin{proof}
    Both fields are of degree $p-1$ over $\mathbb{Q}_p$, and the minimal polynomials $X^{p-1} + p $ and $\Phi_p(X)$ are both irreducible over $\mathbb{Q}_p$. Therefore it suffices to exhibit $\zeta$ as an element of $\mathbb{Q}_p(\pi)$. Consider the minimal polynomial of $(\zeta -1)$ in $\mathbb{Q}_p$ which is given by $\Phi_p(X+1) = [(X+1)^p - 1]/X$; it is easy to compute
    \begin{equation}
        \frac{\Phi_p(\pi X+1)}{\pi^{p-1}} \equiv X^{p-1}  - 1  \ (\mathrm{mod} \ \pi) .
    \end{equation}
    This has solutions in $\mathbb{Z}_p[\pi]/(\pi) \cong \mathbb{F}_p$, so by Hensel's lemma the original polynomial $\Phi_p(\pi X+ 1) $ has solutions in $\mathbb{Q}_p(\pi)$. Hence $(\zeta-1) /\pi \in \mathbb{Q}_p(\pi)$ which implies the desired result.
\end{proof}

Using the expansion in $\zeta - 1 = \beta t$ instead of $\pi \sim \zeta - 1 \ (\mathrm{mod} \ \pi^2)$ used in \cite{koroteev-smirnov}, the result from \cite[Lemma 5.1, Section 5.4]{koroteev-smirnov} can be rewritten as follows. As in the expansion $q = 1+ \beta t$, one expands the $\mathbf{T}$-equivariant parameters $a \in K_{\mathbf{T}}(\mathrm{pt})$ as $a = 1 + \beta \bar{a}$ where $\bar{a} \in H^*_{\mathbf{T}}(\mathrm{pt})$ are the cohomological equivariant parameters. 

\begin{prop}[{\cite[Section 5.2]{koroteev-smirnov}}]\label{prop:pcurvature-cohomological-limit}
    Consider the expansion of $M_{L, \zeta}(z)$ in $(q-1) = \beta t$; then 
    \begin{equation}
        M_{L, \zeta}(a, z) = 1 + \beta^p  F_{c_1(L)}(\bar{a}, z) + O(\beta^{p+1})
    \end{equation}
    where $F_{c_1(L)}(\bar{a}, z) = \nabla^p_{c_1(L)} - t^{p-1} \nabla_{c_1(L)}$ modulo $p$ is the $p$-curvature of the quantum differential connection in the direction of $c_1(L) \in H^*_{\mathbf{T}}(X)$.
\end{prop}

Combining the previous \cref{thm:qadams=pcurvature} together with the cohomological limit taken by \cite{koroteev-smirnov}, we obtain the following diagram:
\begin{equation}\label{eqn:cohomology-degeneration-diagram}
    \begin{tikzcd}
        Q \psi_{\hat{L}}^p(z) \arrow[r, leftrightarrow] \arrow[d,  dashed] & M_{L, \zeta_p}(z) \arrow[d, squiggly] \\ Q\Sigma_{c_1(L)} \arrow[r, leftrightarrow]  & F_{c_1(L)}(z),
    \end{tikzcd}
\end{equation}
where the horizontal arrows are \cref{thm:qadams=pcurvature} and \cref{thm:qst-is-pcurv-onthenose} respectively, and the vertical arrows refer to the ``cohomological limit'' as explained in \cref{prop:pcurvature-cohomological-limit}, i.e. the process of extracting the suitable term (the $\beta^p$-term) in the associated graded of the $\beta$-filtration. 

Based on the diagram we therefore propose that the cohomological limit of the quantum Adams operator is given by the quantum Steenrod operation. 
\begin{conj}\label{conj:qadams-degenerates-to-qst-divisors}
    Let $[L] \in K_{\mathbf{T}}(X)$ be a class of a line bundle and $c_1(L) \in H^2_{\mathbf{T}}(X;\mathbb{F}_p)$ be its first Chern class modulo $p$. Then
    \begin{equation}
        Q\psi_L^p(z) = 1 + \beta^p Q\Sigma_{c_1(L)}(z) + O(\beta^{p+1}).
    \end{equation}
\end{conj}

For more general insertions than line bundles, we need to consider the support filtration on $K$-theory to associate a cohomology class for a given $K$-theoretic insertion. Denote $X_i \subseteq X$ by the $i$-skeleton of $X$, and let $K_i(X) = \ker (K(X) \to K(X_{i-1}))$ to be the support filtration on $K$-theory, which is a decreasing filtration.

\begin{assm}
    There exists a natural identification of
    \begin{equation}
        \mathrm{Gr} K(X) \otimes \mathbb{F}_p \cong \bigoplus_k H^{2k}(X;\mathbb{F}_p)
    \end{equation}
    of the associated graded of the support filtration on $K$-theory with even cohomology.
\end{assm}

Such assumption holds (in fact over $\mathbb{Z}$) for $X$ with no torsion in cohomology. Now let $K_{2\ell}(X) = \ker (K(X) \to K(X_{2\ell-1})) \subseteq K(X)$ to be the filtered piece in the support filtration which reduces to $\bigoplus_{k \ge \ell} H^{2k}(X;\mathbb{F}_p)$ in mod $p$.

\begin{conj}\label{conj:qadams-degenerates-to-qst}
    Let $\mathcal{F} \in K_{2\ell}(X)$ and $\overline{\mathcal{F}} \in H^{2\ell}(X;\mathbb{F}_p)$ be its image in the associated graded of the support filtration. Then
    \begin{equation}
        Q\psi_\mathcal{F}^p (z) = \beta^{p\ell} Q\Sigma_{\overline{\mathcal{F}}} (z) + O(\beta^{p\ell+1}).
    \end{equation}
\end{conj}

\begin{rem}
    The conjecture can be addressed more easily in the context of topological $K$-theory, see also \cref{rem:degeneration-k-to-h} and \cref{rem:degeneration-atiyah}. Indeed for topological $K$-theory and $p > \dim_{\mathbb{C}} X$, assuming $X$ has only even cohomology with no torsion, there is a natural splitting of the mod $p$ Atiyah--Hirzebruch spectral sequence $K(X) \otimes \mathbb{F}_p \cong \bigoplus H^{2k}(X;\mathbb{F}_p)$ induced by the truncated Chern character. The case discussed in \cref{conj:qadams-degenerates-to-qst-divisors} corresponds to $c_1(L) = \mathrm{ch}_1(L) \in H^2$ being the associated graded image of $[L]-1 \in K_2(X)$ in the first filtered piece in $K$-theory.
\end{rem}

By \cref{thm:qadams=pcurvature} and contingent on \cref{prop:pcurvature-cohomological-limit}, the more restricted \cref{conj:qadams-degenerates-to-qst-divisors} is true for all examples for which the bottom horizontal arrow in the diagram \cref{eqn:cohomology-degeneration-diagram} is a theorem, as in \cref{thm:qst-is-pcurv-onthenose}:

\begin{cor}\label{cor:qadams-degenerates-to-qst}
    \cref{conj:qadams-degenerates-to-qst-divisors} is true for any conical Hamiltonian symplectic resolutions $X$ with isolated $T$-fixed points whose quantum multiplication operators by divisor classes have jointly simple spectrum.
\end{cor}
In particular, \cref{conj:qadams-degenerates-to-qst-divisors} holds for Springer resolutions and hypertoric varieties \cite{Lee23b, bai-lee}.

Following \cref{conj:qadams-degenerates-to-qst}, it is also possible to \emph{define} the quantum Steenrod operations as a certain degenerating limit of the quantum Adams operations. Such approach provides an indirect definition of the quantum Steenrod operations based on quasimap quantum cohomology, which would agree with the $p$-curvature of the quantum connection in the case of divisor classes by construction and \cref{thm:qadams=pcurvature}.

We expect such quasimap quantum Steenrod operations (at least for sufficiently positive targets) agree with the Gromov--Witten quantum Steenrod operations as defined in \cite{seidel-wilkins, Lee23b}. A direct proof, involving the comparison of different compactifications of moduli spaces, is not pursued here. Nevertheless we note that \cref{cor:qadams-degenerates-to-qst} verifies this equivalence for a large class of examples.

\begin{rem}\label{rem:degeneration-k-to-h}
Following the approach in \cite[Remark 9.8]{devalapurkar-satake}, it is possible to explain the cohomological limit above if one replaces algebraic $K_0$ with \emph{connective $K$-theory} coefficients $ku$. The relevant computations following \cite{devalapurkar-satake}, are as follows (our $t$ is $t^{-1}$ in \emph{loc. cit.}), assuming the isomorphism $\mu_p(\mathbb{C}) \cong \mathbb{Z}/p$:
\begin{itemize}
    \item There is an isomorphism $ku^{t\mathbb{Z}/p} \simeq \mathbb{Z}_p[\zeta_p]^{tS^1} $, where unit map $ku \to ku^{t\mathbb{Z}/p}$ sends $\beta \mapsto (\zeta-1)/t$ for $\beta$ the Bott element. On the level of homotopy groups, $\pi_*(ku^{t\mathbb{Z}/p}) \cong \mathbb{Z}_p[\zeta_p, t^{\pm 1}]$, where $|t|=2$ in cohomological grading, 
    \item By inverting $\beta$ (hence $\zeta-1$ and $p$), connective $K$-theory degenerates to complex $K$-theory, where $ku^{t\mathbb{Z}/p}[\beta^{-1}] \simeq KU^{t\mathbb{Z}/p} \simeq \mathbb{Q}_p(\zeta)^{tS^1}$.
    \item By modding out $\beta$  (hence $\zeta-1$ and $p$), connective $K$-theory degenerates to (mod $p$) cohomology, where $ku^{t\mathbb{Z}/p} /\beta \simeq \mathbb{Z}^{t\mathbb{Z}/p} \simeq \mathbb{F}_p^{tS^1}$ with $\pi_*(\mathbb{F}_p^{tS^1}) \cong \mathbb{F}_p[t, t^{-1}]$.
\end{itemize}
In other words, working with connective $K$-theory has the effect of interpolating between $K$-theory and cohomology, which in the $\mathbb{Z}/p$-equivariant setting puts the cohomological limit we described in a rigorous footing.

As the theory of quasimap quantum Adams operators that we introduce in this paper is based on algebraic $K_0$ and its functorial properties as described in \cite[Chapter 5]{chriss-ginzburg}, this should be considered as a heuristic justification. We expect to be able to describe the limit from quantum Adams operators to quantum Steenrod operations based on the computations above in the setting of \emph{stable map} quantum Adams operators based on the formalism of Floer homotopy theory, which we will pursue separately.
\end{rem}



\begin{rem}\label{rem:degeneration-atiyah}
    To describe the limit from Adams operations to Steenrod operations in the classical setting, Atiyah \cite[Lemma 5.1, Proposition 5.6, Theorem 6.5]{atiyah-power} uses the associated graded with respect to the skeletal filtration $K_{2m}(X) := \ker \left( K(X) \to K(X_{2m-1}) \right)$ where $X_i$ denotes the $i$-skeleton of $X$. The approach in \cref{rem:degeneration-k-to-h} is a reformulation of Atiyah's result in the sense that the filtration by powers of $\beta$ in $ku^0$ is exactly the skeletal filtration, as explained below. The elements of $K_{2m}(X)$ are (homotopy classes of) maps $X \to \mathrm{BU}$ which lift to the $2m$'th Postnikov tower $\mathrm{BU}\langle 2m \rangle \to \mathrm{BU}$. Since $\{ \mathrm{BU}\langle 2m\rangle \}$ represent connective $K$-theory, the skeletal filtration can be understood as
    \begin{equation}
        K_{2m}(X) = \beta^m ku^{2m}(X) \subseteq ku^0(X) = K(X).
    \end{equation}
    In particular, modding out by $\beta$ corresponds to passing to the associated graded of the skeletal filtration.
\end{rem}

\section{$K$-theoretic quantum Hikita conjecture at roots of unity}\label{sec:quantumHikita}
As an application of our theory, we formulate the $K$-theoretic quantum Hikita conjecture at roots of unity. The $K$-theoretic quantum Hikita conjecture is a $K$-theoretic refinement of the 3D mirror symmetry statement originally envisioned in \cite[Remark 1.4]{KMP21}, and introduced in \cite{zhou2021virtual}. It is a generalization of the $K$-theoretic Hikita conjecture of \cite{dumanski-krylov}, and is the subject of general consideration in the upcoming work of Dinkins--Karpov--Krylov.

Our conjecture, which discusses the statement at $q^k=1$, should be considered as a $K$-theoretic analog of the cohomological ``arithmetic'' 3D mirror symmetry conjecture in characteristic $p$ \cite{bai-lee}. This continues the theme of degenerating from $K$-theory to cohomology introduced in \cref{ssec:quantumSteenrod-limit}, and under the degeneration limit we expect that the $K$-theoretic version implies the cohomological version.

To illustrate the arithmetic aspect of the $K$-theoretic quantum Hikita conjecture at roots of unity, we verify the conjecture for abelian gauge theory (hypertoric varieties) as a proof of principle. In these examples, the claim reduces to checking the action of quantum Adams operators by line bundles, which can be described via $p$-curvature of the $q$-difference module by \cref{thm:qadams=pcurvature}.

\subsection{3D mirror symmetry and quantum Hikita conjecture}\label{ssec:quantumHikita-introduction}
In this subsection, we briefly survey 3D mirror symmetry and the statement of quantum Hikita conjecture. For a general introduction we refer the reader to \cite{kamnitzer-survey} and the original paper \cite{KMP21}.

In short, 3D mirror symmetry posits that (conical Hamiltonian) symplectic resolutions (see, e.g., \cite[Section 2]{bai-lee}) should come in pairs, $X$ and $X^!$, such that certain ``equivariant'' data of $X$ and ``K\"ahler'' data of $X^!$ are interchanged under this duality. The quantum Hikita conjecture is one manifestation of the above slogan, where the notion of equivariant and K\"ahler data are given by certain $D$-modules (cf. \cite{KMP21}) depending on the quantization of $X$ and curve counts of $X^!$. 

\begin{rem}
    Unlike the conventions used in previous sections, we use $X^!$ instead of $X$ to denote the Higgs branch whose enumerative geometry is the subject of our discussion following the conventions in \cite{KMP21, bai-lee}.
\end{rem}

\subsubsection{Cohomological quantum Hikita conjecture}
We first quickly recall the cohomological formulation of the quantum Hikita conjecture, to later compare with the $K$-theoretic version. To simplify the discussion we will assume that $X$ is a BFN Coulomb branch \cite{BFN2} of a gauge theory for a pair $(G, N)$ of a reductive group $G$ and its complex representation $N$, and correspondingly $X^! = T^*N /\!\!/\!\!/\!\!/G$ is a Higgs branch of the same theory. For the general formulation, we refer the reader to the original paper \cite{KMP21}; the notations are drawn from \cite{bai-lee}.


There are tautological line bundles on $X^!$ indexed by $X^\bullet(G) \cong H^2_G(\mathrm{pt})$ arising from the map
\begin{equation}
    X^\bullet(G) \cong \mathrm{Pic}^G (\mu^{-1}(0)) \to \mathrm{Pic}(X^!),
\end{equation} by a slight abuse of notation we denote these by $L \in X^\bullet (G) \cong \mathrm{Hom}(G, \mathbb{G}_{\mathrm{m}})$. These also generate the $K$-theory of $X^!$ by Kirwan surjectivity \cite[Section 2.3]{smirnov-zhou} in the case $G$ is abelian so that $X^!$ is the associated hypertoric variety. Denote by $d \in \mathrm{Eff}(X) \subseteq H_2(X^!;\mathbb{Z})$ the classes represented by rational curves in $X^!$, so we have a well-defined pairing $\langle L, d \rangle \in \mathbb{Z}$.

\begin{defn}
    The ring of K\"ahler differential operators is
    \begin{equation}
        R := \mathbb{Z}[\hbar][\![z^d : d \in \Sigma]\!][ \partial_L : L \in X^\bullet (G)]
    \end{equation}
    where the ring structure is subject to the relation $\partial_L z^d = z^d (\partial_L  + \hbar \langle L, d \rangle)$; the indices $d$ range over the submonoid $\Sigma \subseteq H_2(X^!;\mathbb{Z})$ generated by the K\"ahler roots of $X^!$, see \cite[Definition 2.21]{bai-lee}. 
\end{defn}

To describe the $D$-module of graded traces of $X$, we use the presentation of the BFN Coulomb branch of gauge theories as the spectrum of the $G[\![t]\!]$-equivariant Borel-Moore homology of a certain ind-scheme $\mathcal{R}_{G, N}$ which is homotopy equivalent to the affine Grassmannian $\mathrm{Gr}_G$, hence with connected components indexed by $\pi_0(\mathcal{R}) \cong \pi_0(\mathrm{Gr}_G) \cong \pi_1(G)$ \cite[3(v)]{BFN2}. The ring structure is given by convolution, see \cite[Section 3(iv), Remark 3.14]{BFN2}. Hence, the ring of functions on the Coulomb branch admits a decomposition into $\pi_1(G)$-graded pieces. The same decomposition applies to the quantization, which is obtained by considering the $\mathbb{C}^\times_q$-loop-equivariant Borel-Moore homology \cite[Definition 3.13]{BFN2}:
\begin{equation}
    \mathcal{A}=  H_{*, BM}^{G[\![t]\!] \rtimes \mathbb{C}^\times_q} (\mathcal{R}_{K, N}) \cong \bigoplus_{d \in \pi_1(G)} \mathcal{A}_d.
\end{equation}

Assuming that $X^\bullet(G) \to H^2(X^!)$ is surjective, which holds for, e.g., $X^!$ being a Nakajima quiver variety (cf. \cite{hyperkahler-kirwan}), we see that the monoid spanned by K\"ahler roots $\Sigma \subseteq H_2(X^!) \cong H^2(X^!)^\vee$ injects into $X^\bullet(G)^\vee$. Since $X^\bullet(G)^\vee \cong X_\bullet(Z(G)^\circ)$ are the coweights of the central torus $Z(G)^\circ \subseteq G$, injecting into the coweights of the maximal torus $X^\bullet(G)^\vee \hookrightarrow X_\bullet(T) \twoheadrightarrow \pi_1(G)$, one can also consider the indices $d \in \Sigma$ for the variables $z^d$ as indices in (a submonoid of) $\pi_1(G)$.


Note that the quantized Coulomb branch algebra is linear over the ring $H^*_{G[\![t]\!] \rtimes \mathbb{C}^\times_q} (\mathrm{pt})$. Denote by $L \in H^2_G(\mathrm{pt})\cong X^\bullet (G)$ the weight of $G$ giving rise to the tautological line bundle and, by a slight abuse of notation, its image in the quantized Coulomb branch algebra by the inclusion.

\begin{defn}
    The $D$-module of graded traces of $X$ is the left $R$-module
    \begin{equation}
        M_{ eq}(X) := R \otimes_{\mathbb{Z}[\hbar][\partial_L]} \mathcal{A}_0 / ( 1 \otimes ab - z^d \otimes ba : a \in \mathcal{A}_{d}, b \in \mathcal{A}_{-d}).
    \end{equation}
    where $R$ acts by the algebra structure on the first factor. The tensor product is taken over the commutative subalgebra $\mathbb{Z}[\hbar][\partial_L] \subseteq R$; $R$ is naturally a right module over the subalgebra, and $\partial_L$ acts by left multiplication by $L$ on $\mathcal{A}_0$.
\end{defn}

Recall that the Higgs branch $X^!$ is acted on by $\mathbf{T} = \mathbb{G}_{\mathrm{m}}^{\hbar} \times T$. To simplify our discussion, we turn off the action along the $T$-direction when discussing equivariant cohomology, $K$-theory and quasimap counts; see Remark \ref{rem:flavor-symmetry} for how to bring back $T$-equivariance. For the $D$-module of $X^!$ we use a particular specialization of the quantum $D$-module from \cref{eqn:qconn}, where the equivariant parameters $t$ and $\hbar$ of the loop and conical equivariant parameters are identified. This specialization is called the  ``CY specialization'' in \cite[Remark 4.3]{KMP21}.

\begin{defn}
    The ($\hbar=t$ specialized) quantum $D$-module of $X^!$ is the left $R$-module
    \begin{equation}
        M_{Kah}(X^!) := H^*_{\mathbb{G}_{\mathrm{m}}^{t} \times \mathbb{G}_{\mathrm{m}}^{\hbar}}(X^!)[\![z^d]\!]|_{\hbar = t}
    \end{equation}
    where $\mathbb{Z}[\![z^d]\!] \subseteq R$ acts via multiplication and $\mathbb{Z}[\partial_L] \subseteq R$ acts by the quantum connection $\nabla_L|_{\hbar = t}$ from \cref{eqn:qconn} specialized at $t = \hbar$ after passing to the nonequivariant limit for the $T$-action.
\end{defn}

\begin{conj}[Quantum Hikita conjecture \cite{KMP21} for Higgs--Coulomb duality]
    For symplectically dual varieties $X$ and $X^!$ given by Coulomb branch and Higgs branch constructions for the gauge theory $(G, N)$, there is an isomorphism of $R$-modules
    \begin{equation}
        M_{\mathrm{eq}}(X) \cong M_{Kah}(X^!).
    \end{equation}
\end{conj}

In \cite{KMP21, bai-lee} a more general version applying to general symplectically dual pairs of conical symplectic resolutions are explained; we have given the simplest version so that the analogy with the $q$-difference equations below becomes more transparent.

\subsubsection{$K$-theoretic quantum Hikita conjecture}

The $K$-theoretic version of the quantum Hikita conjecture we explain below replaces the $D$-modules from \cite{KMP21} by $q$-difference modules. 

\begin{defn}
    The ring of K\"ahler $q$-difference operators is
    \begin{equation}
        R_q := \mathbb{Z}[q^\pm][\![z^d: d \in \Sigma]\!][q^L : L \in X^\bullet (G) ]
    \end{equation}
    where the ring structure is subject to the relation $q^L z^d  = q^{\langle L, d \rangle } z^d q^L$.
\end{defn}
Here, the symbol $q^L$ should be interpreted as the operator $q^L \cdot z^d = q^{\langle L, d \rangle}z^d$, hence the defining relation of $R_q$. Again, the indices $d$ range over the submonoid of $H_2(X^!;\mathbb{Z})$ or $\pi_1(G)$, generated by the K\"ahler roots of $X^!$ from  \cite[Definition 2.21]{bai-lee}.

\begin{rem}
    The usual definition of a $q$-difference module involves the $q$-shift operations $q^L : \mathbb{Z}[\![z^d ]\!] \to \mathbb{Z}[q][\![z^d]\!]$ and a $\mathbb{Z}[q^\pm][\![z^d]\!]$-module $V$ equipped $q^L$-twisted linear endomorphisms $A_L : V \to V$ satisfying $A_L(f(z)v) = f(q^L z) A_L(v)$. This definition is equivalent to having an $R_q$-module structure on $V$ where $q^L \in R_q$ acts by the operator $A_L$.
\end{rem}

The $q$-difference module of graded traces of $X$ is again described using the presentation of the BFN Coulomb branch as the spectrum of the $G[\![t]\!]$-equivariant $K$-homology of $\mathcal{R}_{G, N}$. The same argument as in the cohomological case yields the grading by $\pi_1(G)$ for the multiplicative  Coulomb branch algebra
\begin{equation}
\mathcal{O}:=   \mathcal{O}(X) \cong K^{G[\![t]\!]} (\mathcal{R}_{G, N}) \cong \bigoplus_{d \in \pi_1(G)} \mathcal{O}_d.
\end{equation}
and its quantization,
\begin{equation}
    \mathcal{A}=  K^{G[\![t]\!] \rtimes \mathbb{C}^\times_q} (\mathcal{R}_{G, N}) \cong \bigoplus_{d \in \pi_1(G)} \mathcal{A}_d.
\end{equation}

Note that the quantized Coulomb branch algebra is linear over the ring $K^{G[\![t]\!] \rtimes \mathbb{C}^\times_q} (\mathrm{pt})$. Denote by $L \in K^G(\mathrm{pt})$ the $K$-theory class of a tautological line bundle coming from a character of $G$ and its image in the quantized Coulomb branch algebra by the inclusion. Define a $\mathbb{Z}[q^L]$-module structure on $\mathcal{A}_0$ such that $q^L \cdot a = L a$.

\begin{defn}
    The $q$-difference module of graded traces of $X$ is the left $R_q$-module
    \begin{equation}\label{eqn:graded-trace-module}
        M_{q, \mathrm{eq}}(X) := R_q \otimes_{\mathbb{Z}[q^\pm][q^L]} \mathcal{A}_0 / ( 1 \otimes ab - z^d \otimes ba : a \in \mathcal{A}_{d}, b \in \mathcal{A}_{-d}).
    \end{equation}
    where $R_q$ acts by the algebra structure on the first factor. The tensor product is taken over the commutative subalgebra $\mathbb{Z}[q^\pm][q^L] \subseteq R_q$, over which $R_q$ is naturally also a right module.
\end{defn}

The $q$-difference module of the Higgs branch $X^!$ replaces quantum cohomology with quantum $K$-theory, and also involves the identification of the equivariant parameters $q$ and $\hbar$ of the loop and conical equivariant parameters.

\begin{defn}
    The ($\hbar=q$ specialized) quantum $q$-difference module of $X^!$ is the left $R_q$-module
    \begin{equation}
        M_{q, \mathrm{Kah}}(X^!) := K_{\mathbb{G}_{\mathrm{m}}^q \times \mathbb{G}_{\mathrm{m}}^{\hbar}}(X^!)[\![z^d]\!]|_{\hbar = q}
    \end{equation}
    where $\mathbb{Z}[\![z^d]\!] \subseteq R_q$ acts via multiplication and $\mathbb{Z}[q^L] \subseteq R_q$ acts by the K\"ahler $q$-difference operators from \cref{defn:kahler-shift} specialized at $q = \hbar$ after passing to the nonequivariant limit for the $T$-action.
\end{defn}

\begin{conj}[$K$-theoretic Quantum Hikita conjecture for Higgs--Coulomb duality]\label{conj:K-quantum-hikita}
    For symplectically dual varieties $X$ and $X^!$ given by Coulomb branch and (resolved) Higgs branch constructions for the gauge theory $(G,N)$, there is an isomorphism of $q$-difference modules
    \begin{equation}
        M_{q, \mathrm{eq}}(X) \cong M_{q, \mathrm{Kah}}(X^!).
    \end{equation}
\end{conj}

\begin{rem}\label{rem:flavor-symmetry}
    One may also allow \emph{flavor symmetries} in the formulation. For the deformation on the graded traces side, one first assumes that there is an extension $\widetilde{G}$ such that $T := \widetilde{G}/G$ is a torus, such that the $G$-action on $N$ extends to a $\widetilde{G}$-action. Then one replaces $\mathcal{A} = K^{G[\![t]\!] \rtimes \mathbb{C}^\times_q} (\mathcal{R}_{G, N})$ in the construction with the full $\widetilde{G}[\![t]\!] \rtimes \mathbb{C}^\times_q$-equivariant $K$-homology of $\mathcal{R}_{G, N}$. For the deformation on the quantum side, one replaces $\mathbb{G}_{\mathrm{m}} \times \mathbb{G}_{\mathrm{m}}^\hbar$-equivariant $K$-theory with $\mathbb{G}_{\mathrm{m}}^q \times T \times \mathbb{G}_{\mathrm{m}}^\hbar = \mathbb{G}_{\mathrm{m}}^q \times \mathbf{T}$-equivariant $K$-theory. In the cohomological setting, this is the formulation in \cite{KMP21} and \cite{bai-lee}, where additional information from the $T$-action is encoded as a $X^{\bullet}(T)$-grading on the graded traces side.
\end{rem}

\subsection{Quantum Hikita conjecture at roots of unity}\label{ssec:quantumHikita-formulation}
In this subsection, we formulate an ``arithmetic'' version of the $K$-theoretic quantum Hikita conjecture by investigating the $q$-difference modules at root of unity. In the arithmetic setting, both $q$-difference modules are equipped with an extra module structure arising from ($q$-)Frobenius-constant quantizations \cite[Section 4]{lonergan} and quantum Adams operations respectively, whose actions are also intertwined under the isomorphism of $q$-difference modules. This discussion could be understood as the parallel of the cohomological case of our early work \cite{bai-lee}.

\subsubsection{Action of quantum Adams operations}
Following previous discussions, we turn off the $T$-action on $X^!$ to simplify the notations. As in \cref{ssec:qAdams}, one can further specialize the ($\hbar = q$ specialized) quantum $q$-difference module $M_{q, \mathrm{Kah}}(X^!)$ to a root of unity $q = \zeta$ by quotienting the underlying module by the $k$th cyclotomic polynomial $\Phi_k(q)$;
\begin{equation}
    M_{q, \mathrm{Kah}}(X^!)|_{q=\zeta} := \mathbb{Z}[q^\pm]/(\Phi_k(q)) \otimes M_{q, \mathrm{Kah}}(X^!).
\end{equation}

This arithmetic reduction of the K\"ahler quantum $q$-difference module admits an additional structure of a module over the PSZ quantum $K$-theory ring.

\begin{prop}
    The quantum Adams operators $Q\psi^k_{\mathcal{F}}|_{\hbar=q}$, specialized at $\hbar = q$, define an algebra action of the PSZ quantum $K$-theory ring $QK(X^!) |_{\hbar = 1} := (K_{\mathbb{G}_{\mathrm{m}}^{\hbar}}(X^!)[\![{z^{\mathrm{eff}}}]\!], \star)|_{\hbar = 1}$ on the quantum $q$-difference module $M_{q, \mathrm{Kah}}(X^!)|_{q=\zeta}$  at $q = \zeta$. 
\end{prop}
\begin{proof}
    \cref{thm:kahler-covariant-constancy} shows the claim that the quantum Adams operators act as $q$-difference module endomorphisms. Property (3) from \cref{prop:adams-property} and \cref{lem:qadams-additivity} show that the action respects PSZ quantum $K$-theory multiplication and linearity, respectively; hence it defines an algebra action of the PSZ quantum $K$-theory ring.
\end{proof}

\subsubsection{Action of Lonergan's Frobenius center}
The $q$-difference module $M_{q, \mathrm{eq}}(X)$ of graded traces specialized to $q=\zeta$ also admits an action by a certain algebra, where the action is defined by the following construction from \cite{lonergan}. Recall that $X$ is the spectrum of the convolution algebra given by
\begin{equation}
    \mathcal{O} := \mathcal{O}(X) =  K^{G[\![t]\!]}(\mathcal{R}_{G, N}) \cong \bigoplus_{d \in \pi_1(G)}\mathcal{O}_d
\end{equation}
whose quantization is given by
\begin{equation}
    \mathcal{A} := K^{G[\![t]\!] \rtimes \mathbb{C}^\times_q}(\mathcal{R}_{G, N}) \cong \bigoplus_{d \in \pi_1(G)}\mathcal{A}_d.
\end{equation}
The quantization $\mathcal{A}$ is linear over $K^{\mathbb{C}^\times_q}(\mathrm{pt}) \cong \mathbb{Z}[q^\pm]$, and we denote by $\mathcal{A}_\zeta := \mathcal{A} / (\Phi_k(q))$ for the quotient by the $k$-th cyclotomic polynomial. The quantized algebra $\mathcal{A}_\zeta$ at $q = \zeta$ is in general non-commutative, and we denote its center by $Z(A_\zeta)$. \cite{lonergan} constructs the following $K$-theoretic analog of a ``Frobenius-constant quantization'' (cf. \cite{bezrukavnikov-kaledin-quantp}):

\begin{thm}[{\cite[Theorem 4.1]{lonergan}}]
    There is an injective map of algebras
    \begin{equation}
        \Lambda : \mathcal{O} \to Z (\mathcal{A}_\zeta)
    \end{equation}
    that is linear, transports multiplication $r \in K^{G[\![t]\!]}(\mathrm{pt})$ to multiplication by the Adams operation $\psi^k(r) \in K^{G[\![t]\!] \rtimes \mathbb{C}^\times_q}(\mathrm{pt})/(\Phi_k(q))$, and takes $\mathcal{O}_d$ to $\mathcal{A}_{kd}$.
\end{thm}
The map $\Lambda$ is the $K$-theoretic analog of the Frobenius-constant quantization in the sense that it centrally embeds the commutative algebra $\mathcal{O}(X)$ into its quantization. 

\begin{rem}
    A streamlined exposition of (the cohomological version of) Lonergan's construction is provided in \cite[Section 3.2.2]{bai-lee}. The key geometric insight for the construction is as follows. The convolution algebra structure on BFN Coulomb branches has an interpretation in terms of Beilinson--Drinfeld Grassmannian, in which the multiplication of two elements is geometrically modelled by a collision of two points on a curve. Lonergan's construction generalizes the Beilinson--Drinfeld picture so that one considers $k$ points arranged $\mu_k$-symmetrically on a curve colliding simultaneously in a way that preserves the $\mu_k$-symmetry. The resemblance of this configuration to the moduli spaces of quantum power operations was the starting point of the conjectures in \cite{bai-lee}, continued in this paper.
\end{rem}

The arithmetic reduction of the $q$-difference module of graded traces admits an additional structure of a module over the commutative ring $\mathcal{O}_0 \subseteq \mathcal{O}$ given by the Frobenius center $\Lambda$. The action can be further extended to a ring which should be considered the ``mirror'' of the PSZ quantum $K$-theory ring:

\begin{prop}[{cf. \cite[Proposition 3.14, Proposition 3.15]{bai-lee}}]
    Left multiplication by the central element $\Lambda(x)$, $x \in \mathcal{O}_0$ defines an algebra action of the ring
    \begin{equation}\label{eqn:quantum-b-algebra}
        B_z(X) := \mathbb{Z}[\![z^d]\!] \otimes \mathcal{O}_0 / ( 1 \otimes ab - z^d \otimes ba : a \in \mathcal{O}_{-d}, b \in \mathcal{O}_d)
    \end{equation}
    on the $q$-difference module of graded traces $M_{q, \mathrm{eq}}(X)|_{q=\zeta}$ at $q = \zeta$. Here, $z^d \otimes 1 \in B_z(x)$ acts by multiplication by $z^{kd}$. 
\end{prop}
\begin{proof}
    To show that the multiplication is a well-defined action on
    \begin{equation}
        M_{q, \mathrm{eq}}(X) := R_q \otimes_{\mathbb{Z}[q^\pm][q^L]} \mathcal{A}_0 / ( 1 \otimes ab - z^d \otimes ba : a \in \mathcal{A}_{d}, b \in \mathcal{A}_{-d}),
    \end{equation}
    we must check that the action factors through the relations
    \begin{equation}
        J := ( 1 \otimes ab - z^d \otimes ba : a \in \mathcal{A}_{d}, b \in \mathcal{A}_{-d}).
    \end{equation}
    This claim follows from centrality, as for a fixed $x \in \mathcal{O}_0$ we have
    \begin{equation}
        \Lambda(x) (ab - z^d ba) = a \left( b \Lambda(x) \right) - z^d (b \Lambda(x)) a \in J.
    \end{equation}
    To check that the endomorphism respects $R_q$-linear structure, so that it is a $q$-difference module endomorphism, we again use the centrality of $\Lambda$. Indeed, for $L \in X^\bullet(G)$, we have
    \begin{align}
        q^L  \Lambda(x) (z^d \otimes a) & = q^L (z^d \otimes \Lambda(x) a) \\
        &= q^{\langle L, d\rangle} z^d \otimes L\Lambda(x)a \\
        &= q^{\langle L, d \rangle} z^d \otimes \Lambda(x) La \\
        &= \Lambda(x) \left( q^{\langle L, d \rangle} z^d \otimes La \right) \\
        &= \Lambda(x) \left( q^L (z^d \otimes a) \right).
    \end{align}
    Since $\Lambda$ is an algebra map, the action naturally defines an algebra action of $\mathcal{O}_0$ on $M_{q, \mathrm{eq}}(X)|_{q=\zeta}$. We extend the action to $\mathbb{Z}[\![z^d ]\!] \otimes \mathcal{O}_0$ by letting $z^d \otimes 1 $ act by left multiplication by $z^{kd}$. This action then factors through $B_z(X)$, since elements in the relation
    \begin{equation}
        1 \otimes ab - z^d \otimes ba, \quad a \in \mathcal{O}_{-d},\  b \in \mathcal{O}_d
    \end{equation}
    of the algebra $B_z(X)$ act by
    \begin{equation}
        1 \otimes \Lambda(a) \Lambda(b) - z^{kd} \otimes \Lambda(b) \Lambda(a) \in J,
    \end{equation}
    which is zero in the $q$-difference module $M_{q,\mathrm{eq}}(X)$.\qedhere
    
\end{proof}

\subsubsection{Arithmetic $K$-theoretic quantum Hikita conjecture}

Having described the actions of quantum Adams operations and Lonergan's $K$-theoretic Frobenius center on the $q$-difference modules from quantum Hikita conjecture, we may now formulate the following ``arithmetic refinement'' of the $K$-theoretic quantum Hikita conjecture.

\begin{conj}[$K$-theoretic quantum Hikita conjecture at roots of unity]\label{conj:arithmetic-k-hikita}
    For symplectically dual varieties $X$ and $X^!$ given by Coulomb branch and (resolved) Higgs branch constructions for the gauge theory $(G,N)$, there is an identification of the following data.

    \begin{enumerate}
        \item An isomorphism as left $R_q$-modules,
        \begin{equation}
            M_{q, \mathrm{eq}}(X) \cong M_{q, \mathrm{Kah}}(X^!),
        \end{equation}
        in particular an isomorphism of their $q = \zeta$ specializations;
        \item A ring isomorphism $B_z(X) \cong K_{\mathbb{G}_{\mathrm{m}}}(X^!)[\![z^{\mathrm{eff}}]\!]|_{\hbar = 1}$;
        \item Under the isomorphisms above, the actions of Frobenius-constant quantizations and the quantum Adams operators of $B_z(X) \cong K_{\mathbb{G}_m}(X^!)[\![z^{\mathrm{eff}}]\!]|_{\hbar =1}$ on $M_{q, \mathrm{eq}}(X) \cong M_{q, \mathrm{Kah}}(X^!)$ are intertwined. 
    \end{enumerate}
\end{conj}

\cref{conj:arithmetic-k-hikita} is the $K$-theoretic analogue of the arithmetic 3D mirror symmetry conjecture we proposed in \cite[Conjecture 3.21]{bai-lee}. The first item (1) is a restatement of Conjecture \ref{conj:K-quantum-hikita}, and items (2) and (3) describes the additional ingredients in the arithmetic version. 

\begin{rem}
    As in Remark \ref{rem:flavor-symmetry}, we can introduce the flavor torus $T$ into the formulation of Conjecture \ref{conj:arithmetic-k-hikita} as well. On a different note, the conjectural isomorphism $B_z(X) \cong K_{\mathbb{G}_{\mathrm{m}}}(X^!)[\![z^{\mathrm{eff}}]\!]|_{\hbar =1}$ is compatible with Conjecture \ref{conj:K-quantum-hikita} by specializing $q=1$ in Conjecture \ref{conj:K-quantum-hikita}. See \cite[Theorem 16]{PSZ-quantum} on how to recover the PSZ quantum $K$-ring from the quantum $q$-difference module of $X^!$; on the Coulomb side, the defining relation in \eqref{eqn:quantum-b-algebra} is compatible with the one in \eqref{eqn:graded-trace-module} and the commutator relation of $q$-difference module.
\end{rem}

\subsection{Example: Abelian gauge theories}\label{ssec:quantumHikita-proof}
We will conclude this section by establishing the validity of the $K$-theoretic quantum Hikita conjecture for the Coulomb/Higgs mirrors, through a nontrivial example of abelian gauge theories (hypertoric varieties). In these examples, the computation of the action of Frobenius-constant quantizations and the quantum Adams operators can be reduced to the computation for line bundles, effectively reducing \cref{conj:arithmetic-k-hikita} to \cref{conj:K-quantum-hikita} and \cref{thm:qadams=pcurvature}. 

We begin with the rank $1$ case to illustrate the structure of the argument, then generalize the argument to the general case where $G$ is abelian, i.e., when the varieties $X, X^!$ are hypertoric.

\subsubsection{Rank 1 abelian gauge theory}

We first consider the case of $(G, N) = (\mathbb{G}_{\mathrm{m}}, \mathbb{C})$. The general case for $G$ abelian is similar, with the additional complexity being purely combinatorial. For $(G,N) =  (\mathbb{G}_{\mathrm{m}}, \mathbb{C})$, the Coulomb branch is given by $X \cong \mathrm{Spec} K^{\mathbb{G}[\![t]\!]} (\mathcal{R}_{\mathbb{G}_{\mathrm{m}}, \mathbb{C}})$, and the Higgs branch is given by $X^! \cong T^*\mathbb{P}^0$.

First consider the K\"ahler $q$-difference module. The corresponding variety is $T^*\mathbb{C}/\!\!/\!\!/\!\!/ \mathbb{G}_{\mathrm{m}} \cong T^*\mathbb{P}^0$ as the hypertoric variety associated to the exact sequence of tori $1 \to \mathbb{G}_{\mathrm{m}} \to \mathbb{G}_{\mathrm{m}} \to 1 \to 1$. 

The computation of the K\"ahler $q$-difference module in the hypertoric case is completed by \cite{smirnov-zhou}. Recall that attached to a hypertoric variety defined from a short exact sequence of tori
\begin{equation}
    1 \to \mathbb{G}_{\mathrm{m}}^k \overset{\iota}{\to} \mathbb{G}_{\mathrm{m}}^n \overset{\beta}{\to} \mathbb{G}_{\mathrm{m}}^d \to 1
\end{equation}
and a stability condition $\theta \in \mathrm{Lie}(\mathbb{G}_{\mathrm{m}}^k)^\vee$ is a hyperplane arrangement $\{ H_1, \dots, H_n\}$ for $H_i := \{ x \in (\mathbb{R}^d)^\vee : \langle x, \beta(e_i) \rangle = - \langle \widetilde{\theta} , e_i \rangle \}$ where $\widetilde{\theta} \in \mathrm{Lie}(\mathbb{G}_{\mathrm{m}}^n)^\vee$ is a lift of $\theta$ along $\iota^\vee$. The relations of the $q$-difference module are indexed by \emph{circuits}, which are minimal subsets $S \subseteq \{1, \dots, n\}$ such that $\bigcap_{i \in S} H_i = \emptyset$; that is, the normal vectors $\beta(e_i)$ must be linearly dependent, hence circuits are minimal subsets such that $e_i$ have linear dependence relations in $\ker \beta$. Each circuit admits a decomposition $S = S^+ \sqcup S^-$ depending on the sign of corresponding $e_i$'s in the linear dependence relation. (Primitive) effective curves in smooth hypertoric varieties are in bijection with circuits, see \cite[Section 2.5]{smirnov-zhou}.

\begin{lemma}[{\cite[Theorem 5.14(1)]{smirnov-zhou}}]\label{lem:kahler-qdiff-computation}
    The K\"ahler $q$-difference module of $X^!$ is equivalent to the ring $K_{\mathbb{G}_{\mathrm{m}}^n \times \mathbb{G}_{\mathrm{m}}^\hbar \times \mathbb{G}_{\mathrm{m}}^q}(X)[\![z^d]\!]|_{\hbar = q}$ modulo the relations
    \begin{equation}
        \prod_{i \in S^+} (1-q^{L_i}) \prod_{i \in S^-} (1-\hbar q^{L_i}) - z^{\beta_S} \prod_{i \in S^+} (1-\hbar q^{L_i} )  \prod_{i \in S^-} (1-q^{L_i})
    \end{equation} 
    for circuits $S = S^+ \sqcup S^-$ and associated curve classes $\beta_S$.
\end{lemma}
\begin{rem}
    In Smirnov--Zhou, the relations defining the quantum $q$-difference modules are described by ``shifted K\"ahler variables" $z_\#^\beta := z^\beta (-\hbar^{\frac{1}{2}\beta \cdot \det T^{1/2}X})$. This conventional difference can be remedied by an equivalent choice of polarization $T_X^{1/2} = \sum L_i - \mathcal{O}^{\oplus k} + (\mathcal{O}^{\oplus k} -  \hbar^{-1} \mathcal{O}^{\oplus k}) = \sum L_i - \hbar^{-1} \mathcal{O}^{\oplus k}$; replacing polarizations up to classes of the form $\mathcal{G} - \hbar^{-1} \mathcal{G}^\vee$ only acts by rescaling the K\"ahler variables (cf. \cite[7.5.8]{Oko17}) and does not affect the isomorphism type as a $q$-difference module.
\end{rem}

\begin{cor}\label{cor:kahler-qdiff-computation-rank1}
    The K\"ahler $q$-difference module of the Higgs branch $X^!$ for $(G,N) = (\mathbb{G}_{\mathrm{m}}, \mathbb{C})$ is equivalent to $\mathbb{Z}[q^\pm][a^\pm][\![z]\!]$ modulo the unique relation
    \begin{equation}
        (1-q^L) - z (1- \hbar q^L).
    \end{equation}
\end{cor}
\begin{proof}
    The hyperplane arrangement associated to $X^!$ is trivially empty, and hence there is a unique circuit $S = S^+ = \{1\}$ giving rise to the relation above. 
\end{proof}

We point out that the vertex functions (solutions of $q$-difference modules) used to compute the relations manifestly do not have poles at $1 -q^k$ after the $\hbar=q$ specialization (see \cite[Proposition 5.3]{smirnov-zhou}), so the $q$-difference module specialized at $\hbar = q$ do admit $q = \zeta$ specializations.

The computation of the $q$-difference module of graded traces is also simple in this case, with the Coulomb branch algebra given by the BFN convolution algebra $K^{\mathbb{G}[\![t]\!]} (\mathcal{R}_{\mathbb{G}_{\mathrm{m}}, \mathbb{C}})$.

\begin{lemma}\label{lem:coulomb-algebra-computation-rank1}
    For the pair $(G, N) = (\mathbb{G}_{\mathrm{m}}, \mathbb{C})$, the $K$-theoretic quantized Coulomb branch algebra $\mathcal{A}$ is the subalgebra of
    \begin{equation}
        K^{G[\![t]\!] \rtimes \mathbb{C}^\times_q}(\mathcal{R}_{\mathbb{G}_{\mathrm{m}}, 0}) \cong \mathbb{Z}[q^\pm ][y^\pm]\langle x^\pm \rangle/(yx=qxy)
    \end{equation}
    spanned additively over $\mathbb{Z}[q^\pm][y^\pm]$ by
    \begin{equation}
        x_d := \begin{cases}
            x^{d} & d \le 0 \\ \prod_{i=0}^{d-1} (1-y/q^i) x^d & d >0 
        \end{cases}
    \end{equation}
    with $x_d \in \mathcal{A}_d$.
\end{lemma}
\begin{proof}
    This computation can be found in \cite[Example 4.4]{lonergan}; it is also the $K$-theoretic analogue of the computation in \cite[4(ii)]{BFN2}.
\end{proof}

\begin{cor}\label{cor:trace-qdiff-computation-rank1}
    The $q$-difference module of graded traces for the Coulomb branch $X$  for $(G,N) = (\mathbb{G}_{\mathrm{m}}, \mathbb{C})$ is equivalent to $\mathbb{Z}[q^\pm][y^\pm][\![z]\!]$ modulo the unique relation
     \begin{equation}
        (1-y) - z (1- qy).
    \end{equation}
\end{cor}
\begin{proof}
    The inclusion $K^{G[\![t]\!] \rtimes \mathbb{C}^\times_q}(\mathrm{pt}) \to \mathcal{A}_0 \subseteq \mathcal{A}$ is the inclusion of the underlying commutative ring $\mathbb{Z}[q^\pm, y^\pm]$, and in particular $q^L$ acts by multiplication of $y$. The relation $1 \otimes ab - z^d \otimes ba$ for $a \in \mathcal{A}_d$ is now generated by
    \begin{equation}
     1 \otimes x_1 x_{-1} - z \otimes x_{-1} x_1 =     1\otimes (1-y)xx^{-1} - z \otimes x^{-1}(1-y) x.
    \end{equation}
    Using the relation $yx = qxy$ to simplify the second term, this is indeed equivalent to
    \begin{equation}
        (1-y) - z(1- qy),
    \end{equation}
    as desired. 
\end{proof}

Therefore we now have the following:

\begin{prop}\label{prop:arithmetic-k-hikta-hypertoric-rank1}
    \cref{conj:arithmetic-k-hikita} holds for Coulomb/Higgs pair of symplectically dual varieties for $(G,N) = (\mathbb{G}_{\mathrm{m}}, \mathbb{C})$.
\end{prop}
\begin{proof}
    \cref{cor:kahler-qdiff-computation-rank1} and \cref{cor:trace-qdiff-computation-rank1} show that the $q$-difference modules $M_{q, \mathrm{eq}} (X) \cong M_{q, \mathrm{Kah}}(X^!)$ are isomorphic before passing to $q  = \zeta$, hence remains isomorphic after the $q=\zeta$ specialization. It suffices to check that the quantum Adams operators and the action of Lonergan's center are intertwined.

    In the hypertoric case, the quantum $K$-theory ring and the quantum $B$-algebra $B_z(X)$ are both generated by classes of line bundles. By the algebra property of the actions of quantum Adams operators and Lonergan's center, it suffices to check that the actions of line bundles are intertwined.

    It suffices to show the stronger claim that the action of line bundles $L$ as quantum Adams operators and the actions of Lonergan's center can be characterized as the $p$-curvature $(q^L)^k$ of both $q$-difference modules. This claim is \cref{thm:qadams=pcurvature} for the quantum Adams operators, and \cite[Example 4.4]{lonergan} for the Lonergan operators.
\end{proof}

\subsubsection{General hypertoric varieties}
The general case for pairs $(G,N)$ such that $G$ is abelian, so that the Higgs branch is a hypertoric variety for the exact sequence $1 \to G \to \widetilde{G} \to F \to 1$ with $\widetilde{G} \cong (\mathbb{G}_{\mathrm{m}})^n$ acting by coordinate-wise multiplication on $N$, can be dealt with with the same strategy. The proof for the more general ``flavor deformed'' version (cf. \cref{rem:flavor-symmetry}) is easier, which we outline below.

\begin{lemma}\label{lem:coulomb-algebra-computation}
For the pair $((\mathbb{G}_{\mathrm{m}})^n, \mathbb{C}^n)$, where action is by characters $r_1, \dots, r_n$, the $K$-theoretic quantized Coulomb branch algebra $\mathcal{A}$ is the subalgebra of
    \begin{equation}
        K^{G[\![t]\!] \rtimes \mathbb{C}^\times_q}(\mathcal{R}_{G, 0}) \cong \mathbb{Z}[q^\pm ][y_i^\pm]\langle x_i^\pm \rangle/(y_ix_i=qx_iy_i)
    \end{equation}
    spanned additively over $\mathbb{Z}[q^\pm][y^\pm_i]$ by products of monomials $x_d := \prod x_{i, d_i}$ where
    \begin{equation}
        x_{i,d_i} := \begin{cases}
            x^{d_i}_i & d_i \le 0 \\ \prod_{j=0}^{rd_i-1} (1-y_i/q^j) x^{d_i}_i & d_i >0 
        \end{cases}
    \end{equation}
    with $x_d = \prod x_{i,d_i} \in \mathcal{A}_{d} = \mathcal{A}_{(d_1, \dots, d_n)}$.    
\end{lemma}
\begin{proof}
    When the rank is $n=1$, this is again the computation of \cite[Example 4.4]{lonergan}, for which \cref{lem:coulomb-algebra-computation-rank1} is the weight $1$ case. The case for higher $n$ follows from the fact that $\mathcal{A}(G_1 \times G_2, N_1 \oplus N_2) \cong \mathcal{A}(G_1,N_1) \otimes_{\mathbb{Z}[q^\pm]}\mathcal{A}(G_2, N_2)$ for quantized Coulomb branches, \cite[3(vii)(a)]{BFN2}.
\end{proof}

Let $\mathcal{A} = \bigoplus_{d \in \mathbb{Z}^n} \mathcal{A}_d$ be the resulting decomposition of the quantized Coulomb branch algebra for $(\mathbb{G}_{\mathrm{m}}^n, \mathbb{C}^n)$ acting by usual coordinate multiplication. The index $\mathbb{Z}^n$ can be identified with $\pi_1(\mathbb{G}_{\mathrm{m}}^n)$, or equivalently, $\mathrm{Lie}(\mathbb{G}_{\mathrm{m}}^n)_{\mathbb{Z}}$. 
\begin{lemma}
    The (flavor deformed) $K$-theoretic quantized Coulomb branch algebra for the pair $(G,N)$ is isomorphic as associative algebras to the subalgebra $\bigoplus_{d \in \mathrm{Lie}(G)_{\mathbb{Z}}} \mathcal{A}_d \subseteq \mathcal{A}$ where the index is over the sublattice $\mathrm{Lie}(G)_{\mathbb{Z}} \to \mathbb{Z}^n$.
\end{lemma}
\begin{proof}
    It follows from \cite[3(ix), 4(ii, vii)]{BFN2}. 
\end{proof}

The discussion above in the rank $1$ case shows that, assuming \cref{lem:kahler-qdiff-computation}, it suffices to show that the $q$-difference module of graded traces for $X$ is defined by the same relations indexed by the circuits of $X^!$. 

\begin{lemma}\label{lem:trace-qdiff-computation}
    The $q$-difference module of graded traces for the (flavor deformed) Coulomb branch $X$  for $(G,N) = (\mathbb{G}_{\mathrm{m}}^n, \mathbb{C}^n)$ is defined as $\mathbb{Z}[q^\pm][y^\pm_i][\![z^d]\!]$ modulo the relations
 \begin{equation}
        \prod_{i \in S^+} (1-y_i) \prod_{i \in S^-} (1-q y_i) - z^{\beta_S} \prod_{i \in S^+} (1-q y_i ) (1-y_i)
    \end{equation} 
    for circuits $S = S^+ \sqcup S^-$ and associated curve classes $\beta_S$ from $X^!$.
\end{lemma}

\begin{proof}
    First recall that all relations are of the form
    \begin{equation}
        1 \otimes ab - z^d \otimes ba, \ a \in \mathcal{A}_d, b \in \mathcal{A}_{-d}
    \end{equation}
    for $d \in \mathbb{N}\{ \beta_S: S \mbox{ circuit} \}$ in the monoid generated by the circuits by definition of the $q$-difference module of graded traces; therefore, the claim is that the ideal of relations is generated by the a priori smaller set $d \in \{ \beta_S : S \mbox{ circuit} \}$. 

    By \cref{lem:coulomb-algebra-computation}, such relations are of the form
    \begin{equation}
        f(y)x_d g(y)x_{-d} - z^d g(y)x_{-d} f(y) x_d
    \end{equation}
    for $f(y), g(y) \in \mathbb{Z}[q^\pm][y_i^\pm]$ and $x_d \in \mathcal{A}_d$, $x_{-d} \in \mathcal{A}_{-d}$ for $d \in \mathbb{N}\{ \beta_S: S \mbox{ circuit} \}$. Using the relations $f(y) x_d = x_d f(q^dy) \in \mathcal{A}$ and $f(y)z^d = z^df(q^dy) \in R_q$, one can simplify the relation to
    \begin{align*}
        f(y)x_d g(y)x_{-d} - z^d g(y)x_{-d} f(y) x_d &= f(y) g(q^{{-d}}y) x_d x_{-d} - z^d g(y) f(q^d y) x_{-d} x_d\\ &=     f(y) g(q^{{-d}}y) x_d x_{-d} - g(q^{-d }y) f( y) z^d  x_{-d} x_d     \\
        &=f(y) g(q^{-d}y) x_d x_{-d} (1-z^d).
    \end{align*}
    where in the last equality we use that $x_d x_{-d}|_{y \mapsto q^d y} = x_{-d} x_d$. Hence, the ideal of relations of the $q$-difference module is spanned as left $R_q$-modules by $x_dx_{-d} (1-z^d)$ for $d \in \mathbb{N}\{ \beta_S : S \mbox{ circuit} \}$. It suffices now to show tha it is spanned by the same elements $x_d x_{-d} (1-z^d)$ indexed by the smaller set $d \in \{ \beta_S : S \mbox{ circuit} \}$. The proof is by induction on $m$,  the number of summands in the expansion of $d \in\mathbb{N}\{ \beta_S : S \mbox{ circuit} \} $ into circuits $\delta_1, \dots, \delta_m$, in a way that
    \begin{equation}
        d = \delta_1 + \cdots + \delta_m
    \end{equation}
    is cancellation-free (see \cite[Lemma 6.2]{KMP21}). The argument from here is exactly the same as in \cite[Proposition 6.8]{KMP21}, to which we refer the reader.
\end{proof}

Combining \cref{lem:trace-qdiff-computation} with \cref{lem:kahler-qdiff-computation}, we arrive at:
\begin{thm}\label{thm:arithmetic-k-hikta-hypertoric}
    \cref{conj:arithmetic-k-hikita} holds for symplectically dual pairs of hypertoric varieties.
\end{thm}

\bibliographystyle{amsalpha}
\bibliography{ref}

\end{document}